\pdfoutput=1
\RequirePackage{pdf14}
\documentclass[11pt]{article}  
\usepackage{amsthm,amssymb} 
\usepackage[leqno]{amsmath}
\newif\ifmicrotype
\relax
\AtBeginDocument{
  \def\tempaa{\thesection.\arabic{equation}}
  \ifx\theequation\tempaa
     \def\theequation{\thesection\hbox{--}\arabic{equation}}
  \fi
  \ifmicrotype
     \RequirePackage{microtype}%
     \DeclareMicrotypeAlias{mnt}{ptm}%
  \fi
  \DeclareMathSizes{10}   {10}   {7}{6}
  \DeclareMathSizes{10.95}{10.95}{8}{6}
}

\DeclareFontFamily{OT1}{ptm}{}
\DeclareFontShape{OT1}{ptm}{m}{n} { <-> ptmr}{}
\DeclareFontShape{OT1}{ptm}{m}{it}{ <-> ptmri}{}
\DeclareFontShape{OT1}{ptm}{m}{sl}{ <->ptmro}{}
\DeclareFontShape{OT1}{ptm}{m}{sc}{ <-> ptmrc}{}
\DeclareFontShape{OT1}{ptm}{b}{n} { <-> ptmb}{}
\DeclareFontShape{OT1}{ptm}{b}{it}{ <-> ptmbi}{}     
\DeclareFontShape{OT1}{ptm}{bx}{n} {<->ssub * ptm/b/n}{}
\DeclareFontShape{OT1}{ptm}{bx}{it}{<->ssub * ptm/b/it}{}

\DeclareSymbolFont{bold}{OT1}{ptm}{b}{n}
\DeclareMathAlphabet{\mathbf}{OT1}{ptm}{b}{n}  
\DeclareMathAlphabet{\mathrm}{OT1}{ptm}{m}{n}
 
\let\Bbb=\mathbb  
\DeclareFontFamily{OT1}{psy}{}      
\DeclareFontShape{OT1}{psy}{m}{n}{ <-> s * [0.9] psyr}{}
\DeclareFontFamily{OMS}{ptm}{}     
\DeclareFontShape{OMS}{ptm}{m}{n}{ <8> <9> <10> gen * cmsy }{}
\DeclareFontFamily{OMS}{cmtt}{}     
\DeclareFontShape{OMS}{cmtt}{m}{n}{ <8> <9> <10> gen * cmsy }{}

\SetSymbolFont{operators}{normal}{OT1}{ptm}{m}{n}   
\SetSymbolFont{operators}{bold}{OT1}{ptm}{b}{n}     
\DeclareSymbolFont{emsy}{OT1}{ptm}{m}{it}
\DeclareSymbolFont{emsr}{OT1}{ptm}{m}{n}
\DeclareSymbolFont{emcmr}{OT1}{cmr}{m}{n}   
\DeclareSymbolFont{emsymb}{OT1}{psy}{m}{n}  
\DeclareMathSymbol a{\mathalpha}{emsy}{"61}
\DeclareMathSymbol b{\mathalpha}{emsy}{"62}
\DeclareMathSymbol c{\mathalpha}{emsy}{"63}
\DeclareMathSymbol d{\mathalpha}{emsy}{"64}
\DeclareMathSymbol e{\mathalpha}{emsy}{"65}
\DeclareMathSymbol f{\mathalpha}{emsy}{"66}
\DeclareMathSymbol g{\mathalpha}{emsy}{"67}
\DeclareMathSymbol h{\mathalpha}{emsy}{"68}
\DeclareMathSymbol i{\mathalpha}{emsy}{"69}
\DeclareMathSymbol j{\mathalpha}{emsy}{"6A}
\DeclareMathSymbol k{\mathalpha}{emsy}{"6B}
\DeclareMathSymbol l{\mathalpha}{emsy}{"6C}
\DeclareMathSymbol m{\mathalpha}{emsy}{"6D}
\DeclareMathSymbol n{\mathalpha}{emsy}{"6E}
\DeclareMathSymbol o{\mathalpha}{emsy}{"6F}
\DeclareMathSymbol p{\mathalpha}{emsy}{"70}
\DeclareMathSymbol q{\mathalpha}{emsy}{"71}
\DeclareMathSymbol r{\mathalpha}{emsy}{"72}
\DeclareMathSymbol s{\mathalpha}{emsy}{"73}
\DeclareMathSymbol t{\mathalpha}{emsy}{"74}
\DeclareMathSymbol u{\mathalpha}{emsy}{"75}
\DeclareMathSymbol v{\mathalpha}{emsy}{"76}
\DeclareMathSymbol w{\mathalpha}{emsy}{"77}
\DeclareMathSymbol x{\mathalpha}{emsy}{"78}
\DeclareMathSymbol y{\mathalpha}{emsy}{"79}
\DeclareMathSymbol z{\mathalpha}{emsy}{"7A}
\DeclareMathSymbol A{\mathalpha}{emsy}{"41}
\DeclareMathSymbol B{\mathalpha}{emsy}{"42}
\DeclareMathSymbol C{\mathalpha}{emsy}{"43}
\DeclareMathSymbol D{\mathalpha}{emsy}{"44}
\DeclareMathSymbol E{\mathalpha}{emsy}{"45}
\DeclareMathSymbol F{\mathalpha}{emsy}{"46}
\DeclareMathSymbol G{\mathalpha}{emsy}{"47}
\DeclareMathSymbol H{\mathalpha}{emsy}{"48}
\DeclareMathSymbol I{\mathalpha}{emsy}{"49}
\DeclareMathSymbol J{\mathalpha}{emsy}{"4A}
\DeclareMathSymbol K{\mathalpha}{emsy}{"4B}
\DeclareMathSymbol L{\mathalpha}{emsy}{"4C}
\DeclareMathSymbol M{\mathalpha}{emsy}{"4D}
\DeclareMathSymbol N{\mathalpha}{emsy}{"4E}
\DeclareMathSymbol O{\mathalpha}{emsy}{"4F}
\DeclareMathSymbol P{\mathalpha}{emsy}{"50}
\DeclareMathSymbol Q{\mathalpha}{emsy}{"51}
\DeclareMathSymbol R{\mathalpha}{emsy}{"52}
\DeclareMathSymbol S{\mathalpha}{emsy}{"53}
\DeclareMathSymbol T{\mathalpha}{emsy}{"54}
\DeclareMathSymbol U{\mathalpha}{emsy}{"55}
\DeclareMathSymbol V{\mathalpha}{emsy}{"56}
\DeclareMathSymbol W{\mathalpha}{emsy}{"57}
\DeclareMathSymbol X{\mathalpha}{emsy}{"58}
\DeclareMathSymbol Y{\mathalpha}{emsy}{"59}
\DeclareMathSymbol Z{\mathalpha}{emsy}{"5A}
\DeclareMathSymbol{\bullet}{\mathalpha}{emsymb}{"B7}
\DeclareMathSymbol{\regis}{\mathalpha}{emsymb}{"D2}
\def\Bullet{\leavevmode\unkern{$\m@th\bullet$}\kern.32em\ignorespaces}
\def\Regis{\leavevmode\raise.5ex\hbox{$\m@th\regis$}}
\DeclareMathSymbol +{\mathbin}{emcmr}{`+}
\DeclareMathSymbol ={\mathrel}{emcmr}{`=}  
\DeclareMathSymbol{\Gamma}{\mathalpha}{emcmr}{"00}
\DeclareMathSymbol{\Delta}{\mathalpha}{emcmr}{"01}
\DeclareMathSymbol{\Theta}{\mathalpha}{emcmr}{"02}
\DeclareMathSymbol{\Lambda}{\mathalpha}{emcmr}{"03}
\DeclareMathSymbol{\Xi}{\mathalpha}{emcmr}{"04}
\DeclareMathSymbol{\Pi}{\mathalpha}{emcmr}{"05}
\DeclareMathSymbol{\Sigma}{\mathalpha}{emcmr}{"06}
\DeclareMathSymbol{\Upsilon}{\mathalpha}{emcmr}{"07}
\DeclareMathSymbol{\Phi}{\mathalpha}{emcmr}{"08}
\DeclareMathSymbol{\Psi}{\mathalpha}{emcmr}{"09}
\DeclareMathSymbol{\Omega}{\mathalpha}{emcmr}{"0A}
\DeclareMathSizes{7.6}{8}{6}{5}

\def\`#1{{\accent"12 #1}}            
            
\chardef\J="11                  
                   
\chardef\AA="C8                      
\chardef\gbp="A3                    
\chardef\TIL="81                     
\chardef\endash="B1                
\chardef\emdash="D0                
\chardef\pourmille="BD              
\chardef\aoben="E3                   
\chardef\ooben="EB                  
\def\S{\leavevmode\unkern{\char"A7}\kern.1em\ignorespaces} 
            
\DeclareMathAccent{\dot}{\mathalpha}{operators}{"C7} 
\def\og{{\char"AB}\kern.15em}
\def\fg{\relax\ifhmode\unskip\kern.15em\fi{\char"BB}}

\def\cedpol#1{\setbox0=\hbox{#1}\ifdim\ht0=1ex \accent"CE #1%
  \else{\ooalign{\hidewidth\char"CE\hidewidth\crcr\unhbox0}}\fi}

\newskip\stdskip                      %
\newcommand{\stdspace}{\hskip 0.75em plus 0.15em \ignorespaces}

\newtheoremstyle{plain}{13.2pt plus6.6pt minus6.6pt}{6.6pt plus3.3pt minus3.3pt}%
{\sl}{}{\bf}{}{0.75em}{\thmname{#1}\thmnumber{ #2}\thmnote{\rm\stdspace(#3)}}

\newtheoremstyle{definition}{13.2pt plus6.6pt minus6.6pt}{6.6pt plus3.3pt minus3.3pt}%
{\rm}{}{\bf}{}{0.75em}{\thmname{#1}\thmnumber{ #2}\thmnote{\rm\stdspace(#3)}}

\newtheoremstyle{remark}{13.2pt plus6.6pt minus6.6pt}{6.6pt plus3.3pt minus3.3pt}%
{\rm}{}{\bf}{}{0.75em}{\thmname{#1}\thmnumber{ #2}\thmnote{\rm\stdspace(#3)}}

\theoremstyle{plain}

\usepackage{geometry}

    \newgeometry{vmargin={0mm, 37mm}, hmargin={30mm,23mm}}   
\usepackage{tikz-cd}
\usepackage{xparse}
\usepackage[most]{tcolorbox}
\usepackage{amsthm}

\RequirePackage[bookmarks, bookmarksopen=true, plainpages=false, pdfpagelabels, pdfpagelayout=SinglePage, breaklinks = true]{hyperref}
\usepackage[nameinlink, noabbrev,capitalize]{cleveref}
\usepackage{xcolor}
\hypersetup{
    colorlinks,
    linkcolor={red!50!black},
    citecolor={blue},
    urlcolor={red!80!black}
}

\usepackage{savesym}
\savesymbol{checkmark}
\usepackage{dingbat}

\makeatletter
\newcommand{\tpitchfork}{%
  \vbox{
    \baselineskip\z@skip
    \lineskip-.52ex
    \lineskiplimit\maxdimen
    \m@th
    \ialign{##\crcr\hidewidth\smash{$-$}\hidewidth\crcr$\pitchfork$\crcr}
  }%
}
\makeatother

\newtheorem*{namedthm}{\namedthmname}
\newcounter{namedthm}



\usepackage{amssymb} 
\renewcommand{\Bbb}{\mathbb}

\newtheorem{theorem}{Theorem}[subsection]
\newtheorem{lemma}[theorem]{Lemma}
\newtheorem{definition}[theorem]{Definition}
\newtheorem{remark}[theorem]{Remark}

\newtheorem{proposition}[theorem]{Proposition}

\newtheorem{subsubtheorem}{Theorem}[subsubsection]
\newtheorem{subsublemma}[subsubtheorem]{Lemma}
\newtheorem{subsubdefinition}[subsubtheorem]{Definition}
\newtheorem{subsubremark}[subsubtheorem]{Remark}
\newtheorem{subsubcorollary}[subsubtheorem]{Corollary}
\newtheorem{subsubproposition}[subsubtheorem]{Proposition}

\newtheorem*{mthe}{Theorem}

\newtheorem*{def*}{Definition}
\newtheorem*{rem*}{Remark}

\newtheorem{sectheorem}{Theorem}[section]

\newtheorem{secproposition}[sectheorem]{Proposition}
\usepackage[numbered]{bookmark}

\newtheorem{claim}{Claim}[theorem]
\makeatletter
\newcommand*{\sectionbookmark}[1][]{%
  \bookmark[%
    level=section,%
    dest=\@currentHref,%
    #1%
  ]%
}
\makeatother
\usepackage{etoolbox}
\makeatletter
\if@twoside
   \patchcmd{\ps@headings}%
      {\@evenhead{\thepage\hfil\slshape\leftmark}}%
      {\@evenhead{\thepage\hfil}}{}{}
\fi
   \patchcmd{\ps@headings}%
      {\@oddhead{{\slshape\rightmark}\hfil\thepage}}%
      {\@oddhead{\hfil\thepage}}{}{}  
\makeatother

\usepackage{mathrsfs}
\usepackage{adjustbox}
\usepackage[figurename=Fig.]{caption}
\usepackage{float}
\usepackage{graphicx}
\usepackage{amssymb}

\usepackage{tikz}

\DeclareFontFamily{U}{mathb}{\hyphenchar\font45}
\DeclareFontShape{U}{mathb}{m}{n}{%
  <-6> mathb5
  <6-7> mathb6
  <7-8> mathb7
  <8-9> mathb8
  <9-10> mathb9
  <10-12> mathb10
  <12-> mathb12
  } {}%
\DeclareSymbolFont{mathb}{U}{mathb}{m}{n}

\DeclareMathSymbol{\abxpitchfork}{\mathord}{mathb}{"26}

\usepackage{stackengine}

\usepackage{mathtools}

\author{Sumanta Das}
\date{}
\title{Strong Topological Rigidity of Non-Compact Orientable Surfaces}

\usepackage{thmtools}

\makeatletter
\newcommand*{\theorembookmark}{%
  \bookmark[
    dest=\@currentHref,
    rellevel=1,
    keeplevel,
  ]{%
    \thmt@thmname\space\csname the\thmt@envname\endcsname
    \ifx\@currentlabelname\@empty
    \else
      \space(\@currentlabelname)%
    \fi
  }%
}
\makeatother

\declaretheorem[title=Theorem, postheadhook=\theorembookmark]{theo}

\usepackage[shortlabels]{enumitem} 

\usepackage{varwidth}
\usepackage{bbding}
\usepackage[nottoc,notlot,notlof]{tocbibind}
\begin{document}

\maketitle

\begin{abstract}\mbox{}\\We show that every orientable infinite-type surface is properly rigid as a consequence of a more general result. Namely, we prove that if a homotopy equivalence between any two non-compact orientable surfaces is a proper map, then it is properly homotopic to a homeomorphism, provided surfaces are neither the plane nor the punctured plane. Thus all non-compact orientable surfaces, except the plane and the punctured plane, are \emph{topologically rigid in a strong sense}. \end{abstract}

\section{Introduction}All manifolds will be assumed to be second countable and Hausdorff. A surface is a two-dimensional manifold with an empty boundary. Throughout this note, all surfaces will be considered connected and orientable. We say a surface is of \emph{infinite-type} if its fundamental group is not finitely generated; otherwise, we say it is of \emph{finite-type}.

A fundamental question in topology is that if two closed $n$-manifolds are homotopy equivalent, are they homeomorphic? This has a positive answer in dimension two, as two closed surfaces having isomorphic fundamental groups are homeomorphic. But the same doesn't happen in other dimensions; for example, there are homotopy equivalent lens spaces (a particular type of spherical $3$-manifolds) that are not homeomorphic. A closed topological $n$-manifold $M$ is said to be \emph{topologically rigid} if any homotopy equivalence $N\to M$ with a closed topological $n$-manifold $N$ as the source is homotopic to a homeomorphism. The \emph{Borel conjecture} \cite[Conjecture (A. Borel)]{MR842428} asserts that every closed aspherical (i.e., $\pi_k=0$ if $k\neq 1$) manifold is topologically rigid. In dimension two, every closed surface is topologically rigid; it is known as the \emph{Dehn-Nielsen-Baer theorem} \cite[Appendix]{MR881797}. The Borel conjecture is known to be true in other dimensions under some additional hypotheses; for example, see \cite[Theorem 6.1.]{MR224099} and \cite[Theorem 0.1. i)]{MR1973051} for dimension three, and for high dimensions, consider \cite[Proof of Theorem 3.2]{MR1216623}.

Though non-compact manifolds are not rigid in the above sense, for example, in \cite[Theorem 2]{MR137105}, the author has constructed (generalizing a construction given by  J. H. C. Whitehead) uncountably many contractible open subsets of $\Bbb R^3$ such that any two of them are not homeomorphic. Similarly, for non-compact surfaces, we have several examples. In the case of finite-type surfaces, for example, we may consider the once-punctured torus and thrice-punctured sphere, which are homotopy equivalent but non-homeomorphic, as any homomorphism preserves the cardinality of the puncture set as well as the genus.  On the other hand, up to homotopy equivalence, there is precisely one infinite-type surface, but up to homeomorphism, there are $2^{\aleph_0}$ many infinite-type surfaces (see \Cref{upto}).   This note considers only non-compact surfaces and discusses their topological rigidity in the \emph{proper category}; here, proper category means the category of spaces with proper maps (recall that a map from a space $X$ to a space $Y$ is called a \emph{proper map} if the inverse image of each compact subset of $Y$ is a compact subset of $X$). At first, we define the analogs of homotopy, homotopy equivalence, etc., in the proper category.

If a homotopy $\mathcal H\colon X\times [0,1]\to Y$ is a proper map, then we call $\mathcal H$ a \emph{proper homotopy}. Two proper maps from $X$ to $Y$ are said to be \emph{properly homotopic} if there is a proper homotopy between them. We say that a proper map $f\colon X\to Y$ is a \emph{proper homotopy equivalence} if  there exists a proper map $g\colon Y\to X$ such that both $g\circ f$ and $f\circ g$ are properly homotopic to the identity maps (when such a $g$ exists, we say that $g$ is a \emph{proper homotopy inverse} of $f$.). Two spaces $X$ and $Y$ are said to have the same \emph{proper homotopy type} if there is a proper homotopy equivalence between them. It is worth noting that homotopy through proper maps is a weaker notion than proper homotopy. For example, consider $H\colon \Bbb C\times [0,1]\to \Bbb C$ given by $H(z,t)\coloneqq tz^2-z$. Being a polynomial, each $H(-,t)$ is proper. But, $H$ itself is not proper as $H\left(n,\frac{1}{n}\right)=0$ for all integers $n\geq 1$.

The analog of topological rigidity in the proper category is defined as follows: A non-compact topological manifold $M$ without boundary is said to be \emph{properly rigid} if, whenever $N$ is another boundaryless topological manifold of the same dimension and $h\colon N\to M$ is a proper homotopy equivalence, then $h$ is properly homotopic to a homeomorphism. The analog of the Borel conjecture in the proper category, often called \emph{proper Borel conjecture} \cite[Conjecure 3.1.]{MR4207574}, asserts that every non-compact aspherical topological manifold without boundary is properly rigid.

It is known that non-compact finite-type surfaces are properly rigid. Further, using the algebraic tools of classification of non-compact surfaces \cite[Theorem 4.1.]{MR275436}, Goldman showed that two non-compact surfaces of the same proper homotopy type are homeomorphic; see \cite[Corollary 11.1]{MR2616858}. We show that infinite-type surfaces are also properly rigid. In fact, we show the rigidity of all non-compact surfaces, except for the plane and the punctured plane, under a weaker assumption, namely only assuming the existence of homotopy inverse, which a priori may or may not be proper. For brevity, define a weaker version of proper homotopy equivalence: 

\begin{def*}\textup{A homotopy equivalence is said to be \emph{pseudo proper homotopy equivalence} if it is proper.}
\end{def*}
 Indeed, a proper homotopy equivalence is a pseudo proper homotopy equivalence, though not conversely: a pseudo proper homotopy equivalence has an ``ordinary'' homotopy inverse but may not have a proper homotopy inverse, for example, consider below given $\varphi$ and $\psi$. Our main theorem is the following:

\begin{mthe}
\emph{Let $f\colon \Sigma'\to \Sigma$ be a pseudo proper homotopy equivalence between two non-compact surfaces. Then $\Sigma'$ is homeomorphic to $\Sigma$. If we further assume that $\Sigma$ is homeomorphic to neither the plane nor the punctured plane, then $f$ is a proper homotopy equivalence, and there exists a homeomorphism $g_\text{homeo}\colon \Sigma\to \Sigma'$ as a proper homotopy inverse of $f$.}\label{mthe}
\end{mthe}
The reason for the exclusion of the plane and the punctured plane from the hypothesis is almost immediate; for example, consider $\varphi\colon\Bbb C\ni z\longmapsto z^2\in \Bbb C$ and $\psi\colon \Bbb S^1\times \Bbb R\ni (z,x)\longmapsto \big(z,|x|\big)\in \Bbb S^1\times \Bbb R$; each of these proper maps is a homotopy equivalence, but none of them is a proper homotopy equivalence as the degree of a proper homotopy equivalence is $\pm 1$ (see \Cref{degreeofapropermap}), though  $\deg(\varphi)=\pm 2$ (as $\varphi$ is a two-fold branched covering) and $\deg(\psi)=0$ (as $\psi$ is not surjective; see \cref{non-surjectivepropermaphasdegreezero}).

 In general, additional assumptions must be imposed on a pseudo proper homotopy equivalence to become a proper homotopy equivalence.
For example, using the binary symmetry of the Cantor tree $\mathcal T_\text{Cantor}$, we have a two-fold branched covering $f_\text{Cantor}\colon\mathcal T_\text{Cantor}\to\mathcal T_\text{Cantor}$, which is undoubtedly a pseudo proper homotopy equivalence (trees are contractible) but not a proper homotopy equivalence \big(the induced map on $\text{Ends}(\mathcal T_\text{Cantor})$ by $f_\text{Cantor}$ is non-injective; see part (1) and (3) of \Cref{endssandproperhomotopy}\big). Here is another example. Let $M$ be a  connected, non-compact, contractible, boundaryless manifold of dimension $n\geq 2$; and let $f\colon M\to M$ be the composition of the following two proper maps: a proper map $M\to [0,\infty)$ (using partition of unity) and a non-surjective proper map $[0,\infty)\to M$ corresponding to an end of $M$ (using compact exhaustion by connected codimension $0$-submanifolds; see \cite[Exercise 3.3.18]{MR3598162}). Then $f$ is a pseudo proper homotopy equivalence ($M$ is contractible) but not a proper homotopy equivalence (a proper homotopy equivalence is a surjective map as its degree is $\pm 1$; see \cref{non-surjectivepropermaphasdegreezero}).

Brown showed that a pseudo proper homotopy equivalence between two connected, finite-dimensional, locally finite simplicial complexes is a proper homotopy equivalence if and only if it induces a homeomorphism on the spaces of ends and isomorphisms on all proper homotopy groups; see \cite[Whitehead theorem]{MR0356041}.  In \cite[Corollary 4.10.]{MR334226}, the authors have shown that if $f\colon M\to N$ is a pseudo proper homotopy equivalence between two  simply-connected, non-compact, boundaryless $n$-dimensional smooth manifolds, where both $M$ and $N$ both are simply-connected at infinity, then $f$ is a proper homotopy equivalence if and only if $\deg(f)=\pm 1$. Another interesting statement in this context is that a proper map $f\colon X\to Y$ between two locally finite, infinite, connected, $1$-dimensional CW-complexes is a proper homotopy equivalence if $\text{Ends}(f)$ is a homeomorphism and $f$ is an extension of a proper homotopy equivalence $X_g\to Y_g$ \big(where $X_g$ (resp. $Y_g$) denotes the smallest connected sub-complex of $X$ (resp. $Y$) that contains
all immersed loops of $X$ (resp. $Y$)\big); see  \cite[Corollary 3.7.]{arxiv}.

We conclude this section by citing a few more related results of two different flavors: when does a proper homotopy equivalence exist, and if it does exist, whether it determines the space up to homeomorphism. Similar to Kerékjártó’s classification theorem (see \Cref{richard2}), there exists a classification of graphs up to proper homotopy type: Two
locally finite, infinite, connected,
$1$-dimensional CW-complexes $X$ and $Y$ have the same proper homotopy type if and only if $\text{rank}\big(\pi_1(X)\big)=\text{rank}\big(\pi_1(Y)\big)$ and there exists a homeomorphism $\varphi\colon\text{Ends}(X)\to \text{Ends}(Y)$ with $\varphi\big(\text{Ends}(X_g)\big)=\text{Ends}(Y_g)$; see \cite[Theorem 2.7.]{MR1082009}.

As stated earlier, any two non-compact surfaces of the same proper homotopy type are homeomorphic. Sometimes this happens also in other dimensions; for instance, a boundaryless topological manifold of dimension $n\geq 3$ with the same proper homotopy type of $\Bbb R^n$ is homeomorphic to $\Bbb R^n$; see \cite[Theorem 1]{MR150745} for $n=3$, \cite[Corollary 1.2.]{MR679066} for $n=4$, and \cite[Corollary 1.4.]{MR238325} for $n\geq 5$.  In contrast, there are exotic pairs, for example,  two non-compact, connected, boundaryless manifolds $N$ and $M$ of the dimension $n\geq 5$ exist, where $N$ is smoothable, and $M$ is a nonuniform arithmetic manifold, such that $M$ and $N$ has the same proper homotopy type, but $M$ is not homeomorphic $N$; see \cite[Theorem 2.6.]{MR3403758}  with \cite[Pages 137 and 138]{MR4207574}.

\subsection{Main results}The analog of \cite[First proof of Theorem 8.9.]{MR2850125} in the proper category is \Cref{MC1.1.}, which follows almost directly from our main result \Cref{MR1}. Indeed, \Cref{MR1} is more general.

 \begin{theo}[Strong topological rigidity] 
\textup{Let $f\colon \Sigma'\to \Sigma$ be a pseudo proper homotopy equivalence between two non-compact surfaces. Suppose $\Sigma$ is homeomorphic to neither $\Bbb R^2$ nor $\Bbb S^1\times \Bbb R$. Then $\Sigma'$ is homeomorphic to $\Sigma$, and $f$ is properly homotopic to a homeomorphism.} \label{MR1}
\end{theo}

\begin{theo}[Proper rigidity]
 \textup{If $f\colon \Sigma'\to \Sigma$ is a proper homotopy equivalence between two non-compact surfaces, then $\Sigma'$ is homeomorphic $\Sigma$ and $f$ is properly homotopic to a homeomorphism.} \label{MC1.1.}
\end{theo}

A theorem of Edmonds ~\cite[Theorem 3.1.]{MR541331} says that any $\pi_1$-injective map of degree one between two closed surfaces is homotopic to a homeomorphism. The analog fact for non-compact surfaces is \Cref{MR2}, which classifies all $\pi_1$-injective degree one maps between two non-compact surfaces and also follows almost directly from \Cref{MR1}. 
\begin{theo}[Classification of \texorpdfstring{$\pi_1$}{π₁}-injective degree one maps] 
 \textup{Let $\Sigma,\Sigma'$ be any two non-compact oriented surfaces. Suppose there exists a $\pi_1$-injective proper map $f\colon \Sigma'\to \Sigma$ of degree $\pm 1$. Then $\Sigma$ is homeomorphic to $\Sigma'$, and $f$ is properly homotopic to a homeomorphism.} \label{MR2}
\end{theo}
 Proofs of \Cref{MR1}, \Cref{MC1.1.}, and \Cref{MR2} can be found in \Cref{allproofs}.
\subsection{Outline of the proof of \texorpdfstring{\Cref{MR1}}{Theorem \ref{MR1}}}

Let $f\colon \Sigma'\to \Sigma$ be a pseudo proper homotopy equivalence between two non-compact oriented surfaces. Suppose $\Sigma$ is homeomorphic to neither $\Bbb R^2$ nor $\Bbb S^1\times \Bbb R$.
\begin{subsubsection}{\textup{\textbf{Decomposition and transversality}}}
Let $\mathscr C$ be a locally finite pairwise disjoint collection  of smoothly embedded circles on $\Sigma$ such that $\mathscr C$ decomposes $\Sigma$ into bordered sub-surfaces, and a complementary component of this decomposition is homeomorphic to either the one-holed torus or the pair of pants or the punctured disk (see \Cref{completedecomposition}).

Properly homotope $f$ to make it smooth as well as transverse to $\mathscr C$. Thus, $f^{-1}(\mathcal C)$ is either \emph{empty} or a pairwise disjoint finite collection of smoothly embedded circles on $\Sigma'$ for each component $\mathcal C$ of $\mathscr C$. See \Cref{remarktranshomotopy}.
\end{subsubsection}

\begin{subsubsection}{\textup{\textbf{Removing unnecessary circles}}}

 Now, following the three steps given below, we properly homotope $f$ further so that for each component of $\mathcal C$ of $\mathscr C$, either $f^{-1}(\mathcal C)$ is empty or $f\vert f^{-1}(\mathcal C)\to \mathcal C$ is a homeomorphism.
\begin{enumerate}[leftmargin=\widthof{[(1.1.)]}+\labelsep]

\item[(1)] Notice that $f^{-1}(\mathscr C)$ may have infinitely many disk-bounding components. But, in such a case, an arbitrarily large disk in $\Sigma'$ bounded by a component of the locally finite collection $f^{-1}(\mathscr C)$ is not possible as $\Sigma'\not\cong \Bbb R^2$ (see \Cref{charofplane}), i.e., there always exists an ``outermost disk'' bounded by some component of $f^{-1}(\mathscr C)$. Now, properly homotope $f$ to remove all disk bounding components of $f^{-1}(\mathscr C)$ upon considering all these outermost disks simultaneously. See \Cref{diskremoval}.
\item[(2)] Thereafter, using $\pi_1$-bijectivity of $f$, properly homotope $f$ to map each (primitive) component of $f^{-1}(\mathscr C)$ onto a component of $\mathscr C$ homeomorphically. See \Cref{deg1tohomeo}.
\item[(3)] Since $f$ has homotopy left inverse,  any two components of $f^{-1}(\mathscr C)$ co-bound an annulus in $\Sigma'$ if and only if their $f$-images are the same,  i.e., arbitrarily large annulus in $\Sigma'$ co-bounded by two components of $f^{-1}(\mathscr C)$ is impossible. So, considering all these ``outermost annuli'' simultaneously, we complete the goal, as stated in the beginning. See \Cref{annulusremovalfinal}.
 \end{enumerate}\label{removingunnecessarycircles}\end{subsubsection}
\begin{subsubsection}{\textbf{\textup{Showing \texorpdfstring{$f$}{f} is a degree \texorpdfstring{$\pm 1$}{±1} map (see \texorpdfstring{\Cref{properplushomotopyequivalnceisofdegreeone}}{Theorem 2.7.1})}}}
To rule out the possibility that $f^{-1}(\mathcal C)$ is empty, where $\mathcal C$ is a component of $\mathscr C$, we prove  $\deg(f)=\pm 1$; this is because $\deg(f)$ remains the same after any proper homotopy of $f$, and a map of non-zero degree is surjective; see \Cref{non-surjectivepropermaphasdegreezero} and \Cref{properhomotopypreservessurjectivity}. Our aim is to properly homotope $f$ to obtain a closed disk $\mathcal D\subseteq \Sigma$ so that $f\vert f^{-1}(\mathcal D)\to \mathcal D$ becomes a homeomorphism, and thus we show $\deg(f)=\pm 1$; see \Cref{degreeonemapchecking}.  The argument is based on finding a smoothly embedded finite-type bordered surface $\mathbf S$ in $\Sigma$ such that for each component $c$ of $\partial \mathbf S$; we have $f^{-1}(c)\neq \varnothing$, even after any proper homotopy of $f$. Depending on the nature of $\mathbf S$, we consider two cases.

\begin{itemize}[leftmargin=\widthof{[(1.1)]}+\labelsep]
    \item[(1)]  If $\Sigma$ is either  an infinite-type surface or any $S_{g,0,p}$ with high complexity (i.e., $g+p\geq 4$ or $p\geq 6$), then using $\pi_1$-surjectivity of $f$, we can choose $\mathbf S$ as a smoothly embedded pair of pants in $\Sigma$ such that $\Sigma\setminus \mathbf S$ has at least two components and every component of $\Sigma\setminus \mathbf S$ has a non-abelian fundamental group; see \Cref{exitenceofessentialpairofpants1} and \Cref{exitenceofessentialpairofpants2}. Properly homotope $f$ so that it becomes transverse to $\partial \mathbf S$. Then  remove unnecessary components from the transversal pre-image $f^{-1}(\partial \mathbf S)$, i.e., after a proper homotopy, we may assume $f\vert f^{-1}(c)\to c$ is a homeomorphism for each component $c$ of $\partial \mathbf S$. Now, since $f$ is $\pi_1$-injective, by the rigidity of pair of pants (see \Cref{rigidityofpairofpants}), after a proper homotopy, one can show that $f\vert f^{-1}(\mathbf S)\to \mathbf S$ is a homeomorphism; see \Cref{inversepairofpants}. Therefore, the required $\mathcal D$ can be any disk in $\text{int}(\mathbf S)$.

    \item[(2)] If $\Sigma$ is a finite-type surface, then we choose a smoothly embedded punctured disk in $\Sigma$ as $\mathbf S$ so that the puncture of $ \mathbf S$ is an end $e\in \text{im}\big(\text{Ends}(f)\big)\subseteq \text{Ends}(\Sigma)$. By \Cref{surjectivityonendpreservesbyproperhomotopy}, it means every deleted neighborhood of $e$ in $\Sigma$ intersects $\text{im}(f)$, even after any proper homotopy of $f$.
    Now, properly homotope $f$ so that it becomes transverse to $\partial \mathbf S$. Then  remove unnecessary components from the transversal pre-image $f^{-1}(\partial \mathbf S)$, i.e., after a proper homotopy, we may assume $f\vert f^{-1}(\partial \mathbf S)\to\partial \mathbf S$ is a homeomorphism (as $\Sigma\not\cong \Bbb S^1\times \Bbb R$, the fundamental group of $\Sigma\setminus \mathbf S$ is non-abelian; and so $\pi_1$-surjectivity of $f$ says $f^{-1}(\partial \mathbf S)\neq \varnothing$, even after any proper homotopy of $f$). Now, since $f$ is $\pi_1$-injective, by the proper rigidity of the punctured disk (see \Cref{Alexandertrick}), after a proper homotopy, one can show that $f\vert f^{-1}(\mathbf S)\to \mathbf S$ is a homeomorphism; see \Cref{inversepunctureddisk}. Therefore, the required $\mathcal D$ can be any disk in $\text{int}(\mathbf S)$.
\end{itemize}\label{degreeone}\end{subsubsection}
 
\begin{subsubsection}{\textup{\textbf{Inverse decomposition}}}  By the last three parts, after a proper homotopy,  removing unnecessary components from the transversal pre-image $f^{-1}(\mathscr C)$, we may assume that $f\vert f^{-1}(\mathcal C)\to \mathcal C$ is a homeomorphism for each component $\mathcal C$ of $\mathscr C$. Thus, $\mathscr C$ and $f^{-1}(\mathscr C)$ decompose $\Sigma$, $\Sigma'$, respectively; and there is a shape-preserving bijective-correspondence between these two collections of complementary components (see \Cref{inversepairofpants} and \Cref{inversepunctureddisk}). On each complementary component, apply either the rigidity of compact bordered surfaces (see \Cref{rigidityofpairofpants}) or  the proper rigidity of the punctured disk (see \Cref{Alexandertrick}). Thus, we have a collection of  boundary-relative proper homotopies so that by pasting them,  a proper homotopy from $f$ to a homeomorphism $\Sigma'\to \Sigma$ can be constructed; see the proof of \Cref{MR1} in \Cref{allproofs}.

\end{subsubsection}

 \section{Background}
 \subsection{Conventions}
A \emph{bordered surface} (resp. \emph{surface}) is a connected, orientable two-dimensional manifold with a non-empty (resp. an empty) boundary. For integers $g\geq 0,\ b\geq 0, p\geq 0$, denote the  $2$-manifold of genus $g$ with $b$ boundary components by $S_{g,b}$; and let $S_{g,b, p}$ be the $2$-manifold after removing $p$ points from $\text{int}(S_{g,b})$. Note that for a manifold $M$, we use $\text{int}(M)$ to denote the interior of $M$. Sometimes $S_{0,0}$, $S_{0,2}$, $S_{0,3}$, $S_{1,2}$, and $S_{0,1,1}$ will be called a disk, an annulus, a pair of pants, a two-holed torus, and a punctured disk, respectively.

  We say a connected $2$-manifold with or without boundary is of \emph{infinite-type} if its fundamental group is not finitely generated; otherwise, we say it is of \emph{finite-type}.
 
  \begin{subsection}{Simple closed curves on two-manifolds} 
 
\begin{definition}
\textup{Let $\mathbf S$ be a connected, orientable two-dimensional manifold with or without boundary. A \emph{circle} (resp. \emph{smoothly embedded circle}) on $\mathbf S$ is the image of an embedding (resp. a smooth embedding) of $\Bbb S^1$ into $\mathbf S$. We say a circle $\mathcal C$ on $\mathbf S$ is a \emph{trivial circle} if there is an embedded disk $\mathcal D$ in $\mathbf S$ such that $\partial \mathcal D=\mathcal C$; and, we say a circle $\mathcal C$ on $\mathbf S$ is a \emph{primitive circle} if it is not a trivial circle.}
\end{definition} 
The following theorem justifies naming a non-disk bounding circle as a primitive circle: a primitive circle represents a primitive element of the fundamental group. Recall that an element $g$ of a group $G$ is \emph{primitive} if there does not exist any $h\in G$ so
that $g=h^k$, where $|k|>1$. 

\begin{theorem}{\textup{\cite[Theorems 1.7. and  4.2.]{MR214087}}}\textup{ Let $\mathbf S$ be a connected, orientable two-dimensional manifold with or without boundary. Let $\mathcal C$ be a primitive circle on $\mathbf S$, and let $f\colon \Bbb S^1\hookrightarrow \mathbf S$ be an embedding with $f(\Bbb S^1)=\mathcal C$. Then $[f]\in \pi_1(\mathbf S)$ is a primitive element. In particular, $[f]$ is a non-trivial element of $\pi_1(\mathbf S).$} \label{primitivecircle}\end{theorem}
Recall that for a path-connected space $X$, there is a bijective correspondence between the set of all conjugacy classes of $\pi_1(X,*)$ and the set of all free homotopy classes of maps $\Bbb S^1\to X$. The next theorem says that two pairwise disjoint freely homotopic primitive circles on a two-manifold co-bound an annulus.
\begin{theorem}{\textup{\cite[Lemma 2.4.]{MR214087}}}\textup{ Let $\mathbf S$ be a connected, orientable two-dimensional manifold with or without boundary. Let $\ell_0,\ell_1\colon \Bbb S^1\hookrightarrow \mathbf S$ be two  embeddings such that $\ell_0(\Bbb S^1)$ is a smoothly embedded submanifold of $\mathbf S$ and $\ell_0(\Bbb S^1)\cap \ell_1(\Bbb S^1)=\varnothing$. If $\ell_0$ and $\ell_1$ represent the same non-trivial conjugacy class in $\pi_1(\mathbf S,*)$, then there is a  embedding $\mathcal L\colon \Bbb S^1\times [0,1]\hookrightarrow \mathbf S$ so that $\mathcal L(-,0)=\ell_0$ and $\mathcal L(-,1)=\ell_1$. }
\label{AnnulusEmbedding}
\end{theorem}

\label{notation}
\end{subsection}

 \begin{subsection}{Ends of spaces}\label{ends} Let $X$ be a connected, separable, locally compact, locally connected Hausdorff  ANR (absolute neighborhood retract) space. For example, $X$ can be any connected topological manifold. We say $X$ admits an \emph{efficient exhaustion by compacta} if there is a nested sequence $K_1\subseteq K_2\subseteq\cdots$  of compact, connected subsets of $X$ such that $\cup_i K_i=X$, $K_i\subseteq \text{int}(K_{i+1})$, $\cap_i (X\setminus K_i)=\varnothing$, and closure of each component of  any $X\setminus K_i$ is non-compact. For the existence of efficient exhaustion of $X$ by compacta,  see \cite[Exercise 3.3.4]{MR3598162}.

 Let $\text{Ends}(X)$ be the set of all sequences $(V_1, V_2,...)$, where $V_i$ is a component of $X\setminus K_i$ and $V_1\supseteq V_2\supseteq \cdots$. Give $X^\dag\coloneqq X\cup \text{Ends}(X)$ with the topology generated by the basis consisting of all open subsets of $X$, and all sets $V_i^\dag$, where
$$V_i^\dag\coloneqq V_i\cup \big\{(V_1', V_2',...)\in \text{Ends}(X)\big| V_i'= V_i\big\}.$$Then $X^\dag$ is separable, compact, and metrizable such that $X$ is an open dense subset of $X^\dag$; it is known as the \emph{Freudenthal compactification} of $X$ (recall that we say a space $X_\textbf{c}$ is a \emph{compactification} of $X$ if $X_\textbf{c}$ is compact Hausdorff space, and $X$ is a dense subset of $X_\textbf{c}$).  The subspace $\text{Ends}(X)$ of $X^\dag$ is a totally-disconnected space; hence $\text{Ends}(X)$ is a closed subset of the Cantor set.  

The Freudenthal compactification \emph{dominates} any other   compactification: If $\widetilde X$ is a compactification of $X$ such that  $\widetilde X\setminus X$ is totally-disconnected,  then there exists a map $f\colon X^\dag\to \widetilde X$ extending $\text{Id}_X$.

Also, the Freudenthal compactification is \emph{unique} in the following sense: If $X^{\dag\dag}$ is a compactification of $X$ such that $X^{\dag\dag}\setminus X$ is totally-disconnected and $X^{\dag\dag}$ dominates any other compactification, then there exists a homeomorphism $X^{\dag\dag}\to X^\dag$ extending $\text{Id}_X$; see \cite[Theorem 3.1]{MR2616858}. Thus, the definition of $\text{Ends}(X)$ is independent of the choice of efficient exhaustion of $X$ by compacta.

Now, we consider a relationship between \text{Ends} and proper maps.\begin{proposition}{\textup{\cite[Proposition 3.3.12]{MR3598162}}}\textup{ Let $X$ and $Y$ be two connected, separable, locally compact, locally connected Hausdorff  ANRs. Then we have the following:}
\begin{itemize}

\item[\textup{(1)}] \textup{Every proper map $f\colon X\to Y$ induces a map $\text{Ends}(f)\colon \text{Ends}(X)\to \text{Ends}(Y)$ that can be used to extend $f\colon X\to Y$ to a map $f^\dag\colon X^\dag\to Y^\dag$ between the Freudenthal compactifications.}
\item[\textup{(2)}] \textup{If two proper maps $f_0,f_1\colon X\to Y$ are properly homotopic, then $\text{Ends}(f_0)=\text{Ends}(f_1)$.}

\item[\textup{(3)}]\textup{If $f\colon X\to Y$ is a proper homotopy equivalence, then $\text{Ends}(f)\colon \text{Ends}(X)\to \text{Ends}(Y)$ is a homeomorphism.}
\end{itemize}  \label{endssandproperhomotopy}\end{proposition}
More about the ends of spaces and proper homotopy can be found in \cite{MR1361888} and \cite{MR1410261}. 
  \end{subsection}

 \subsection{Kerékjártó's classification theorem and Ian Richard's representation theorem} Let $\Sigma$ be a non-compact surface with an efficient exhaustion $\{K_i\}_1^\infty$. Let $e\coloneqq (V_1, V_2,...)\in \text{Ends}(\Sigma)$ be an end, where $V_i$ is a component of $X\setminus K_i$. We say $e$ is a \emph{planar end} if $V_i$ is embeddable in $\Bbb R^2$ for some positive integer $i$. An end is said to be \emph{non-planar} if it is not planar. Denote the subspace of $\text{Ends}(\Sigma)$ consisting of all planar (resp. non-planar) ends by $\text{Ends}_\text{p}(\Sigma)$ \big(resp. $\text{Ends}_\text{np}(\Sigma)$\big). Note that $\text{Ends}_\text{p}(\Sigma)$ is an open subset of $\text{Ends}(\Sigma)$. Define the \emph{genus} of $\Sigma$ as $g(\Sigma):=\sup g(\mathcal S)$, where $\mathcal S$ is a compact bordered subsurface of $\Sigma.$ Therefore, the genus counts the number of handles of a surface, i.e., the number of embedded copies of $S_{1,1}$ in a surface, which may be any non-negative integer or countably infinite.
 
 \begin{theorem}{\textup{\big(Kerékjártó's classification of non-compact surfaces \cite[Theorem 1]{MR143186}\big)}} \textup{Let $\Sigma$ and $\Sigma'$ be non-compact surfaces of genus $g,g'$, respectively.  Then $\Sigma$ is homeomorphic to $\Sigma'$ if and only if $g=g'$ and there is homeomorphism $\varphi\colon \text{Ends}(\Sigma)\to \text{Ends}(\Sigma')$ with $\varphi\big( \text{Ends}_\text{np}(\Sigma)\big)= \text{Ends}_\text{np}(\Sigma')$.}
\label{richard2}\end{theorem}

\begin{theorem}{\textup{\big(Realization of ends and representation of a non-compact surface \cite[Theorems 2, 3]{MR143186}\big)}}
\textup{Let $\mathscr E_\text{np}\subseteq \mathscr E$ be two closed totally-disconnected subsets of $\Bbb S^1$, and let $\mathscr G$ be an at most countable set with the following properties: $\mathscr E\neq \varnothing$, and  $\mathscr E_\text{np}\neq \varnothing$ if and only if $\mathscr G$ is infinite. Define $\Bbb D\coloneqq \{z\in \Bbb C:0\leq |z|\leq 1\}$. Then there exists a pairwise disjoint collection $\{\mathcal D_i:i\in \mathscr G\}$ of disks in $\text{int}(\Bbb D)$ such that a  point $p\in \Bbb D$ is an element of $\mathscr E_\text{np}$ if and only if every neighborhood of $p$ in $\Bbb D$ contains infinitely many elements of $\{\mathcal D_i:i\in \mathscr G\}$. Moreover,   $\mathbf S\coloneqq\big(\Bbb D\setminus \mathscr E\big)\setminus\cup_{i\in \mathscr G}\text{int}(\mathcal D_i)$ is a non-compact bordered surface, and $$D\mathbf S\coloneqq\frac{(\mathbf S\times 0)\sqcup (\mathbf S\times 1)}{(p,0)\sim (p,1),\ p\in \partial \mathbf S}$$ is a $|\mathscr G|$-genus non-compact surface with $\text{Ends}(D\mathbf S)\cong \mathscr E$ and $\text{Ends}_\text{np}(D\mathbf S)\cong \mathscr E_\text{np}$.}

\textup{Thus, given any non-compact surface $\Sigma$, in this procedure, if we assume $\mathscr E_\textup{np}\subseteq \mathscr E$ is homeomorphic to the pair $\text{Ends}_\textup{np}(\Sigma)\subseteq\text{Ends}(\Sigma)$, and $|\mathscr G|$ is equal to $g(\Sigma)$; then $D\mathbf S\cong \Sigma$, by \Cref{richard2}. }\label{richard1}\end{theorem}

\begin{remark}
\textup{The classification of non-compact bordered surfaces is also possible:  When the boundary is compact, it follows from \Cref{richard2} together with \cite[Proposition A.3.]{MR561583}. When each boundary component is compact, this follows from \cite{unclassified} (based on the classification of their interiors) or \cite[Theorem A.7]{MR561583} (based on the classification of non-compact surfaces obtained from gluing a disk along each boundary component). For arbitrary boundary, see \cite[Theorem 2.2.]{MR542887}.}
\end{remark}

\subsection{Goldman's inductive procedure of constructing all non-compact surfaces}
A non-compact surface $\Sigma_\text{std}$ is said to be in \emph{standard form} if it is built up from four building blocks, $S_{0,1}$, $S_{0,2}$, $S_{0,3}$, and $S_{1,2}$, in the following inductive manner: Start with $S_{0,1}$. Suppose the $i$-th step of the induction has already been done. Let $K_i$ be the compact bordered subsurface of $\Sigma_\text{std}$ after the $i$-th step of induction. In particular, $K_1\cong S_{0,1}$. Now, to obtain $K_{i+1}$ from $K_i$, consider one of the last three building blocks, say $\mathcal S$ (i.e., $\mathcal S$ is homeomorphic to either $S_{0,2}$, $S_{0,3}$, or $S_{1,2}$); finally, suitably identify one boundary circle of $\mathcal S$ with a boundary circle of $K_i$. See \Cref{figureofinductiveconstruction}.

\begin{theorem}{\textup{\cite[Section 2.6.]{MR275436} and \cite[Page 173]{v_Ker_kj_rt__1923}}}
\textup{Let $\Sigma$ be a non-compact surface. Then $\Sigma$ is homeomorphic to a non-compact surface $\Sigma_\text{std}$ in standard form. Thus every non-compact surface is homeomorphic to a non-compact surface constructed using an inductive procedure as above, though two non-compact surfaces obtained from two different inductive procedures may be homeomorphic.}\label{incudctiveconstruction} 
\end{theorem}
\begin{figure}[ht]$$\adjustbox{trim={0.0\width} {0.0\height} {0.0\width} {0.0\height},clip}{\def\svgwidth{\linewidth}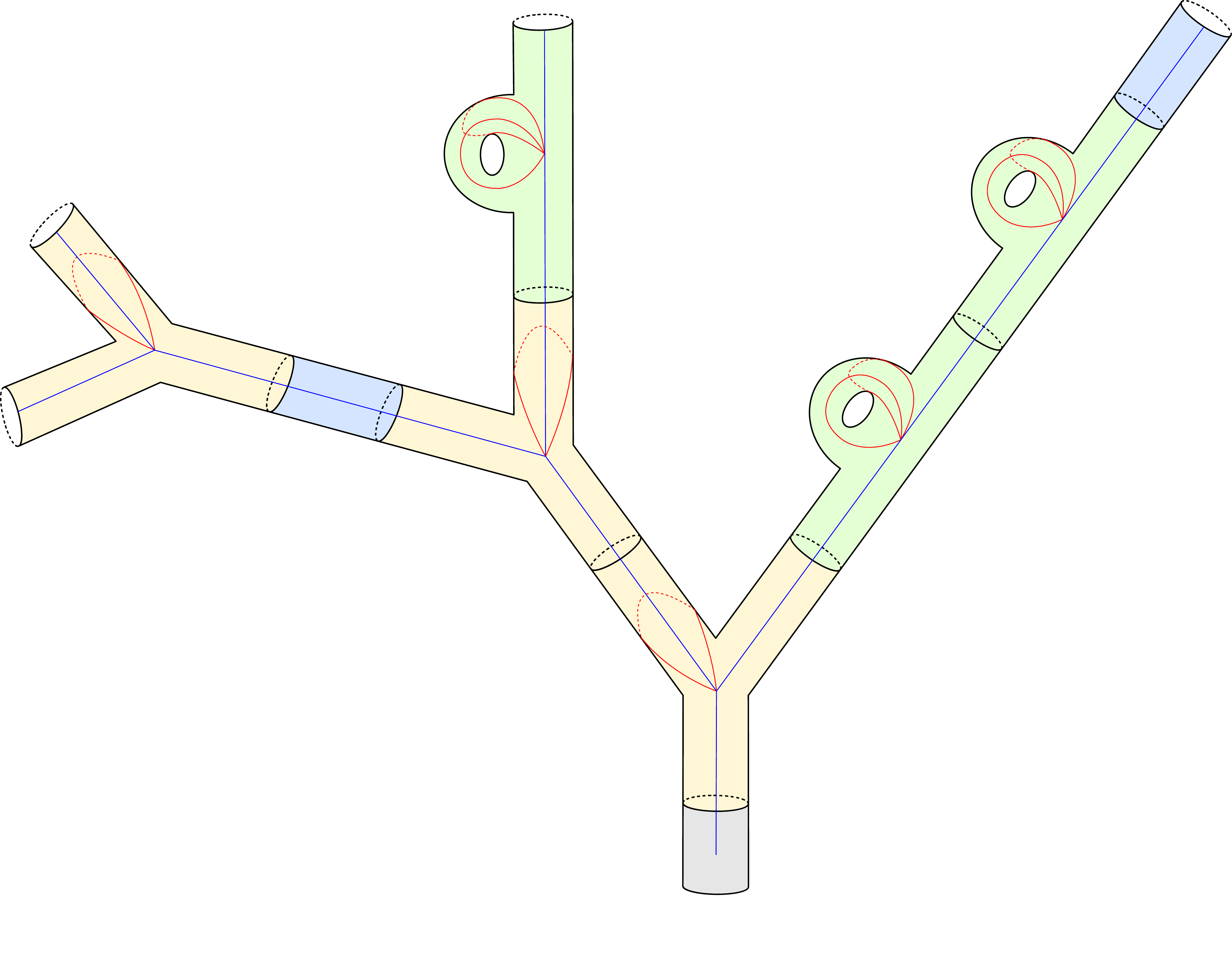}$$\caption{Inductive construction of any non-compact surface and its spine uses four compact bordered surfaces: disk, annulus, pair of pants, and torus with two holes.  }\label{figureofinductiveconstruction}\end{figure}

 \begin{theorem}{\textup{\cite[Section 2.6. and Section 7.3.]{MR275436}}}
 \textup{The graph in \Cref{figureofinductiveconstruction} consisting of blue straight line segments and red circles is a deformation retract of the surface $\Sigma$. Thus, $\Sigma$ is homotopy equivalent to the wedge of at most countably many circles. In particular, $\pi_1(\Sigma)$ is free.} \label{spine}
 \end{theorem}

\begin{remark}
\textup{An alternative way of proving the last two sentences of \Cref{spine} is given in \cite[Lemma 3.2.2]{scottbook}.}
\end{remark}

\begin{subsection}{The degree of a proper map}\label{degreeofapropermap}
We use singular cohomology with compact support  to define the notion of the degree of a proper map. Recall that for a topological manifold $X$, the $r$-th singular cohomology with compact support $H^r_\textbf{c}(X,\partial X;\Bbb Z)$ is equal to the direct limit $\varinjlim H^r\big(X, \partial X\cup(X\setminus K);\Bbb Z\big)$, where $K$ is a compact subset of $X$ and various maps to define this direct system are inclusion induced maps. It is worth noting that when $X$ is compact topological manifold, $H_\textbf{c}^r(X, \partial X;\mathbb Z)=H^r(X, \partial X;\mathbb Z)$ for all $r$. 

Let $X$ and $Y$ be two topological manifolds. If $f\colon X\to Y$ is a proper map with $f(\partial X)\subseteq \partial Y$, then for each $r$, $f$ induces a map $H^r_\textbf{c}(f)\colon H^r_\textbf{c}(Y,\partial Y;\Bbb Z)\to  H^r_\textbf{c}(X,\partial X;\Bbb Z)$ so that $H^r_\textbf{c}$ becomes a functor in the following sense: the induced map of the identity is the identity, and the induced map of a (well-defined) composition of two proper maps (each of which sends boundary into boundary) is the composition of their induced maps. Moreover, if $\mathcal H\colon X\times [0,1]\to Y$ is a proper homotopy such that $\mathcal H(\partial X, t)\subseteq \partial Y$ for each $t\in [0,1]$, then $H^r_\textbf{c}\big(\mathcal H(-,0)\big)=H^r_\textbf{c}\big(\mathcal H(-,1)\big)$ for all $r$. For more details, see \cite[Pages 320, 322, 323, 339, 341]{MR1325242}.

Let $M$ be a connected, orientable, topological $n$-manifold. Then $H_\textbf{c}^n(M, \partial M;\mathbb Z)$ is an infinite cyclic group; see \cite[Page 342]{MR1325242}. If we choose an orientation of $M$ (i.e., $M$ is oriented), then there exists a unique element $[M]\in H_\textbf{c}^n(M, \partial M;\mathbb Z)$ such that the following hold: $(1)$ $[M]$ generates $H_\textbf{c}^n(M, \partial M;\mathbb Z)$, and $(2)$ for each $x\in M\setminus \partial M$, the unique generator of $H^n(M,M\setminus x;\Bbb Z)$, which comes from the chosen orientation of $M$, is sent to $[M]$ by the inclusion-induced isomorphism $H^n(M,M\setminus x;\Bbb Z)\to H^n_\textbf{c}(M,\partial M;\Bbb Z)$; see \cite[Proof of Lemma 2.1]{MR192475}. Thus, if $f\colon M\to N$ is a proper map between two connected, oriented, topological $n$-manifolds with $f(\partial M)\subseteq \partial N$, then the \emph{(compactly supported cohomological) degree} of $f$ is the unique integer $\deg(f)$ defined as follows: $H^n_\textbf{c}(f)\big([N]\big)=\deg(f)\cdot [M]$. 

By the previous two paragraphs, we have the following: $(i)$ When manifolds are compact, the notion of compactly supported cohomological degree agrees with the notion of the usual degree defined by singular cohomology. $(ii)$ The degree is proper homotopy invariant: If $f, g\colon M\to N$ are proper maps between two connected, oriented, topological $n$-manifolds with $f(\partial M)\cup g(\partial M)\subseteq \partial N$ such that there is a proper homotopy $\mathcal H\colon M\times [0,1]\to N$ with $\mathcal H\big(\partial M\times[0,1]\big)\subseteq \partial N$, then $\deg(f)=\deg(g)$. $(iii)$ The degree is multiplicative: The degree of the (well-defined) composition of two proper maps (each of which sends boundary into boundary) is the product of their degrees.

Therefore, the degree of a proper homotopy equivalence between two oriented, connected, boundaryless $n$-manifolds is $\pm 1$ due to $(ii)$ and $(iii)$ above. We use the following well-known characterizations of a map of degree $\pm 1$. In the below two theorems, $D$ is a disk in a smooth $n$-manifold $X$ means $D$ is the image of $\{z\in\Bbb R^n:|z|\leq 1\}$ under a smooth embedding $\{z\in\Bbb R^n:|z|\leq 2\}\hookrightarrow X$.

\begin{theorem}{\textup{\cite[Lemma 2.1b.]{MR192475}}}
\textup{Let $f\colon M\to N$ be a proper map between two connected, oriented, smooth manifolds of the same dimension such that $f^{-1}(\partial N)=\partial M$. Suppose for a disk $D$ in $\textup{int}(N)$, $f^{-1}(D)$ is a disk in $\text{int}(M)$ such that $f$ maps $f^{-1}(D)$ homeomorphically onto $D$. Then $\deg (f)=+1$ or $-1$ according as $f\vert f^{-1}(D)\to D$ is orientation-preserving or orientation-reversing.}
\label{degreeonemapchecking}\end{theorem}

The following theorem is due to Hopf, which says that for a degree one map, we can achieve such a disk with nice properties, as mentioned in \Cref{degreeonemapchecking} above, after a proper homotopy. 
\begin{theorem}{\textup{\cite[Theorems 3.1 and 4.1]{MR192475}}} \textup{Let  $f\colon M\to N$ be a proper map between two connected, oriented, smooth manifolds of the same dimension such that $f^{-1}(\partial N)\subseteq \partial M$. Suppose $\deg(f)=\pm 1$. Then there is a proper map $g\colon M\to N$ with $g(\partial M)\subseteq \partial N$ and a homotopy $\mathcal H\colon M\times [0,1]\to N$ with the following properties:} 
\begin{itemize}
    \item \textup{There exists a compact subset $K\subseteq \text{int}(M)$ such that $\mathcal H(x,t)=f(x)$ for all $(x,t)\in (M\setminus K)\times [0,1]$. In particular,  $H$ is a proper homotopy and $\mathcal H(\partial M,t)\subseteq \partial N$ for all $t\in [0,1]$.}
    \item\textup{There exists a disk $D\subseteq \textup{int}(N)$ such that $g^{-1}(D)$ is a disk in $\text{int}(M)$ and $g\vert g^{-1}(D)\to D$ is a homeomorphism.}
\end{itemize}
\label{reverseprocessindegreefinding}
\end{theorem}

The theorem below is due to Olum, which roughly says that when there is a degree one map, the domain is more massive than the co-domain.
\begin{theorem}{\textup{\cite[Corollary 3.4]{MR192475}}} \textup{Let  $f\colon M\to N$ be a proper map between two connected, oriented, topological manifolds of the same dimension such that $f(\partial M)\subseteq \partial N$. If $\deg(f)=\pm 1$, then $\pi_1(f)\colon \pi_1(M)\to \pi_1(N)$ is surjective.}
\label{degreeonemapsarepi1surjective}
\end{theorem}

\end{subsection}

\section{Ingredients For Proving \texorpdfstring{\Cref{MR1}}{Theorem \ref{MR1}}}

\subsection{Decomposition of a non-compact surface into pair of pants and punctured disks}
  Every compact surface of genus $g\geq 2$ is the union (with pairwise disjoint interiors) of $2g-2$ many copies of the pair of pants. But the same thing doesn't happen for non-compact surfaces; for example, the thrice punctured sphere is not a union (with pairwise disjoint interiors) of copies of the pair of pants; we need copies of the punctured disk. The main aim of this section is to prove that every non-compact surface, except the plane and the once punctured torus, decomposes into copies of the pair of pants and copies of the punctured disk when we cut it along a collection of circles, where each circle of this collection has an open neighborhood that does not intersect with any other circles of this collection.

First, we define a few terminologies.

\begin{definition}
\textup{Let $X$ be a space, and let $\{X_\alpha:\alpha\in \mathscr I\}$ be a collection of subsets of $X$. We say $\{X_\alpha:\alpha\in \mathscr I\}$ is a \emph{locally finite collection} and write $X_\alpha\to \infty$ if, for each compact subset $K$ of $X$, $X_\alpha\cap K=\varnothing$ for all but finitely many $\alpha\in \mathscr I$.}
\end{definition}
\begin{definition}
\textup{Let $\mathscr A$ be a pairwise disjoint collection of smoothly embedded circles on a surface $\Sigma$. We say $\mathscr A$ is a \emph{locally finite curve system} (in short, \textup{LFCS}) on $\Sigma$ if $\mathscr A$ is a locally finite collection.} \label{LFCS}
\end{definition}
\begin{remark}
\textup{Let $\mathscr A$ be an \textup{LFCS} on a surface $\Sigma$. Notice that $\cup \mathscr A$ (i.e., the union of all elements of $\mathscr A)$ is a closed subset of $\Sigma$ as well as a smoothly embedded submanifold of $\Sigma$ so that the set of all components of $\cup\mathscr A$ is $\mathscr A$. But to avoid too many notations,  whenever needed, we will think of $\mathscr A$ and $\cup\mathscr A$ as the same without any harm.}
\end{remark}
\begin{definition}
\textup{Let $\mathscr A$ be an \textup{LFCS} on a surface $\Sigma$. Suppose there exists an at most countable collection $\{\Sigma_n\}$ of bordered sub-surfaces of $\Sigma$ such that the following hold: $(1)$ each $\Sigma_n$ is a closed subset of $\Sigma$; $(2)$ $\textup{int}(\Sigma_n)\cap \textup{int}(\Sigma_m)=\varnothing$ if $n\neq m$; $(3)$ $\cup_n\Sigma_n=\Sigma$; and $(4)$ $\cup_n\  \partial\Sigma_n=\cup \mathscr A$. In this case, we say $\mathscr A$ \emph{decomposes $\Sigma$ into bordered sub-surfaces}, where \emph{complementary components} are  $\{\Sigma_n\}$. Also, we call each component of $\mathscr A$ a \emph{decomposition circle}.} \label{deompositiondefinition}
\end{definition}

The following theorem asserts that any non-compact surface other than the plane has a decomposition, where each complementary part is either a pair of pants, a one-holed torus, or a punctured disk. This way of decomposition of the co-domain of a pseudo proper homotopy equivalence will be used in all cases.

\begin{theorem}
\textup{Let $\Sigma$ be a non-compact surface not homeomorphic to $\Bbb R^2$. Then there is an \textup{LFCS} $\mathscr C$ on $\Sigma$ such that $\mathscr C$ decomposes $\Sigma$ into bordered sub-surfaces, and a complementary component of this decomposition is homeomorphic to either $S_{1,1}$ (used at most once), $S_{0,3}$, or $S_{0,1,1}$.} \label{completedecomposition}
\end{theorem}

\begin{proof}Enough to find a collection $\{\Sigma_n\}$ of bordered sub-surfaces of $\Sigma$ with four properties, as mentioned in \Cref{deompositiondefinition}, so that each $\Sigma_n$ is homeomorphic to either $S_{0,3}$, $S_{1,1}$, or $S_{0,1,1}$. For that, consider an inductive construction of  $\Sigma$; see \Cref{incudctiveconstruction}. Now, a finite sequence of annuli, when added to the compact bordered surface used just before it, can be ignored. Thus, we may assume $S_{0,3}$ or $S_{1,2}$ is used after $S_{0,1}$ without loss of generality because of $\Sigma\not\cong \Bbb R^2$, and hence pushing $S_{0,1}$ into $\text{int}(S_{0,3})$ or $\text{int}(S_{1,2})$, we end up with $S_{0,2}$ (which can be ignored) or $S_{1,1}$. Now, the proof will be completed by observing the following:  $S_{1,2}$ can be decomposed into two copies of $S_{0,3}$, and $S_{0,1,1}$ is the union (with pairwise disjoint interiors) of countably many copies of $S_{0,2}$. \end{proof}

\begin{remark}
\textup{A statement closely related to \Cref{completedecomposition} is in \cite[Theorem 1.1.]{MR2025333}, which says that \emph{``every surface except for the sphere, the plane, and the torus is the union (with pairwise disjoint interiors) of copies of the pair of pants and copies of the punctured disk''}. But due to the part (4) of \Cref{deompositiondefinition}, if we want that any complementary component is homeomorphic to only either $S_{0,3}$ or $S_{0,1,1}$, then $\Sigma\not\cong S_{1,0,1}$ also needs to consider; see \Cref{oneend} and  \Cref{completedecomposition1} below.}
\end{remark}

\begin{theorem}
\textup{Let $\Sigma$ be a non-compact surface that is not homeomorphic to either $\Bbb R^2$ or $S_{1,0,1}$. Then there is an \textup{LFCS} $\mathscr C'$ on $\Sigma$ such that $\mathscr C'$ decomposes $\Sigma$ into bordered sub-surfaces, and a complementary component of this decomposition is homeomorphic to either  $S_{0,3}$ or $S_{0,1,1}$.} \label{completedecomposition1}
\end{theorem}

\begin{figure}[ht]$$\adjustbox{trim={0.0\width} {0.0\height} {0.0\width} {0.0\height},clip}{\def\svgwidth{\linewidth}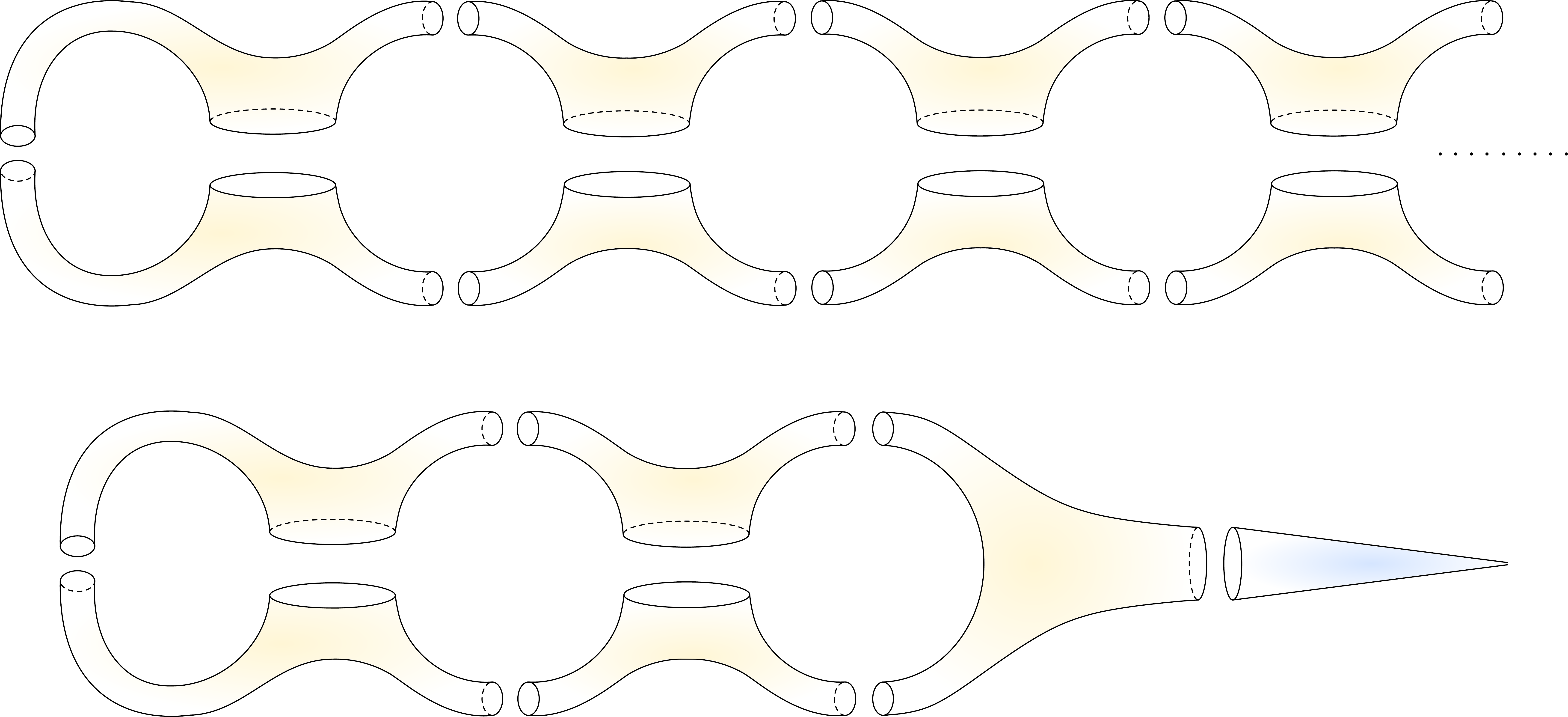}$$\caption{On the top: Decomposition of Loch Ness Monster into countably infinitely many copies of the pair of pants.  At the bottom: Decomposition of $S_{3,0,1}$ into five copies of pair of pants and a copy of the punctured disk. }\label{oneend}\end{figure}

\begin{proof}
Enough to find a collection $\{\Sigma_n\}$ of bordered sub-surfaces of $\Sigma$ with four properties, as mentioned in \Cref{deompositiondefinition}, so that each $\Sigma_n$ is homeomorphic to either $S_{0,3}$ or $S_{0,1,1}$. For that, consider an inductive construction of  $\Sigma$; see \Cref{incudctiveconstruction}. We will divide the whole proof into two cases, depending on whether $\Sigma$ has at least two ends.

At first, suppose the number of ends of $\Sigma$ is at least two. Now, the definition of the space of ends tells us that we need to use at least one pair of pants in the inductive construction of $\Sigma$. By \Cref{interchange}, we may assume that in this inductive construction, a pair of pants is used just after the disk.  Now, an argument similar to before (see the proof of \Cref{completedecomposition}) concludes this case.

Next, consider the case when the number of ends of $\Sigma$ is precisely one. That is, $\Sigma$ can be either Loch Ness Monster (the infinite genus surface with one end) or $S_{g,0,1}$ with $g\geq 2$. Loch Ness Monster decomposes into countably infinitely many copies of the pair of pants, and $S_{g,0,1}$ with $g\geq 2$ decomposes into $2g-1$ many copies of the pair of pants and a copy of the punctured disk. See \Cref{oneend}.\end{proof}

To prove \Cref{completedecomposition1}, we used \Cref{interchange} below, which says that in an inductive construction of a non-compact surface, interchanging the positions of the compact bordered surfaces used in the first few inductive steps doesn't change the homeomorphism type, and its proof is based on the observation that the portions outside compact subsets determine the space of ends.

\begin{lemma}
\textup{Let $\Sigma$ be a non-compact surface with some inductive construction $\mathscr I$. Denote the compact bordered subsurface of $\Sigma$ after the $i$-th step of $\mathscr I$ by $K_i$. Suppose $\{\mathcal B_1, ..., \mathcal B_n:$ each $\mathcal B_\ell$ is homeomorphic to either $S_{0,2}$, $S_{0,3}$, or $S_{1,2}\}$ is a finite collection of compact bordered surfaces such that $\mathcal B_\ell$ is used to construct $K_{i_\ell+1}$ from $K_{i_\ell}$ for each $\ell=1,...,n$. Then there exists a non-compact surface $\Sigma'$ with an inductive construction $\mathscr I'$ such that $\Sigma'\cong \Sigma$ and  $\mathcal B_\ell$ is used to construct $K_{\ell+1}'$ from $K_{\ell}'$ for each $\ell=1,...,n$; where $K_i'$ denotes the compact bordered subsurface of $\Sigma$ after the $i$-th step of $\mathscr I'$.} \label{interchange}
\end{lemma}

\begin{figure}[ht]$$\adjustbox{trim={0.0\width} {0.0\height} {0.0\width} {0.14\height},clip}{\def\svgwidth{\linewidth}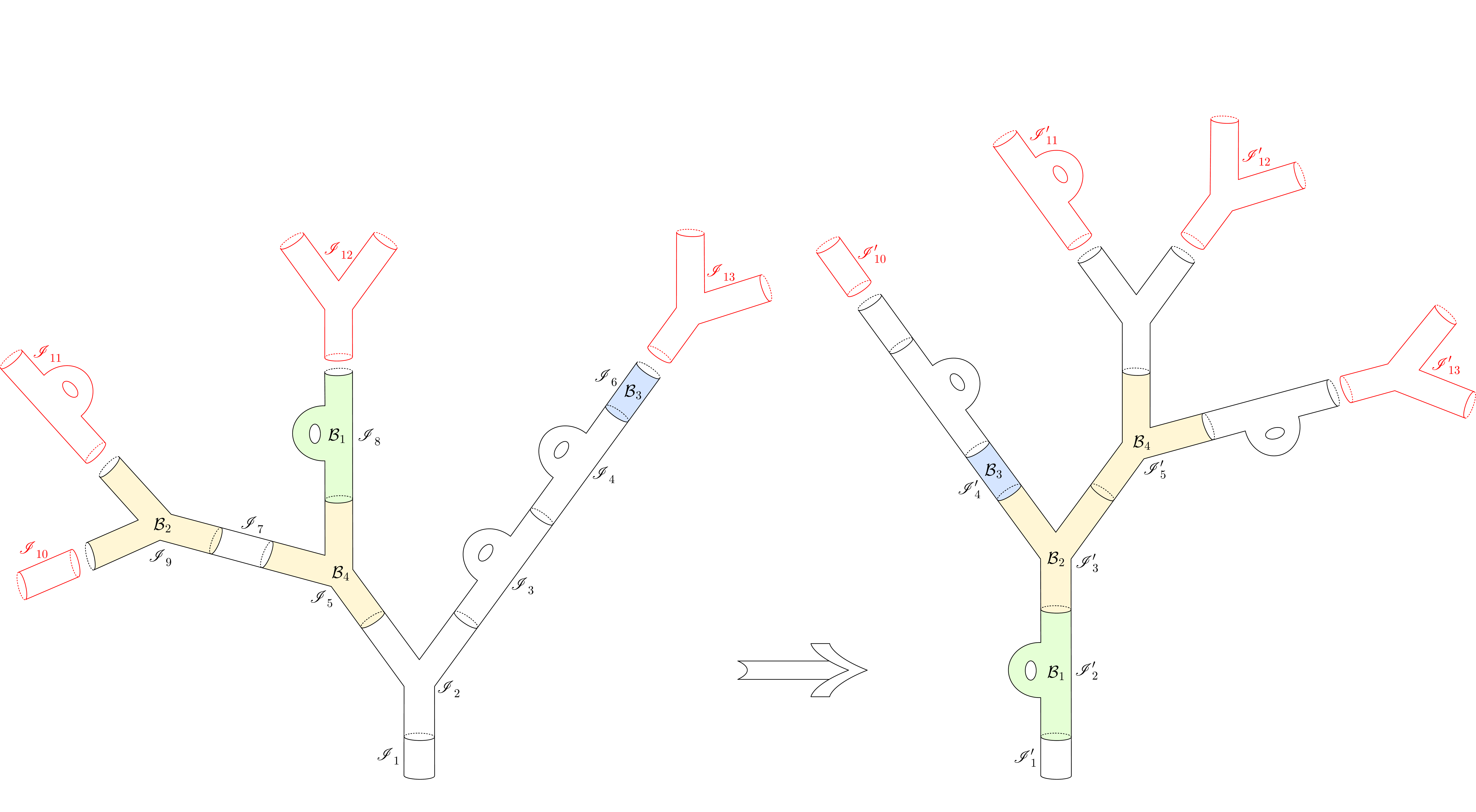}$$\caption{$\mathscr I_r$ (resp. $\mathscr I_r')$ denotes the $r$-th step of $\mathscr I$ (resp. $\mathscr I'$). Notice that here, $n_0=9$ and $n=4$. Also, the red-colored compact bordered surfaces are the portions of $\mathbf S$, and the inductive construction of $\mathbf S$ given by $\mathscr I'$ is inherited from the inductive construction of $\mathbf S$ given by $\mathscr I$.}\label{arrange}\end{figure}

\begin{proof}
Let $n_0$ be a positive integer such that $K_{n_0}$ contains each $\mathcal B_\ell$. Define $\mathbf S\coloneqq \Sigma\setminus\text{int}(K_{n_0})$. Thus $\mathbf S$ is a bordered subsurface of $\Sigma$ with $\partial \mathbf S=\partial K_{n_0}$. Now, consider all copies of different building blocks used up to the $n_0$-th step of $\mathscr I$, and inside $K_{n_0}$ interchange them so that $\mathcal B_1,..., \mathcal B_n$ comes just after the initial disk $K_1$ one by one following the increasing order of their indices. Denote the resultant of this interchange process by $K_{n_0}'$. So $ K_{n_0}\cong  K_{n_0}'$ as $g(K_{n_0})=g(K_{n_0}')$ and $\partial K_{n_0}\cong \partial K_{n_0}'$. Define a non-compact surface $\Sigma'$ as $\Sigma'\coloneqq K_{n_0}'\cup_{\partial K_{n_0}'\equiv\partial \mathbf S} \mathbf S$. Therefore, $\Sigma\setminus K_{n_0}= \text{int}(\mathbf S)=\Sigma'\setminus K_{n_0}'$ (notice that we are thinking $\mathbf S$ as a subset of $\Sigma'$ using the obvious embedding $\mathbf S\hookrightarrow \Sigma'$).

Choose an inductive construction $\mathscr I'_{\leq n_0}$ of $K_{n_0}'$ such that $i$-th element of the ordered sequence $K_1, \mathcal B_1,..., \mathcal B_\ell$ is used in the $i$-th step of $\mathscr I'_{\leq n_0}$. Also, $\mathscr I$ gives a truncated inductive construction $\mathscr I\vert \mathbf S$ on $\mathbf S$ starting from the $(n_0+1)$-th step. Now, $\mathscr I'_{\leq n_0}$ followed by $\mathscr I\vert \mathbf S$ together gives an inductive construction $\mathscr I'$ of $\Sigma'$. Roughly it means $\mathscr I'$ is the same as the inductive construction of $\Sigma$, except for the first few steps.  Denote the compact bordered subsurface of $\Sigma'$ after the $i$-th step of $\mathscr I'$ by $K_i'$. To complete the proof, we show $\Sigma'\cong \Sigma$ using \Cref{richard2}.

Consider the efficient exhaustion $\{K_i\}$ \big(resp. $\{K_i'\}$\big) of $\Sigma$ \big(resp. $\Sigma'$\big) by compacta to define $\text{Ends}(\Sigma)$ \big(resp. $\text{Ends}(\Sigma')$\big). Recall that the space of ends remains the same up to homeomorphism even if we choose a different efficient exhaustion by compacta; see \Cref{ends}. By $\Sigma\setminus K_{n_0}= \text{int}(\mathbf S)= \Sigma'\setminus K_{n_0}'$, for every sequence $(V_1, V_2,...)\in \text{Ends}(\Sigma)$, there exists a unique sequence $(V_1', V_2',...)\in \text{Ends}(\Sigma')$ such that $V_i= V_i'$ for all integer $i\geq n_0$, and conversely. Thus, there exists a homeomorphism $\varphi\colon \text{Ends}(\Sigma)\to \text{Ends}(\Sigma')$ with $\varphi\big( \text{Ends}_\text{np}(\Sigma)\big)= \text{Ends}_\text{np}(\Sigma')$. Also, $\Sigma\setminus K_{n_0}= \text{int}(\mathbf S)= \Sigma'\setminus K_{n_0}'$ and $ K_{n_0}\cong  K_{n_0}'$ together imply $g(\Sigma)=g(\Sigma')$. Therefore,  $\Sigma'\cong \Sigma$ by \Cref{richard2}.
\end{proof}

The spine construction of Goldman's inductive procedure shows that every non-compact surface $\Sigma$ ($\Sigma$ may be of infinite-type) is the interior of a bordered surface: consider the graph $\mathcal G$ consisting of blue straight line segments and red circles, as given in \Cref{figureofinductiveconstruction}. Any thickening \cite[Definition 7.2.]{MR275436} of $\mathcal G$ in $\Sigma$ is the interior of a bordered subsurface $\mathbf S$ of $\Sigma$. Now, \cite[Corollary 7.2. and section 7.3.]{MR275436} says that $\text{int}(\mathbf S)\cong \Sigma$.  When $\Sigma$ is of finite-type, we prove the same thing differently in the following theorem.

\begin{theorem}
\textup{A non-compact finite-type surface is the interior of a compact bordered surface. In particular, if a non-compact surface has infinite cyclic (resp. trivial) fundamental group, then it is homeomorphic to $\Bbb S^1\times \Bbb R$ (resp. $\Bbb R^2$).}\label{finitetypenoncompactsurace}
\end{theorem}
\begin{proof}
Let $\Sigma$ be a finite-type non-compact surface. Consider an inductive construction of  $\Sigma$; see \Cref{incudctiveconstruction}. Since $\pi_1(\Sigma)$ is finitely generated, \Cref{spine} says that $\Sigma$ is homotopy equivalent to $\bigvee^{2r+s}\Bbb S^1$, where in this inductive construction, $r\in \Bbb N$ is the total number of copies of $S_{1,2}$, and $s\in \Bbb N$ is the total number of copies of $S_{0,3}$; see \Cref{figureofinductiveconstruction}. Thus there is an integer $n$ such that $\overline{\Sigma\setminus K_n}$ (where $K_n$ is the compact bordered subsurface of $\Sigma$ after $n$-th inductive step) is a finite collection of punctured disks. Now, $g(\Sigma)=g\big(\text{int}(K_n)\big)$. Also, each end of $\Sigma$ \big(resp. $\text{int}(K_n)$\big) is planar, and the total number of ends of $\Sigma$ \big(resp. $\text{int}(K_n)$\big) is the same as the number of components of $\partial K_n$. By \Cref{richard2}, $\Sigma\cong \text{int}(K_n)$.

If $\Sigma$ is a non-compact surface with the infinite-cyclic fundamental group, then any inductive construction of $\Sigma$ contains no copy of $S_{1,2}$ but precisely one copy of $S_{0,3}$, i.e., $\Sigma\cong \Bbb S^1\times \Bbb R$. Similarly, if $\Sigma$ is a non-compact surface with the trivial fundamental group, then any inductive construction of $\Sigma$ has no copy of $S_{0,3}$ as well as no copy of $S_{1,2}$, i.e., $\Sigma\cong \Bbb R^2$.
\end{proof}

The proposition below follows directly from Goldman's inductive construction, so we quote it without proof. It says that an infinite-type surface has a finite genus only if it has infinitely many ends. On the other hand, \Cref{richard1} guarantees the existence of an infinite-type surface of the infinite genus with infinitely many ends. 
\begin{proposition}\textup{A non-compact surface is of a finite genus if and only if the total number of copies of $S_{1,2}$ used in any inductive construction of $\Sigma$ is finite. Thus, if an infinite-type surface has a finite genus, then it must have infinitely many ends.}
\label{infinitetypefinitegenusmeansinfinitelymanyends}
\end{proposition}
This section's final fact (as promised in the introduction) says that the fundamental group alone can't detect the homeomorphism type of an infinite-type surface.

\begin{proposition}
\textup{Up to homotopy equivalence, there is exactly one infinite-type surface, but up to homeomorphism, there are $2^{\aleph_0}$ many infinite-type surfaces.}\label{upto}
\end{proposition}

\begin{proof}
Any infinite-type surface is homotopy equivalent to the wedge of countably infinitely many circles; see \Cref{spine}. In particular, for an infinite-types surface, the fundamental group is a free group of countably infinite rank, and  each higher homotopy group is trivial. Now, every surface has a $C^\infty$-smooth structure, hence has a CW-complex structure. Therefore, every infinite-type surface $\Sigma$ is a CW-complex $K\big(\pi_1(\Sigma), 1\big)$. Thus, any two infinite-type surfaces are homotopy equivalent; see \cite[Theorem 1B.8.]{MR1867354}.

Now, we prove that up to homeomorphism, there are $2^{\aleph_0}$ many infinite-type surfaces. Notice that except for the first step, in each step of Goldman's inductive procedure, we use either $S_{0,3}$, or $S_{0,2}$, or $S_{1,2}$. Thus we have at most $3^{\aleph_0}=2^{\aleph_0}$ many non-compact surfaces, up to homeomorphism. Therefore, it is enough to show that this upper bound is reachable. Let $\tau$ be a non-empty closed subset of the Cantor set. By \Cref{richard1}, there exists an infinite genus surface $\Sigma_\tau$ such that $\text{Ends}(\Sigma_\tau)=\text{Ends}_\text{np}(\Sigma_\tau)\cong\tau$. Therefore, if $\tau_1$ and $\tau_2$ are two non-homeomorphic non-empty closed subsets of the Cantor set, then  $\Sigma_{\tau_1}$ is not homeomorphic to $\Sigma_{\tau_2}$ by \Cref{richard2}. Now, \cite[Theorem 2]{MR133103} says that up to homeomorphism, there are $2^{\aleph_0}$ many closed subsets of the Cantor set. So, we are done.\end{proof}

\subsection{Transversality of a proper map with respect to all decomposition circles}
In the previous section, the co-domain of a pseudo proper homotopy equivalence has been decomposed into finite-type bordered surfaces by a locally finite pairwise disjoint collection of circles. This section aims to properly homotope the pseudo proper homotopy equivalence to make it transverse to each decomposition circle.

The theorem below follows from the theory developed in the \Cref{appendix}.  We aim to use it to impose a one-dimensional submanifold structure on the inverse image of each decomposition circle.
\begin{theorem}
\textup{Let $f\colon \Sigma'\to \Sigma$ be a proper map between non-compact surfaces, and let $\mathscr A$ be an \textup{LFCS} on $\Sigma$. Then $f$ can be properly homotoped to make it smooth as well as transverse to the manifold $\mathscr A$.} \label{transtoLFCS} 
\end{theorem}
\begin{proof}
Using \Cref{whiteny}, after a proper homotopy, we may assume that $f$ is a smooth proper map. After that, using \Cref{transhomotopy}, properly homotope $f$ so that it becomes transverse to $\mathscr A$.
\end{proof}

\begin{remark}
\textup{Note that in \Cref{transtoLFCS}, we have no control over those proper homotopies, which make the proper map $f$ as smooth as well as transversal to $\mathscr A$, i.e., after these proper homotopies, $f^{-1}(\mathscr A)$ can be empty, even if these proper homotopies start with a surjective proper map. A remedy for this is: Assume $\deg(f)\neq 0$; this is because the degree is invariant under proper homotopy, and a map of non-zero degree is surjective; see \Cref{non-surjectivepropermaphasdegreezero} and \Cref{properhomotopypreservessurjectivity}.
If $f$ is a proper homotopy equivalence, then $f$ has a proper homotopy inverse; hence, $\deg(f)\neq 0$ (see \Cref{degreeofapropermap}). But if $f$ is a pseudo proper homotopy equivalence, then we don't know (at least till this stage) whether $f$ has a proper homotopy inverse or not (though it has a homotopy inverse). Later in \Cref{degreeofapseudoproperhomotopyequivalence}, using $\pi_1$-bijectivity, we will show that most pseudo proper homotopy equivalence is a map of  degree $\pm 1$.}
\end{remark}
The following theorem says that the transversal pre-image of an LFCS under a proper map is an LFCS.
\begin{theorem}
\textup{Let $f\colon \Sigma'\to \Sigma$ be a smooth proper map between two non-compact surfaces, and let $\mathscr A$ be an \textup{LFCS} on $\Sigma$ such that $f\ \stackinset{c}{}{c}{0.1ex}{$\top$}{$\abxpitchfork$}\ \mathscr A$. Then for each component $\mathcal C$ of $\mathscr A$, either $f^{-1}(\mathcal C)$ is empty or a pairwise disjoint finite collection of smoothly embedded circles on $\Sigma'$. Therefore, $f^{-1}(\mathscr A)$ is an \textup{LFCS} on $\Sigma'$.}\label{remarktranshomotopy}
\end{theorem}

\begin{proof}
By the definition of transversality,  $f\ \stackinset{c}{}{c}{0.1ex}{$\top$}{$\abxpitchfork$}\ \mathscr A$ implies $f\ \stackinset{c}{}{c}{0.1ex}{$\top$}{$\abxpitchfork$}\ \mathcal C$ for each component $\mathcal C$ of $\mathscr A$. Thus $f^{-1}(\mathcal C)$ is either empty or is a compact (since $f$ is proper) one-dimensional boundaryless smoothly embedded submanifold of $\Sigma'$. Now, by classification of closed one-dimensional manifolds, we complete the first part.

Next, if possible, let $K'$ be a compact subset of $\Sigma'$ such that $K'$ intersects infinitely many components of $f^{-1}(\mathscr A)$. By the first part, it means the compact set $f(K')$ intersects infinitely many components of $\mathscr A$, which contradicts the fact that $\mathscr A$ is a locally finite collection.
\end{proof}

\subsection{Disk removal}Previously, as observed, after a proper homotopy, the number of components in the collection of transversal pre-images of all decomposition circles can be infinite, and many components (possibly infinitely many) of this collection, maybe trivial circles. Here, at first, our goal is to group all these trivial circles in terms of the size of the disk bounded by a trivial circle and then remove all groups of trivial circles simultaneously by a proper homotopy.

At first, our intended grouping requires a technical lemma, which asserts that on a non-simply connected surface, an LFCS consisting of concentric trivial circles doesn't exist. Roughly it means,  on a non-simply connected surface, arbitrarily large disks bounded by components of an LFCS don't exist.

\begin{lemma}
\textup{Let $\Sigma$ be a surface, and let $\mathscr A\coloneqq\{\mathcal C_i:i\in \Bbb N\}$ be an \textup{LFCS} on $\Sigma$ such that for each $i$ the circle $\mathcal C_i$ bounds a disk $\mathcal D_i\subset \Sigma$ with $ \mathcal C_i\subset \textup{int}(\mathcal D_{i+1})$. Then $\Sigma$ is homeomorphic to $\Bbb R^2$.}\label{charofplane}
\end{lemma}
\begin{proof}
At first, notice that $\Sigma$ must be non-compact as $\mathscr A$ is a locally finite, pairwise disjoint, infinite collection of circles. Using inductive construction (see \Cref{incudctiveconstruction}), we have a sequence $\{\mathbf S_j:j\in \Bbb N\}$ of compact bordered sub-surfaces of $\Sigma$ such that $\cup_j\mathbf S_j=\Sigma$ and for each $j\in \Bbb N$, $\mathbf S_j\subset\text{int}(\mathbf S_{j+1})$. Consider any $p\in \Sigma$. So, a $j_0\in \Bbb N$ exists such that $p\in \mathbf S_{j_0}$ and $\mathbf S_{j_0}\cap\big(\cup_i\mathcal C_i\big)\neq\varnothing$. Since $\mathscr A$ is a locally finite collection, only finitely many components of $\mathscr A$ intersect the compact set $\mathbf S_{j_0}$. Let $\mathcal C_{i_1},...,\mathcal C_{i_\ell}$ be the only components of $\mathscr A$ intersecting $\mathbf S_{j_0}$, where $i_1<\cdots<i_\ell$. Pick an integer $i_0>i_\ell$. Then $\mathcal C_{i_0}\cap \mathbf S_{j_0}=\varnothing.$ Now, since $\mathcal C_{i_\ell}\subset\text{int}(\mathcal D_{i_0})$, $\mathbf S_{j_0}$ is connected, and $\Sigma$ is locally Euclidean, we can say that $\mathbf S_{j_0}\subseteq\text{int}(\mathcal D_{i_0})$. Thus, every point $x\in\Sigma$ has an open neighborhood $\mathcal U_x$ in $\Sigma$ such that $\mathcal U_x\subseteq\mathcal D_i$ for some $i\in \Bbb N$. Therefore, every loop on $\Sigma$ is contained in a disk $\mathcal D_i$ for some large $i\in \Bbb N$, i.e., $\Sigma$ is simply-connected. By \Cref{finitetypenoncompactsurace}, $\Sigma\cong \Bbb R^2$. 
\end{proof}

The following lemma is the primary tool for showing that a homotopy is proper. It tells how a proper map can be properly homotoped so that it changes on infinitely many pairwise disjoint compact sets.

\begin{lemma}
\textup{Let $f\colon \Sigma'\to \Sigma$ be a proper map between two non-compact surfaces, and let $\{\Sigma_n':n\in \Bbb N\}$ be a pairwise disjoint  collection of compact bordered sub-surfaces of $\Sigma'$. For each $n\in\Bbb N$, suppose $H_n\colon \Sigma_n'\times [0,1]\to \Sigma$ is a homotopy relative to $\partial \Sigma_n'$ such that $H_n(-,0)=f\vert\Sigma_n'$ and $\textup{im}(H_n)\to \infty$. Then $\mathcal H\colon \Sigma'\times [0,1]\to \Sigma$ defined by$$\mathcal H(p,t)\coloneqq\begin{cases} H_n(p,t)&\text{ if }p\in \Sigma_n'\text{ and }t\in [0,1],\\ f(p)&\text{ if } p\in \Sigma'\setminus\big(\cup_{n\in\Bbb N}  \Sigma_n'\big)\text{ and }t\in [0,1]\end{cases}$$ is a proper map.}\label{properhomotopy}
\end{lemma}

\begin{proof}
 Let $\mathcal K$ be a compact subset of $\Sigma$. By continuity of $\mathcal H$, $\mathcal H^{-1}(\mathcal K)$ is closed in $\Sigma'$. Since $\textup{im}(H_n)\to \infty$, there exists $n_0\in \Bbb N$ such that $\textup{im}(H_\ell)\cap \mathcal K=\varnothing$ for all integers $\ell\geq n_0+1$. Now, $f^{-1}(\mathcal K)$ is compact as $f$ is proper. Also, the domain of each $H_n$ is compact. Hence, the closed subset $\mathcal H^{-1}(\mathcal K)$ of $\Sigma'$ is contained in the compact set $f^{-1}(\mathcal K)\cup \bigcup_{\ell=1}^{n_0} H_n^{-1}(\mathcal K)$. So, we are done.
\end{proof}

To remove trivial components from the transversal pre-image of an LFCS with infinitely many components, we need to impose some conditions on this LFCS. One such preferred LFCS is given in \Cref{completedecomposition}. But for future use, not only this type of LFCS, we require other kinds of LFCS on the co-domain. So, here is the list of different preferred LFCS.

\begin{definition}\label{preferredLFCS} \textup{Let $\Sigma$ be a non-compact surface such that $\Sigma\not\cong \Bbb R^2$. Suppose,  $\mathscr A$ is a given \textup{LFCS} on $\Sigma$. We say $\mathscr A$ is a \emph{preferred \textup{LFCS}} on $\Sigma$ if either of the following happens: $(i)$ $\mathscr A$ is a finite collection of primitive circles on $\Sigma$; $(ii)$ $\mathscr A$ decomposes $\Sigma$ into bordered sub-surfaces, and a complementary component of this decomposition is homeomorphic to either $S_{1,1}$, $S_{0,3}$, $S_{0,2}$, or $S_{0,1,1}$.}\end{definition} 

\begin{remark}
\textup{The only use of case $(i)$ of \Cref{preferredLFCS} is in \Cref{degreeofapseudoproperhomotopyequivalence}, where we consider the process of removing unnecessary circles from the transversal pre-image of the boundary of an essential pair of pants or an essential punctured disk. It is worth noting that by a finite LFCS, one can't decompose an infinite-type surface into finite-type bordered surfaces.}
\end{remark}

In the theorem below, we construct a proper homotopy, which removes all trivial components keeping a neighborhood of each primitive component stationary from the transversal pre-image of a preferred LFCS. Recall that a homotopy $H\colon X\times [0,1]\to Y$ is said to be \emph{stationary on a subset $A$ of $X$} if $H(a,t)=H(a,0)$ for all $(a,t)\in A\times [0,1]$.

\begin{theorem}\textup{Let $f\colon \Sigma'\to \Sigma$ be a smooth proper map between two non-compact surfaces, where $\Sigma'\not\cong \Bbb R^2\not \cong \Sigma$; and let $\mathscr A$ be a preferred \textup{LFCS} on $\Sigma$ such that $f\ \stackinset{c}{}{c}{0.1ex}{$\top$}{$\abxpitchfork$}\ \mathscr A$. Then we can properly homotope $f$ to remove all trivial components of the \textup{LFCS} $f^{-1}(\mathscr A)$ such that each primitive component of $f^{-1}(\mathscr A)$ has an open neighborhood on which this proper homotopy is stationary.} \label{diskremoval}\end{theorem}

\begin{proof}
Since $\Sigma'\not \cong \Bbb R^2$ and $f^{-1}(\mathscr A)$ is an \textup{LFCS} (see \Cref{remarktranshomotopy}), by \Cref{charofplane}, there don't exist infinitely many components $\mathcal C_1',\mathcal C_2',...$ of $f^{-1}(\mathscr A)$ bounding the disks $\mathcal D_1',\mathcal D_2',...$, respectively such that $\mathcal C_n'$ is contained in the interior of $\mathcal D_{n+1}'$ for each $n$. Thus, if $f^{-1}(\mathscr A)$ has a trivial component, we can introduce the notion of an outermost disk bounded by a component of $f^{-1}(\mathscr A)$ in the following way: A disk $\mathcal D'\subset \Sigma'$ bounded by a component of $f^{-1}(\mathscr A)$ is called an outermost disk; if given another disk $\mathcal D''\subset\Sigma$ bounded by a component of $f^{-1}(\mathscr A)$, then either $\mathcal D''\subseteq \mathcal D'$ or $\mathcal D'\cap \mathcal D''=\varnothing$.

Let $\{\mathcal D_{n}'\}$ be the pairwise disjoint collection (which may be an infinite collection) of all outermost disks. Assume $\mathcal C_n$ represents that component of $\mathscr A$ for which $f(\partial \mathcal D_{n}')\subseteq \mathcal C_n$. Note $\mathcal C_n$ may equal to $\mathcal C_m$ even if $m\neq n$.

Now, for each integer $n$, we will construct a compact bordered subsurface $\mathcal Z_n$ with $f(\mathcal D_n')\subseteq \mathcal Z_n$ such that $\mathcal Z_n\to \infty$. Roughly, $\mathcal Z_n$ will be obtained from taking the union of all those complementary components of $\Sigma$ (if a punctured disk appears, truncate it), which are hit by $f(\mathcal D_n')$.

Fix an integer $n$. Let $\mathcal X_{n,1}',...,\mathcal X_{n, k_n}'$ be the all connected components of $\mathcal D_n'\setminus f^{-1}(\mathscr A)$. By continuity of $f$, for each $\mathcal X_{n,\ell}'$, there exists a complementary component $\mathcal Y_{n,\ell}$ of $\Sigma$ decomposed by $\mathscr A$ such that $f(\mathcal X_{n,\ell}')\subseteq \mathcal Y_{n,\ell}$ and $\partial \mathcal X_{n,\ell}'\subseteq f^{-1}(\partial \mathcal Y_{n,\ell})$. See \Cref{figureforoutermostdisk}. For each $\ell$, define a compact bordered subsurface $\mathcal Z_{n,\ell}$ of $\Sigma$ as follows: If $\mathcal Y_{n,\ell}$ is homeomorphic to either $S_{1,1}$, $S_{0,3}$, or $S_{0,2}$; define $\mathcal Z_{n,\ell}\coloneqq \mathcal Y_{n,\ell}$. On the other hand, if $\mathcal Y_{n,\ell}$ is homeomorphic to $S_{0,1,1}$, then let $\mathcal Z_{n,\ell}$ be an annulus in $\mathcal Y_{n,\ell}$ such that $\partial \mathcal Z_{n,\ell}\cap \partial \mathcal Y_{n,\ell}=\partial \mathcal Y_{n,\ell}$ and  $f\left(\overline{\mathcal X_{n,\ell}'}\right)\subseteq \mathcal Z_{n,\ell}$. Finally, define $\mathcal Z_n\coloneqq \mathcal Z_{n,1}\cup \cdots\cup \mathcal Z_{n,k_n}$. 

  Now, we show $\mathcal Z_n\to \infty$. So, consider a compact subset $\mathcal K$ of $\Sigma$. Let $\mathbf S_1,..., \mathbf S_m$ be a collection of complementary components of $\Sigma$ decomposed by $\mathscr A$ such that $\mathcal K\subseteq \text{int}\left(\bigcup_{\ell=1}^m\mathbf S_\ell\right)$. Define $\mathbf S\coloneqq\bigcup_{\ell=1}^m\mathbf S_\ell$.   Thus, for an integer $n$, $f(\mathcal D_n')\cap  \mathbf S\neq \varnothing$ if and only if $\mathcal D_n'$ contains at least one component of $\bigcup_{\ell=1}^{m}f^{-1}\left(\partial \mathbf S_\ell\right)$. This is due to the construction of each $\mathcal Z_n$; see \Cref{figureforoutermostdisk}. For each component $\mathcal C$ of $\mathscr A$, \Cref{remarktranshomotopy} tells that $f^{-1}(\mathcal C)$ has only finitely many components.  So $\mathcal D_n'$ doesn't contain any component of $\bigcup_{\ell=1}^{m}f^{-1}\left(\partial \mathbf S_\ell\right)$ for all sufficiently large $n$, i.e.,  $f(\mathcal D_n')\cap \mathbf S=\varnothing$ for all sufficiently large $n$. Since $\mathcal K\subseteq\text{int}(\mathbf S)$ and each $\mathcal Z_n$ is obtained from taking the union of all those complementary components of $\Sigma$ (if a punctured disk appears, truncate it), which are hit by $f(\mathcal D_n')$, we can say that $\mathcal Z_n\cap \mathcal K= \varnothing$ for all sufficiently large $n$. Therefore, $\mathcal Z_n\to \infty$ as $\mathcal K$ is an arbitrary compact subset of $\Sigma$.

  \begin{figure}[ht]$$\adjustbox{trim={0.0\width} {0.0\height} {0.0\width} {0.0\height},clip}{\def\svgwidth{1.23\linewidth}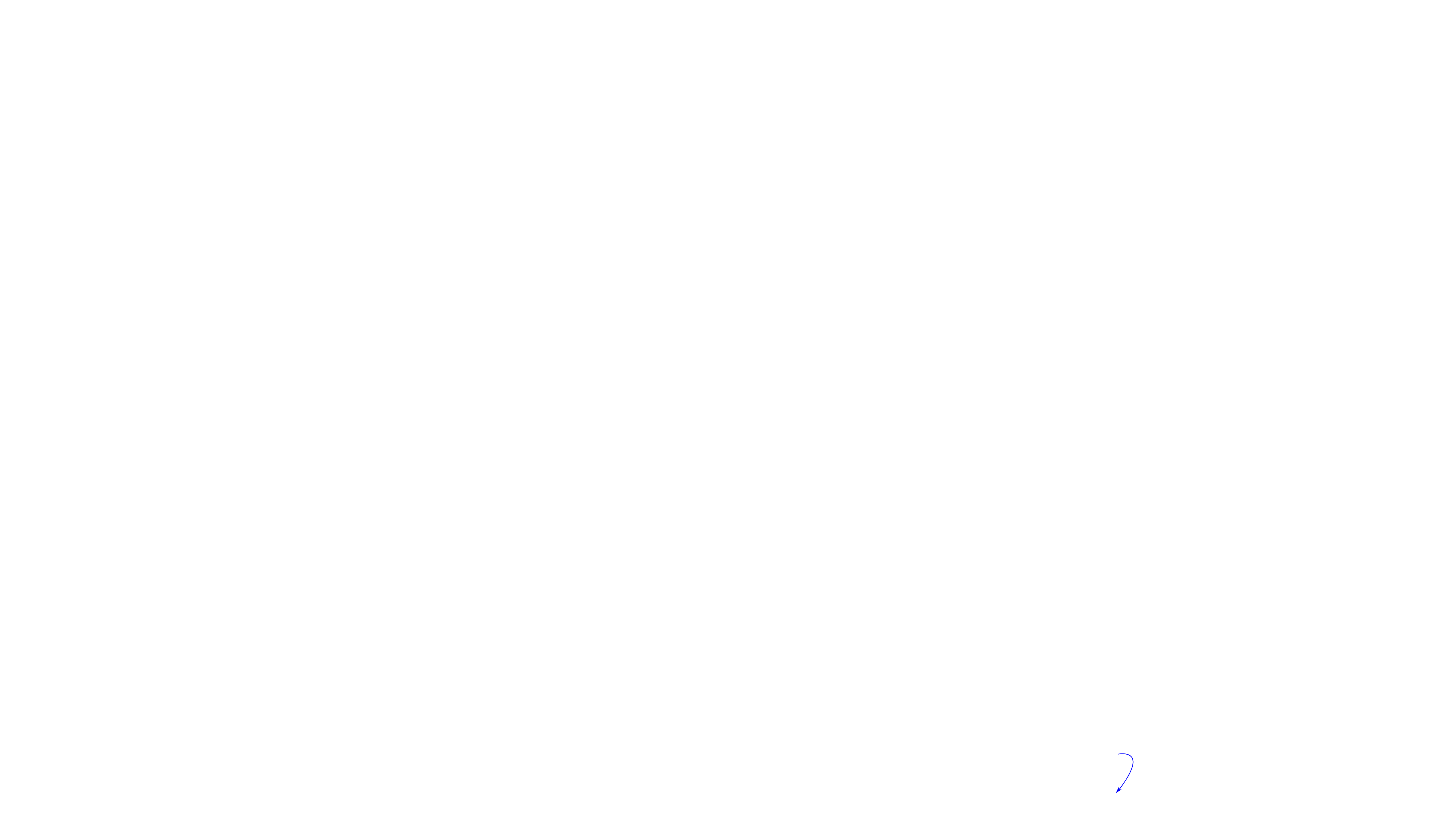}$$\caption{Each component of $\mathcal D_n'\setminus f^{-1}(\mathscr A)$ maps into a component of $\Sigma\setminus \mathscr A$. This fact, together with \Cref{thick}, provides $\Sigma_n$. A black circle denotes a component of either $\mathscr A$ or a component of $f^{-1}(\mathscr A)$. }\label{figureforoutermostdisk}\end{figure}

For each $n$, adding a small external collar to one of the boundary components of $\mathcal Z_n$ (if needed), we can construct a compact bordered surface $\Sigma_n$ with $\mathcal C_n\subseteq\text{int}(\Sigma_n),\   f(\mathcal D_{n}')\subseteq\Sigma_n$ such that $\{\Sigma_n\}$ is a locally finite collection, i.e., $\Sigma_n\to \infty$. See \Cref{figureforoutermostdisk}.

For each $n$, write $\mathcal C_n' \coloneqq \partial \mathcal D_n'$. Thus $f(\mathcal C_n')\subseteq \mathcal C_n$. Since $\mathcal C_n\subseteq \text{int}(\Sigma_n)$, using \Cref{trans1}, choose a one-sided tubular neighborhood $\mathcal C_n\times [0, \varepsilon_n]$ of $\mathcal C_n$ in $\Sigma$ with $\mathcal C_n\times 0\equiv \mathcal C_n$ such that  $f\ \stackinset{c}{}{c}{0.1ex}{$\top$}{$\abxpitchfork$}\ (\mathcal C_{n}\times t_n)$ for each $t_n\in [0, \varepsilon_n]$ and $\mathcal C_n\times [0,\varepsilon_n]\subseteq \Sigma_n$. Without loss of generality, we may further assume that  $f(x')\in \mathcal C_n\times [0,\varepsilon_n]$ for each $x'\in \Sigma'\setminus \mathcal D_n'$ sufficiently near to $\mathcal C_n'$.  Next, since $f^{-1}(\mathscr A)$ is an $\textup{LFCS}$, for each $n$, \Cref{technicallemma} gives a one-sided compact tubular neighborhood $\mathcal U_n'$ of $\mathcal C_n'$ such that the following hold: $\mathcal U_n'\cap \mathcal D_n'=\mathcal C_n'=\mathcal U_n'\cap f^{-1}(\mathscr A)$, $f(\mathcal U_n')\subseteq \mathcal C_n\times [0, \varepsilon_n]$ for each $n$; and $\mathcal U_n'\cap \mathcal U_m'=\varnothing$ for $m\neq n$. Finally, \Cref{thick} gives $\delta_n\in (0,\varepsilon_n)$ and a component $\mathcal C_{\delta_n,n}'$ of $f^{-1}(\mathcal C_{\delta_n,n})$ such that $\mathcal C_{\delta_n,n}'$ bounds a disk $\mathcal D_{\delta_n,n}'$ in $\Sigma'$ with $(\mathcal U_n'\cup \mathcal D_n')\supseteq\mathcal D_{\delta_n,n}'\supset \text{int}(\mathcal D_{\delta_n,n}')\supset \mathcal D'_n$ \big(equivalently, $\mathcal U_n'$ contains the annulus co-bounded by $\mathcal C_{\delta_n,n}'$ and $\mathcal C_n'$\big) and $f\big(\mathcal D_{\delta_n,n}'\setminus \text{int}(\mathcal D_n')\big)\subseteq \mathcal C_n\times [0,\varepsilon_n]$. Thus, $\mathcal D_{\delta_n,n}'\cap f^{-1}(\mathscr A)=\mathcal D_{n}'\cap f^{-1}(\mathscr A)$, $f(\mathcal D_{\delta_n,n}')\subseteq \Sigma_n$ for each $n$; and $\mathcal D_{\delta_n,n}'\cap \mathcal D_{\delta_m,m}'=\varnothing$ when $m\neq n$.

 Since $\mathcal C_{\delta_n,n}$ co-bounds an annulus with the primitive circle $\mathcal C_n$,  the inclusion $\mathcal C_{\delta_n,n}\hookrightarrow \Sigma_n$ is $\pi_1$-injective (see \Cref{primitivecircle}). Also, $\Sigma_n$ is homotopy equivalent to $ \bigvee_{\text{finite}}\Bbb S^1$, which implies that the universal cover of $\Sigma_n$ is contractible, and thus $\pi_2(\Sigma_n)=0$. Therefore, exactness of $$\cdots \longrightarrow \pi_2(\Sigma_n)\longrightarrow \pi_2(\Sigma_n, \mathcal C_{\delta_n,n})\longrightarrow \pi_1(\mathcal C_{\delta_n,n})\longrightarrow \pi_1(\Sigma_n)\longrightarrow\cdots$$ gives $\pi_2(\Sigma_n, \mathcal C_{\delta_n,n})=0$, i.e., we have a homotopy $H_n\colon \mathcal D'_{\delta_n,n}\times [0,1]\to \Sigma_n$ relative to $\mathcal C_{\delta_n,n}'$ from $f\vert(\mathcal D'_{\delta_n,n},\mathcal C'_{\delta_n,n})\to (\Sigma_n, \mathcal C_{\delta_n,n})$ to a map $\mathcal D'_{\delta_n,n}\to \mathcal C_{\delta_n,n}$ for each $n$; see \cite[Lemma 4.6.]{MR1867354}. Now, to conclude, apply \Cref{properhomotopy} on $\{H_n\}$. \end{proof}

\begin{remark}
\textup{In \Cref{diskremoval}, the number of components of $\mathscr A$ can be infinite; thus, the number of trivial components of $f^{-1}(\mathscr A)$ can be infinite. That's why we have removed all trivial components of $f^{-1}(\mathscr A)$ by a single proper homotopy upon considering all outermost disks simultaneously. This process is in contrast to the finite-type surface theory, where the number of decomposition circles is finite, and therefore all trivial circles in the collection of transversal pre-images of all decomposition circles can be removed one by one, considering the notion of an innermost disk.  }
\end{remark}

\subsection{Homotope a degree-one map between circles to a homeomorphism}
Previously, we have removed all trivial components keeping a neighborhood of each primitive component stationary from the transversal pre-image $f^{-1}(\mathscr A)$ of a preferred LFCS $\mathscr A$.  In this section, we properly homotope our pseudo proper homotopy equivalence $f\colon \Sigma'\to \Sigma$ to send each component $\mathcal C'$ of $f^{-1}(\mathscr A)$ homeomorphically onto a component $\mathcal C$ of $\mathscr A$ so that the restriction of $f$ to a small one-sided tubular neighborhood $\mathcal C'\times [1,2]$ of $\mathcal C'$ (on either side of $\mathcal C'$) can be described by the following homeomorphism: $$\mathcal C'\times [1,2]\ni (z,t)\longmapsto \big(f(z),t\big)\in \mathcal C\times [1,2].$$

First, we fix a few notations. Define $\partial_\varepsilon\coloneqq\Bbb S^1\times \varepsilon$ for $\varepsilon\in \Bbb R$ and $\mathbf I\coloneqq[0,1]$. Let $p\colon \Bbb S^1\times \Bbb R\to \Bbb S^1$ be the projection. The following lemma roughly says that a self-map of $\Bbb S^1\times [0,2]$ can be homotoped rel. $\Bbb S^1\times 0$ to map $\Bbb S^1\times [1,2]$ into itself by the product $\theta\times \text{Id}_{[1,2]}$, where $\theta$ is a self-map of $\Bbb S^1$.

\begin{lemma}\textup{Let $\Phi$ be a self-map of $A\coloneqq\Bbb S^1\times [0,2]$ such that $\Phi^{-1}(\partial_b)=\partial_b$ for each $b\in\{0,2\}$. If we are given a map $\varphi_2\colon \partial_2\to \partial_2$ and a homotopy $h_{(2)}\colon \partial_2\times \mathbf I\to \partial_2$ from $\Phi\vert\partial_2\to \partial_2$ to $\varphi_2$, then $\Phi$ can be homotoped relative to $\partial_0$ to map $\Bbb S^1\times [0,1]$ into $\Bbb S^1\times [0,1]$ and to satisfy $\Phi(-,r)=\big(p\circ\varphi_2(-,2),r\big)$ for each $r\in [1,2]$.} \label{primitiveimpliesdegreeone}
\end{lemma}
\begin{remark}
\textup{In \Cref{primitiveimpliesdegreeone}, up to homotopy, $\varphi_2$ is either a constant map or a covering map.}\label{remarkprimitiveimpliesdegreeone}
\end{remark}

\begin{proof}
Homotope $\Phi$ relative to $\partial_0\cup \partial_2$ so that $\Phi\big(\Bbb S^1\times [0,1]\big)\subseteq \Bbb S^1\times [0,1]$ and $\Phi(z,r)=\big(p\circ \Phi(z,2),r\big)$ for all $(z,r)\in\Bbb S^1\times [1,2]$. For each $ r\in [1,2]$, $h_{(2)}$ provides a homotopy $h_{(r)}\colon \partial_r\times \mathbf I\to \partial_r$. Let $H\colon(\partial_0\cup \partial_1)\times \mathbf I\to\partial_0\cup \partial_1$ be the homotopy defined as follows: $H\vert\partial_1\times \mathbf I=h_{(1)}$ and $H(-,t)\vert\partial_0=\Phi\vert\partial_0$ for any $t\in [0,1]$. Homotopy extension theorem gives a homotopy $\widetilde H\colon \Bbb S^1\times [0,1]\times \mathbf I\to \Bbb S^1\times [0,1]$ such that $\widetilde H\vert(\partial_0\cup \partial_1)\times \mathbf I=H$. Finally, paste $\widetilde H$ with the collection $h_{(r)},\ 1\leq r\leq 2$.\end{proof}

The following theorem is the simple modification (in the proper category) of the analog theorem for closed surfaces.

\begin{theorem}
\textup{Let $f\colon \Sigma'\to \Sigma$ be a smooth pseudo proper homotopy equivalence between two non-compact surfaces, where $\Sigma'\not\cong \Bbb R^2\not \cong \Sigma$; and let $\mathscr A$ be a preferred \textup{LFCS} on $\Sigma$ such that $f\ \stackinset{c}{}{c}{0.1ex}{$\top$}{$\abxpitchfork$}\ \mathscr A$. Then $f$ can be properly homotoped to remove all trivial components of the $f^{-1}(\mathscr A)$ as well as to map each primitive component of $f^{-1}(\mathscr A)$ homeomorphically onto a component of $\mathscr A$. Moreover, after this proper homotopy, near each component of $f^{-1}(\mathscr A)$, the map $f$ can be described as follows:}

\textup{Let $\mathcal C_{\textbf{\textup{p}}}'$ (resp. $\mathcal C$) be a component of $f^{-1}(\mathscr A)$ (resp. $\mathscr A$) such that $f\vert \mathcal C_{\textbf{\textup{p}}}'\to \mathcal C$ is a homeomorphism. Then $\mathcal C_{\textbf{\textup{p}}}'$ (resp. $\mathcal C$)  has two one-sided tubular neighborhoods $\mathcal M'$ and $\mathcal N'$ (resp. $\mathcal M$ and $\mathcal N$) with some specific identifications $(\mathcal M',\mathcal C_{\textbf{\textup{p}}}')\cong(\mathcal C'_\textbf{p}\times [1,2],\mathcal C_{\textbf{\textup{p}}}'\times 2)\cong (\mathcal N',\mathcal C_{\textbf{\textup{p}}}')$ \big(resp. $(\mathcal M,\mathcal C)\cong (\mathcal C\times [1,2],\mathcal C\times 2)\cong (\mathcal N,\mathcal C)$\big) such that the following hold:}

\begin{itemize}
    \item \textup{$\mathcal M'\cup \mathcal N'$ is a (two-sided) tubular neighborhood of $\mathcal C_{\textbf{\textup{p}}}'$;}
    \item \textup{$f\vert \mathcal M'\to \mathcal M$ and $f\vert \mathcal N'\to \mathcal N$ are homeomorphisms described by $\mathcal C_{\textbf{\textup{p}}}'\times [1,2]\ni (z,t)\longmapsto \big(f(z),t\big)\in \mathcal C\times [1,2].$}
\end{itemize} 
\label{deg1tohomeo}
\end{theorem}
\begin{remark}
\textup{In \Cref{deg1tohomeo}, though $\mathcal M'\cup \mathcal N'$ is a (two-sided) tubular neighborhood of $\mathcal C_{\textbf{\textup{p}}}'$, both $\mathcal M$ and $\mathcal N$ may lie on the same side of $\mathcal C$, i.e., $\mathcal M\cup\mathcal N$ may not be a two-sided tubular neighborhood of $\mathcal C$.}
\end{remark}

\begin{proof}\label{proofdeg1tohomeo}
Let $\{\mathcal C_{\textbf{p}n}'\}$ be the collection of all primitive components of $f^{-1}(\mathscr A)$. Assume $\mathcal C_n$ represents that component of $\mathscr A$ for which $f(\mathcal C_{\textbf{p}n}')\subseteq \mathcal C_n$. Note $\mathcal C_n$ may equal to $\mathcal C_m$ even if $m\neq n$.

\begin{claim} \label{claim00}
    \textup{There are one-sided compact tubular neighborhoods $\mathcal U'_n,\mathcal V'_n(\subseteq \Sigma')$ of $\mathcal C_{\textbf{p}n}'$, and there are one-sided compact tubular neighborhoods $\mathcal U_n,\mathcal V_n(\subseteq \Sigma)$ of $\mathcal C_n$ such that after defining $\mathcal T_n'\coloneqq\mathcal U'_n\cup \mathcal V'_n$, the following hold:}

\begin{enumerate}
    \item[$(1)$] \textup{$\widetilde{\mathscr A}\coloneqq\mathscr A\cup \big\{(\partial \mathcal U_n\cup \partial \mathcal V_n)\setminus \mathcal C_n\big\}_n$ is an \textup{LFCS} and   $f\ \stackinset{c}{}{c}{0.1ex}{$\top$}{$\abxpitchfork$}\ \widetilde{\mathscr A}$;}
    \item[$(2)$] \textup{$\partial\mathcal U_n'\setminus \mathcal C'_{\textbf{p}n}$ (resp. $\partial\mathcal V_n'\setminus \mathcal C'_{\textbf{p}n}$) is the only component of $f^{-1}(\partial \mathcal U_n\setminus \mathcal C_n)\cap \mathcal U_n'$ \big(resp. $f^{-1}(\partial \mathcal V_n\setminus \mathcal C_n)\cap \mathcal V_n'$\big) that co-bounds an annulus with $\mathcal C'_{\textbf{p}n}$ (see \Cref{DegreeOneToHomeomorphism});}
    \item[$(3)$] \textup{each point of $\text{int}(\mathcal U_n')$ \big(resp. $\text{int}(\mathcal V_n')$\big) that is sufficiently near to $\mathcal C_{\mathbf{p}n}'$ is mapped into $\text{int}(\mathcal U_n)$ \big(resp. $\text{int}(\mathcal V_n)$\big);}
    \item[$(4)$]  \textup{$\mathcal T_n'$ is a two-sided tubular neighborhood of $\mathcal C'_{\textbf{p}n}$ with $f^{-1}(\mathscr A)\cap \mathcal T_n'=\mathcal C'_{\textbf{p}n}$; and }
    \item[$(5)$]\textup{$\mathcal T_n'\cap\mathcal T_m'=\varnothing$ if $m\neq n$, and $(\mathcal U_n\cup \mathcal V_n)\to \infty$.}
\end{enumerate} 
   
\end{claim}
\begin{proof}[Proof of \Cref{claim00}]
    For any positive integer $n_0$, \Cref{remarktranshomotopy} says that the set $\{m\in \Bbb N: \mathcal C_m=\mathcal C_{n_0}\}$ is finite. Also, $\mathscr A$ is locally finite. Thus $\{\mathcal C_n: n\in \Bbb N\}$ is locally finite. So, for each $n$, there exists a two-sided tubular neighborhood $\mathcal C_n\times [-\varepsilon_n,\varepsilon_n]$ of  $\mathcal C_n$ with $\mathcal C_n\times 0\equiv \mathcal C_n$ such that $\big\{\mathcal C_n\times [-\varepsilon_n,\varepsilon_n]:n \in \Bbb N\big\}$ is a locally finite collection. Further, for each $n\in \Bbb N$, we may assume that $f\ \stackinset{c}{}{c}{0.1ex}{$\top$}{$\abxpitchfork$}\ \ (\mathcal C_n\times t_n)$ whenever  $t_n\in [-\varepsilon_n,\varepsilon_n]$ by \Cref{trans1}. 
    
    Now, since $f^{-1}(\mathscr A)$ is a locally finite collection, for each $n$, there are one-sided compact tubular neighborhoods $\mathcal U'_n,\mathcal V'_n$ of $\mathcal C_{\textbf{p}n}'$ in $\Sigma'$ such that after defining $\mathcal T_n'\coloneqq\mathcal U'_n\cup \mathcal V'_n$, the following hold: $\mathcal T_n'$ is a two-sided tubular neighborhood of $\mathcal C'_{\textbf{p}n}$, $f^{-1}(\mathscr A)\cap \mathcal T_n'=\mathcal C'_{\textbf{p}n}$,  and $\mathcal T_n'\cap\mathcal T_m'=\varnothing$ if $m\neq n$. Moreover, using \Cref{technicallemma}, $f(\mathcal T_n')\subseteq \mathcal C_n\times [-\varepsilon_n,\varepsilon_n]$ can also be assumed for each $n$.
    
    Next, by  \Cref{thick}, we may further assume $\partial \mathcal U_n'\setminus \mathcal C_{\textbf{p}n}'$ (resp. $\partial \mathcal V_n'\setminus \mathcal C_{\textbf{p}n}'$) is a component of $f^{-1}(\mathcal C_n\times x_n)$ \big(resp. $f^{-1}(\mathcal C_n\times y_n)$\big) for some $x_n,y_n\in (-\varepsilon_n,0)\cup (0, \varepsilon_n)$ such that after defining $\mathcal U_n$ (resp. $\mathcal V_n$) as the annulus in $\mathcal C_n\times [-\varepsilon_n,\varepsilon_n]$ co-bounded by $\mathcal C_n\times 0$ and $\mathcal C_n\times x_n$ (resp. $\mathcal C_n\times y_n$), both $(2)$ and $(3)$ of \Cref{claim00} do hold. Finally, $\mathcal C_n\times [-\varepsilon_n, \varepsilon_n]\to \infty$ implies $(\mathcal U_n\cup \mathcal V_n)\to \infty$.
\end{proof}

Using \Cref{diskremoval}, keeping stationary a neighborhood of each primitive component of $f^{-1}(\widetilde{\mathscr A})$, we can properly homotope $f$ to remove all trivial components from $f^{-1}(\widetilde{\mathscr A})$. So, after this proper homotopy,  $(2)$ and $(3)$ of \Cref{claim00} imply that $f(\mathcal U_n')\subseteq \mathcal U_n,\ f^{-1}(\partial \mathcal U_n)\cap \mathcal U_n'=\partial \mathcal U_n'$ and $f(\mathcal V_n')\subseteq \mathcal V_n,\ f^{-1}(\partial \mathcal V_n)\cap  \mathcal V_n'=\partial \mathcal V_n'$. Notice the abuse of notation, the initial and final maps of this proper homotopy both are denoted by $f$. 
 \begin{figure}[ht]$$\adjustbox{trim={0.0\width} {0.0\height} {0.0\width} {0.0\height},clip}{\def\svgwidth{1.286\linewidth}
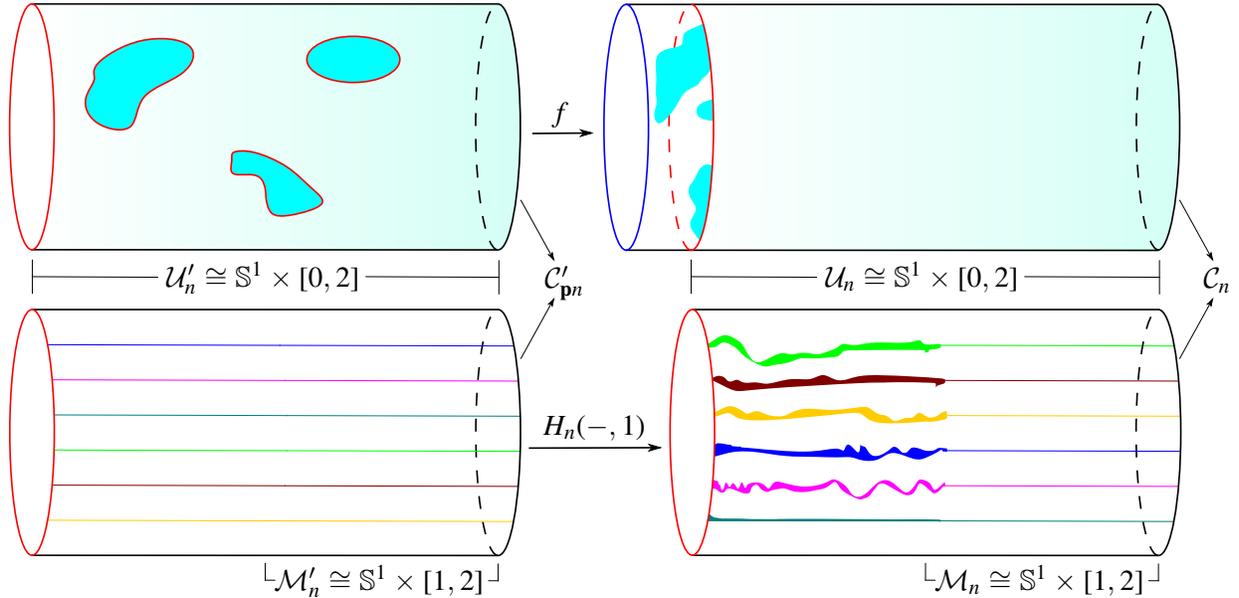}$$\caption{On the top: description of $f\vert\mathcal U_n'\to \mathcal U_n$, using \Cref{thick}. At the bottom: after removing all trivial components of $f^{-1}(\partial \mathcal U_n\backslash \mathcal C_n)$ from $\mathcal U_n'$ and then applying \Cref{primitiveimpliesdegreeone} to $f\vert\mathcal U_n'\to \mathcal U_n$, we obtain $H_n(-,1)\vert\mathcal U_n'\to \mathcal U_n$.  }\label{DegreeOneToHomeomorphism}\end{figure}

 Now, let $h_n\colon\mathcal C_{\textbf{p}n}'\times [0,1]\to \mathcal C_n$ be a homotopy from $f\vert\mathcal C_{\textbf{p}n}'\to \mathcal C_n$ such that $h_n(-,1)$ is either a constant map or a covering map between two circles.  Applying \Cref{primitiveimpliesdegreeone} on $f\vert\mathcal U_n'\to \mathcal U_n$ and $f\vert\mathcal V_n'\to \mathcal V_n$ separately upon considering $h_n$; a homotopy $H_n\colon\mathcal T'_n\times [0,1]\to \mathcal U_n\cup\mathcal V_n$ relative to $\partial \mathcal T_n'$ exists such that $H_n(-,0)=f\vert\mathcal T_n'$, $\big(H_n(-,1)\big)^{-1}(\mathcal C_n)=\mathcal C'_{\textbf{p}n}$, and  $H_n(-,1)\vert\mathcal C'_{\textbf{p}n}\to \mathcal C_n$  is the same as $h_n(-,1)$.  See \Cref{DegreeOneToHomeomorphism}.
 
 Next, $(5)$ of \Cref{claim00} tells that we can apply \Cref{properhomotopy} on $\{H_n\}$ to obtain a proper homotopy $\mathcal H\colon \Sigma'\times [0,1]\to \Sigma$ starting from $f$. Next, being an isomorphism, $\pi_1(f)=\pi_1\big(\mathcal H(-,1)\big)$ preserves primitiveness, i.e.,  $h_n(-,1)=\mathcal H(-,1)\vert\mathcal C'_{\textbf{p}n}\to \mathcal C_n$ must be a homeomorphism.  Thus, $\mathcal H$ is our ultimate required homotopy.

 Finally, we need to describe $f$ near each component of $f^{-1}(\mathscr A)$ after the proper homotopy $\mathcal H$. Abusing notation, the final map of $\mathcal H$ will be denoted by $f$. Since \Cref{primitiveimpliesdegreeone} is being used, we have $\mathcal M_n'\subseteq \mathcal U_n'$ and $\mathcal M_n\subseteq \mathcal U_n$ with the identifications $\big(\mathcal M_n',\mathcal C_{\textbf{\textup{p}}n}'\big)\cong\big(\mathcal C_{\textbf{\textup{p}}n}'\times [1,2],\mathcal C_{\textbf{\textup{p}}n}'\times 2\big),\ \big(\mathcal M_n,\mathcal C_{n}\big)\cong \big(\mathcal C_n\times [1,2],\mathcal C_{n}\times 2\big)$ such that after the proper homotopy $\mathcal H\colon \Sigma'\times [0,1]\to \Sigma$, the map $f$ sends $\mathcal C_{\textbf{\textup{p}}n}'\times r$ onto $\mathcal C_{n}\times r$ using the homeomorphism $f\vert\mathcal C_{\textbf{\textup{p}}n}'\to \mathcal C_{n}$ for all $r\in [1,2]$. See \Cref{DegreeOneToHomeomorphism}. Similar reasoning for $f\vert\mathcal V_n'\to \mathcal V_n$.
\end{proof}
The following proposition, \emph{which we don't need to use anywhere}, tells what happens if we drop the phrase ``homotopy equivalence'' in the statement of \Cref{deg1tohomeo}. Its proof is almost the same.

\begin{proposition}
    \textup{Let $f\colon \Sigma'\to \Sigma$ be a smooth proper map between two non-compact surfaces, where $\Sigma'\not\cong \Bbb R^2\not \cong \Sigma$; and let $\mathscr A$ be a preferred \textup{LFCS} on $\Sigma$ such that $f\ \stackinset{c}{}{c}{0.1ex}{$\top$}{$\abxpitchfork$}\ \mathscr A$. Then $f$ can be properly homotoped to remove all trivial components of the $f^{-1}(\mathscr A)$ as well as to map each primitive component of $f^{-1}(\mathscr A)$  into a component of $\mathscr A$  so that for any component $\mathcal C$ of $\mathscr A$ and any primitive component $\mathcal C_\textbf{p}'$ of $f^{-1}(\mathcal C)$, after this proper homotopy, $f\vert \mathcal C_\textbf{p}'\to \mathcal C$ is either a constant map or a covering map. }
\end{proposition}
\subsection{Annulus removal}
In the previous two sections, after removing all trivial components from the transversal pre-image of a decomposition circle, the remaining primitive circles have been mapped homeomorphically to that decomposition circle. This section aims to remove all these primitive circles except one from the inverse image of each decomposition circle using the following three steps: annulus bounding, then annulus compression, and finally, annulus pushing.

At first, annulus bounding: Consider the collection of inverse images of all decomposition circles. The following lemma says that any two circles in this collection co-bound an annulus in the domain if and only if their images are the same. In other words, in the domain, by pasting all small annuli, we get the outermost annulus corresponding to a decomposition circle. 
\begin{lemma}
\textup{Let $f\colon \Sigma'\to \Sigma$ be a homotopy equivalence between two non-compact surfaces, and let $\mathscr A'$, $\mathscr A$ be two \textup{LFCS} on $\Sigma'$, $\Sigma$, respectively, such that $f$ maps each component of $\mathscr A'$ homeomorphically  onto a component of $\mathscr A$. Suppose each component of $\mathscr A$ is primitive, and any two distinct components of $\mathscr A$ don't co-bound an annulus in $\Sigma$. Let $\mathcal C'_0,\mathcal C_1'$ be two distinct components of $\mathscr A'$. Then $\mathcal C'_0,\mathcal C_1'$ co-bound an annulus in $\Sigma'$ if and only if $f(\mathcal C_0')=f(\mathcal C_1')$.} \label{exitenceofoutermostannulus}
\end{lemma}
\begin{proof} To prove only if part, let $\Phi\colon \Bbb S^1\times [0,1]\hookrightarrow \Sigma'$ be an embedding such that $\Phi(\Bbb S^1, k)=\mathcal C_k'$ for $k=0,1$. Note that $f$ maps each component of $\mathscr A'$ homeomorphically  onto a component of $\mathscr A$, and each component of $\mathscr A$ is a primitive circle on $\Sigma$. Thus, the embeddings $f \Phi(-,0), f \Phi(-,1)\colon \Bbb S^1\hookrightarrow \Sigma$ are freely homotopic; and hence $f \Phi(-,0), f \Phi(-,1)\colon \Bbb S^1\hookrightarrow \Sigma$ represent the same non-trivial conjugacy class in $\pi_1(\Sigma,*)$. Since any two distinct components of $\mathscr A$ don't co-bound an annulus in $\Sigma$, by \Cref{AnnulusEmbedding}, $f(\mathcal C_0')=f(\mathcal C_1')$.

To prove if part, let $g\colon \Sigma\to \Sigma'$ be a homotopy inverse of $f$, and let $\mathcal C$ be the component of $\mathscr A$ defined by $\mathcal C\coloneqq f(\mathcal C_0')=f(\mathcal C_1')$. Now, $f\vert \mathcal C_k'\to f(\mathcal C_k')$ is a homeomorphism for $k=0,1$. Thus, for a homeomorphism $j\colon \mathbb S^1\xrightarrow{\cong} \mathcal C$, there are homeomorphisms $\ell_0\colon \mathbb S^1\xrightarrow{\cong} \mathcal C'_{0}$ and $\ell_1\colon \mathbb S^1\xrightarrow{\cong} \mathcal C'_{1}$ such that $f\ell_0=j=f\ell_1$. Since $\ell_0\simeq  gf\ell_0=gj= g f \ell_1\simeq \ell_1$,  applying \Cref{AnnulusEmbedding} to $\ell_0,\ell_1$, we are done.\end{proof}

The following theorem, which will be used to compress each annulus bounded by two primitive circles of the domain, roughly says that most homotopies of a circle embedded in a surface are trivial.

\begin{theorem}{\textup{\cite[Lemma 4.9.15.]{scottbook}}}
\textup{Let $\mathbf S$ be a compact bordered surface  other than the disk, and let $\Phi$ be a map from $\mathbf A\coloneqq\Bbb S^1\times [0,1]$ to $\mathbf S$ such that $\Phi\big(\textup{int}(\mathbf A)\big)\subseteq \textup{int}(\mathbf S)$ and there is a boundary component $\mathbf C$ of $\mathbf S$ for which $\Phi(-,0),\Phi(-,1)\colon \Bbb S^1\xrightarrow{\cong}\mathbf C$ are the same homeomorphisms. Then $\Phi$ can be homotoped relative to $\partial \mathbf A$ to map $\mathbf A$ onto $\mathbf C$.} \label{compressingannulus}\end{theorem}

The following theorem considers the last two steps - annulus compressing and annulus pushing. At first, by a proper homotopy, each outermost annulus will be mapped onto its decomposition circle; after that, by another proper homotopy, each outermost annulus will be pushed into a one-sided tubular neighborhood of one of its boundary components. 
\begin{theorem}
\textup{Let $f\colon \Sigma'\to \Sigma$ be a smooth pseudo proper homotopy equivalence between two non-compact surfaces, where $\Sigma'\not\cong \Bbb R^2\not \cong \Sigma$; and let $\mathscr A$ be a preferred \textup{LFCS} on $\Sigma$ such that $f\ \stackinset{c}{}{c}{0.1ex}{$\top$}{$\abxpitchfork$}\ \mathscr A$. Suppose any two distinct components of $\mathscr A$ don't co-bound an annulus in $\Sigma$. In that case, $f$ can be properly homotoped to a proper map $g$ such that for each component $\mathcal C$ of $\mathscr A$, either $g^{-1}(\mathcal C)$ is empty or $g^{-1}(\mathcal C)$ is a component of $f^{-1}(\mathscr A)$ that is mapped homeomorphically onto $\mathcal C$ by $g$.} \label{annulusremovalfinal}
\end{theorem}

\begin{proof}
Using \Cref{deg1tohomeo}, we may assume each component of $f^{-1}(\mathscr A)$ is primitive and also mapped homeomorphically onto a component of $\mathscr A$. So for simplicity, we may drop the subscript $\mathbf p$ to indicate a primitive component of $f^{-1}(\mathscr A)$. Let $\{\mathcal C_n\}$ be the pairwise disjoint collection of all those components of $\mathscr A$ so that for each $n$, $f^{-1}(\mathcal C_n)$ has more than one component. By \Cref{exitenceofoutermostannulus}, for each $n$, an annulus $\mathcal A_n'$ (say the $n$-th outermost annulus) exists with the following properties: $(i)$ $\partial \mathcal A'_n\subseteq f^{-1}(\mathcal C_n)$,  $(ii)$ $\mathcal A'_n$ is not contained in the interior of an annulus bounded by any two components of $f^{-1}(\mathscr A)$. Thus $\mathcal A_n'\cap f^{-1}(\mathscr A)=f^{-1}(\mathcal C_n)$ and $\mathcal A_n'\cap \mathcal A_m'=\varnothing$ for $m\neq n$. Now, using \Cref{AnnulusEmbedding}, find a parametrization $\tau_n\colon \Bbb S^1\times [0,k_n]\xrightarrow{\cong}\mathcal A'_n$ for some integer $k_n\geq 1$ so that $\tau_n\big(\Bbb S^1\times \{0,...,k_n\}\big)=f^{-1}(\mathcal C_n)$ and  $f\tau_n(-,\ell)\colon \Bbb S^1\xrightarrow{\cong}\mathcal C_n$ represents the same homeomorphism of $\mathcal C_n$ for each $\ell=0,...,k_n$. 
\begin{claim}
\textup{The proper map $f\colon \Sigma'\to \Sigma$ can be properly homotoped relative to $\Sigma'\setminus \bigcup_n\textup{int}(\mathcal A_n')$ so that $f(\mathcal A_n')=\mathcal C_n$ for each $n$.}\label{claim1}
\end{claim}
\begin{figure}[ht]$$\adjustbox{trim={0.0\width} {0.0\height} {0.0\width} {0.0\height},clip}{\def\svgwidth{.99\linewidth}
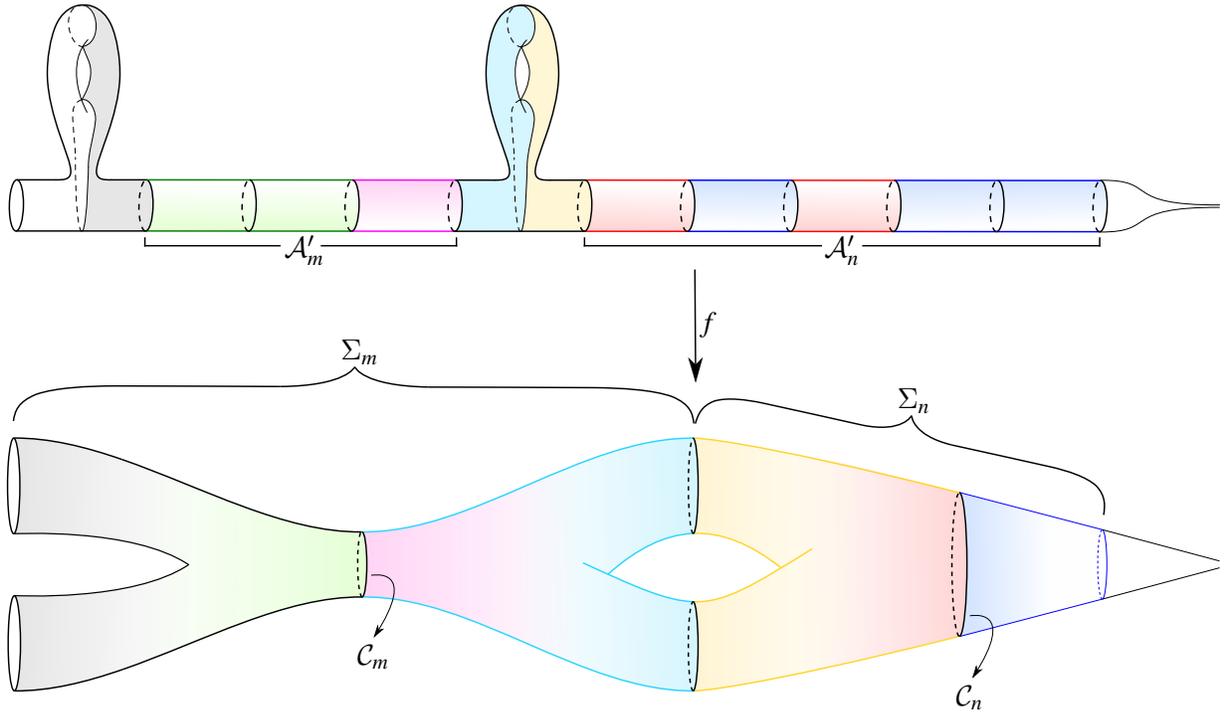
}$$\caption{Illustration of parts (1) and (3) of the definition of $\Sigma_n$ given in the proof of \Cref{claim1}. Only black circles denote a component of either $\mathscr A$ or a component of $f^{-1}(\mathscr A)$.}\label{AnnulusCompression}\end{figure}

\begin{proof}[Proof of \Cref{claim1}]
For each integer $n$, we will construct a compact bordered sub-surface $\Sigma_n$ of $\Sigma$ with $f(\mathcal A_n')\subseteq \text{int}(\Sigma_n)$  such that $\Sigma_n\to \infty$. Roughly, $\Sigma_n$ will be obtained from taking the union of all those complementary components of $\Sigma$ (if a punctured disk appears, truncate it), which are hit by $f(\mathcal A_n')$.

Using continuity of $f\vert\Sigma'\setminus f^{-1}(\mathscr A)\to \Sigma\setminus \mathscr A$, we can say that $f(\mathcal A_n')\subseteq \mathcal X_n\cup \mathcal Y_n$, where $\mathcal X_n$ and $\mathcal Y_n$ are complementary components of $\Sigma$ decomposed by $\mathscr A$ such that $\mathcal C_n\subseteq \partial \mathcal X_n\cap \partial\mathcal Y_n$.   \begin{itemize}
    \item[(1)]  We define $\Sigma_n$ as $\Sigma_n\coloneqq\mathcal X_n\cup \mathcal Y_n$ if either of the following happens: $(i)$ $\mathcal X_n\cong S_{0,3}\cong \mathcal Y_n$; or $(ii)$ $\mathcal X_n\cong S_{1,1}$ and $\mathcal Y_n\cong S_{0,3}$; or $(iii)$ $\mathcal Y_n\cong S_{1,1}$ and $\mathcal X_n\cong S_{0,3}$. See \Cref{AnnulusCompression}.
    \item[(2)] If $\mathcal X_n\cong S_{0,1,1}\cong \mathcal Y_n$ (in this case, $\Sigma$ is homeomorphic to the punctured plane), then using compactness of $f(\mathcal A_n')$, let $\Sigma_n$ be an annulus in $\mathcal X_n\cup\mathcal Y_n$ so that $f(\mathcal A_n')\subseteq\text{int}(\Sigma_n)$.
    \item[(3)]  If $\mathcal X_n\cong S_{0,1,1}$, and $\mathcal Y_n$ is homeomorphic to either  $S_{0,3}$ or $S_{1,1}$, then using compactness of $f(\mathcal A_n')$, find an annulus $\mathcal A_n$ in $\mathcal X_n$ so that $f(\mathcal A_n')\subseteq\text{int}(\mathcal A_n\cup \mathcal Y_n)$. Define $\Sigma_n\coloneqq\mathcal A_n\cup \mathcal Y_n$. See \Cref{AnnulusCompression}.
    \item[(4)] If $\mathcal Y_n\cong S_{0,1,1}$, and $\mathcal X_n$ is homeomorphic to either  $S_{0,3}$ or $S_{1,1}$, define $\Sigma_n$ similarly, as given in $(3)$.
\end{itemize}   Thus, $f(\mathcal A_n')\subseteq \text{int}(\Sigma_n)$ for each $n$. Now, we show $\Sigma_n\to \infty$. So, consider a compact subset $\mathcal K$ of $\Sigma$. Let $\mathbf S_1,..., \mathbf S_m$ be a collection of complementary components of $\Sigma$ decomposed by $\mathscr A$ such that $\mathcal K\subseteq \text{int}\left(\bigcup_{\ell=1}^m\mathbf S_\ell\right)$. Define $\mathbf S\coloneqq\bigcup_{\ell=1}^m\mathbf S_\ell$.   Notice that for an integer $n$, $f(\mathcal A_n')\cap  \mathbf S\neq \varnothing$ if and only if $\mathcal C_n$ is a component of $\bigcup_{\ell=1}^{m}\partial \mathbf S_\ell$. This is due to the construction of each $\Sigma_n$; see \Cref{AnnulusCompression}. Since $\mathcal C_n\to \infty$ and $\bigcup_{\ell=1}^{m}\partial \mathbf S_\ell$ is compact, we can say that $f(\mathcal A_n')\cap \mathbf S=\varnothing$ for all sufficiently large $n$. Now, $\mathcal K\subseteq\text{int}(\mathbf S)$ and each $\Sigma_n$ is obtained from taking the union of all those complementary components of $\Sigma$ (if a punctured disk appears, truncate it), which are hit by $f(\mathcal A_n')$. Thus, $\Sigma_n\cap \mathcal K=\varnothing$ for all sufficiently large $n$. Therefore, $\Sigma_n\to \infty$, as $\mathcal K$ is an arbitrary compact subset of $\Sigma$.

Next, for each $\ell\in\{1,...,k_n\}$, applying \Cref{compressingannulus} to each $f\tau_n\vert\Bbb S^1\times [\ell-1,\ell]\to \mathcal Z_n$, where $\mathcal Z_n$ can be either $\Sigma_n\cap \mathcal X_n$ or $\Sigma_n\cap \mathcal Y_n$, we have a homotopy $H_n\colon\mathcal A_n'\times [0,1]\to \Sigma_n$ relative to $\partial \mathcal A_n'$ such that $H_n(-,0)=f\vert\mathcal A_n'$ and $H_n(\mathcal A_n',1)=\mathcal C_n$. Finally, apply \Cref{properhomotopy} on $\{H_n\}$ to complete proof of the \Cref{claim1}.\label{proofofclaim1}
\end{proof}
\begin{figure}[ht]$$\adjustbox{trim={0.0\width} {0.0\height} {0.3\width} {0.0\height},clip}{\def\svgwidth{1.436\linewidth}
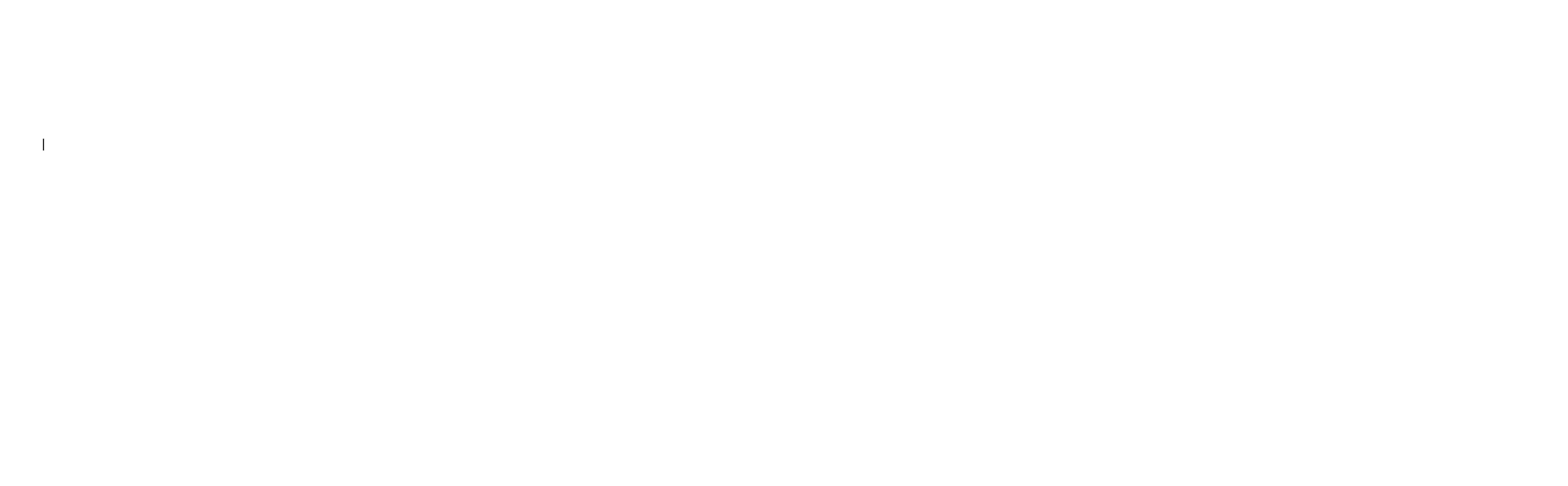
}$$\caption{Description of $f\vert\mathcal A_{\varepsilon n}'\to \mathcal M_n$ \big(resp. $H_n(-,1)\colon \mathcal A_{\varepsilon n}'\to \mathcal M_n$\big)  using \Cref{deg1tohomeo}, \Cref{claim1} (resp. \Cref{pushingleft}). Only black circles denote a component of either $\mathscr A$ or a component of $f^{-1}(\mathscr A)$.}\label{AnnulusRemoval}\end{figure} 
Now, consider \Cref{AnnulusRemoval}, where    $\mathcal M'_{n\alpha},\mathcal M_{n}$ are provided by \Cref{deg1tohomeo} such that after defining $\mathcal A_{\varepsilon n}'$ as $\mathcal A_n'\cup \mathcal M_{n\alpha}'$, we can think $$
    \big(\mathcal A_{\varepsilon n}',\mathcal M_{n\alpha}',\mathcal A'_n\big)\cong \big( \Bbb S^1\times [1,3],\Bbb S^1\times[1,2],\Bbb S^1\times [2,3]\big)\text{ and }\big(\mathcal M_{n},\mathcal C_n\big)\cong\big(\Bbb S^1\times[1,2],\Bbb S^1\times2\big)$$ resulting in the following description of $f$:  If $\theta\colon\Bbb S^1\to \Bbb S^1$ describes the homeomorphism $f\vert\mathcal C'_{n\alpha}\to \mathcal C_{n}$ under the above identification, then $f(z,t)=\big(\theta(z),t\big)$ for $z\in \Bbb S^1\times[1,2] $ and $f(z,t)\in \Bbb S^1\times 2$ for $(z,t)\in \Bbb S^1\times [2,3]$. Consider \Cref{claim1} to see why $f\big(\Bbb S^1\times [2,3]\big)=\Bbb S^1\times 2$.
    
    Now, use \Cref{pushingleft} to construct a homotopy $H_n\colon \mathcal A_{\varepsilon n}'\times [0,1]\to \mathcal M_n$ relative to $\partial \mathcal A'_{\varepsilon n}$ from $f\vert\mathcal A_{\varepsilon n}'\to \mathcal M_n$ to the map $H_n(-,1)$ so that $\big(H_n(-,1)\big)^{-1}(\mathcal C_n)=\mathcal C_{n\beta}'$ and $H_n(-,1)\big|\mathcal C_{n\beta}'\to \mathcal C_n$ is a homeomorphism.

Notice that we are using the setup of proof of \Cref{deg1tohomeo}. By (4) and  (5) of \Cref{claim00} given in the proof of \Cref{deg1tohomeo} show that $\mathcal A'_{\varepsilon n}\cap f^{-1}(\mathscr A)=f^{-1}(\mathcal C_n)$,   $\mathcal A'_{\varepsilon n}\cap \mathcal A'_{\varepsilon m}=\varnothing$ if $m\neq n$,  and $\mathcal M_{n}\to \infty$. Now, consider \Cref{properhomotopy} with $\{H_n\}$ to obtain the desired homotopy.\end{proof}

Now, we prove the annulus-pushing lemma used in the proof of the previous theorem.

\begin{lemma}
\textup{Any map $\varphi\colon \Bbb S^1\times [1,3]\to \Bbb S^1\times [1,2]$ which sends $\Bbb S^1\times r$ into $\Bbb S^1\times r$ for $1\leq r\leq 2$ and sends $\Bbb S^1\times r$ into $\Bbb S^1\times 2$ for $2\leq r\leq 3$; can be homotoped relative to  $\Bbb S^1\times \{1,3\}$ to satisfy $\varphi^{-1}(\Bbb S^1\times 2)=\Bbb S^1\times 3$.} \label{pushingleft}\end{lemma}\begin{proof}
Let $\varphi_1\colon \Bbb S^1\times [1,3]\to \Bbb S^1$ and $\varphi_2\colon \Bbb S^1\times [1,3]\to [1,2]$ be the components of $\varphi$. Consider a homeomorphism $\ell\colon [1,3]\to [1,2]$ with $\ell(1)=1,\ \ell(3)=2$.  Now, $H\colon \Bbb S^1\times [1,3]\times [0,1]\to \Bbb S^1\times [1,2]$ defined by $$H\big((z,s),t\big)\coloneqq\big(\varphi_1(z,s), (1-t)\varphi_2(z,s)+t\ell(s)\big)\text{ for }(z,s)\in \Bbb S^1\times [1,3]\text{ and }t\in [0,1]$$ is our required homotopy. \end{proof}

\begin{remark}
\textup{In \Cref{annulusremovalfinal}, the number of components of $\mathscr A$ can be infinite; thus, the number of outermost annuli (one outermost annulus for each component of $\mathscr A$, if any) can be infinite. That's why we have removed all outermost annuli simultaneously by a single proper homotopy, not one by one. Also, to prove the topological rigidity of closed surfaces, one may ignore the annulus removal process considering induction on the genus; see \cite[Theorem 3.1.]{MR541331} or \cite[Theorems 4.6.2 and 4.6.3]{scottbook}. But, since the genus of a non-compact surface can be infinite, we can't ignore the annulus removal process here.}
\end{remark}

\begin{subsection}{Is pseudo proper homotopy equivalence a map of degree \texorpdfstring{$\pm 1$}{±1}?}\label{degreeofapseudoproperhomotopyequivalence}
Let $f\colon \Sigma'\to \Sigma$ be a pseudo proper homotopy equivalence between two non-compact oriented surfaces, where surfaces are homeomorphic to neither the plane nor the punctured plane. Our aim in this section is to properly homotope $f$ to obtain a closed disk $\mathcal D\subseteq \Sigma$ so that $f\vert f^{-1}(\mathcal D)\to \mathcal D$ becomes a homeomorphism, and thus we show $\deg(f)=\pm 1$; see \Cref{degreeonemapchecking}. Having got this and then using \Cref{non-surjectivepropermaphasdegreezero}, it can be said that $f$ is surjective, which further implies that after a proper homotopy for removing unnecessary components from the transversal pre-image of a decomposition circle $\mathcal C$, a single circle will still be left that can be mapped onto $\mathcal C$ homeomorphically; see \Cref{annulusremovalfinal2}.

The argument for finding such a disk  $\mathcal D$ is based on finding a finite-type bordered surface $\mathbf S$ in $\Sigma$ such that for each component $\mathcal C$ of $\partial \mathbf S$, we have $f^{-1}(\mathcal C)\neq \varnothing$, even after any proper homotopy of $f$. Once we get $\mathbf S$, after a proper homotopy, we may assume that $f\vert f^{-1}(\partial \mathbf S)\to \partial \mathbf S$ is a homeomorphism; see \Cref{annulusremovalfinal}. Now, since $f$ is $\pi_1$-injective, by the topological rigidity of pair of pants together with the proper rigidity of the punctured disk, after a proper homotopy, one can show that $f\vert f^{-1}(\mathbf S)\to \mathbf S$ is a homeomorphism. Therefore, the required $\mathcal D$ can be any disk in $\text{int}(\mathbf S)$.

Now, to find such an $\mathbf S$, we consider two cases: If $\Sigma$ is either  an infinite-type surface or any $S_{g,0,p}$ with high complexity (to us, high complexity always means $g+p\geq 4$ or $p\geq 6$), then using $\pi_1$-surjectivity of $f$, we can choose $\mathbf S$ as a pair of pants in $\Sigma$ so that $\Sigma\setminus \mathbf S$ has at least two components and every component of $\Sigma\setminus \mathbf S$ has a non-abelian fundamental group. On the other hand, if $\Sigma$ is a finite-type surface, then we choose a punctured disk in $\Sigma$ as $\mathbf S$ so that the puncture of $ \mathbf S$ is an end $e\in \text{im}\big(\text{Ends}(f)\big)\subseteq \text{Ends}(\Sigma)$.

We can recall our earlier two examples to show that the plane and the punctured plane are the only surfaces for which our theory fails, i.e., consider the pseudo proper homotopy equivalences $\varphi\colon\Bbb C\ni z\longmapsto z^2\in \Bbb C$ and $\psi\colon \Bbb S^1\times \Bbb R\ni (z,x)\longmapsto \big(z,|x|\big)\in \Bbb S^1\times \Bbb R$. The local-homeomorphism $\varphi$ is a map of $\deg=\pm 2$ by \cite[Lemma 2.1b.]{MR192475} (note that for any local-homeomorphism $p\colon X\to Y$ between two manifolds, an orientation of $Y$ can be pulled back to give an orientation on $X$ so that $p$ becomes an orientation-preserving map). On the other and, $\deg(\psi)=0$ as $\psi$ is not surjective; see \Cref{non-surjectivepropermaphasdegreezero}. 

\subsubsection{Essential pair of pants and the degree of a pseudo proper homotopy equivalence}\label{exitenceofessentialpairofpants}
\begin{subsubdefinition}
\textup{A smoothly embedded pair of pants $\mathbf P$ in a surface $\Sigma$ is said to be an \emph{essential pair of pants} of $\Sigma$ if $\Sigma\setminus \mathbf P$ has at least two components and every component of $\Sigma\setminus \mathbf P$ has a non-abelian fundamental group.}
\end{subsubdefinition}

Finding an essential pair of pants in a non-compact surface will be divided into two cases: when the genus is at least two and when the space of ends has at least six elements.

\begin{definition}
\textup{Let $\mathbf P$ be a smoothly embedded copy of the pair of pants in a two-holed torus $\mathcal S$ (i.e., $\mathcal S$ is a copy of $S_{1,2}$).  We say \emph{$\mathbf P$ is obtained from  decomposing $\mathcal S$ into two copies of the pair of pants} if there exists another smoothly embedded copy $\widetilde{\mathbf P}$ of the pair of pants in $\mathcal S$ such that $\mathbf P\cup \widetilde{\mathbf P}=\mathcal S$ and $\mathbf P\cap \widetilde {\mathbf P}=\partial\mathbf P\cap \partial\widetilde {\mathbf P}$ is the union of two smoothly embedded disjoint circles in the interior of $\mathcal S$ (i.e., $\partial \mathbf P$ shares exactly two of its components with $\partial\widetilde {\mathbf P}$).}
\end{definition}

The following lemma says that every non-compact surface with a genus of at least two has an essential pair of pants with some additional properties.

\begin{subsublemma}
\textup{Let $\Sigma$ be a non-compact surface of the genus of at least two. Then $\Sigma$ has an essential pair of pants $\mathbf P$ with the following additional properties: $(1)$ $\Sigma$ contains a smoothly embedded copy $\mathcal S$ of $S_{1,2}$ such that $\Sigma\setminus \mathcal S$ has precisely two components and each component of $\Sigma\setminus \mathcal S$ has a non-abelian fundamental group, $(2)$ $\mathbf P$ is a smoothly embedded copy of the pair of pants in $\mathcal S$ obtained by decomposing $\mathcal S$ into two copies of the pair of pants.}\label{exitenceofessentialpairofpants1}  
\end{subsublemma}

\begin{figure}[ht]$$\adjustbox{trim={0.0\width} {0.00\height} {0.0\width} {0.0\height},clip, scale=1}{\def\svgwidth{\linewidth}
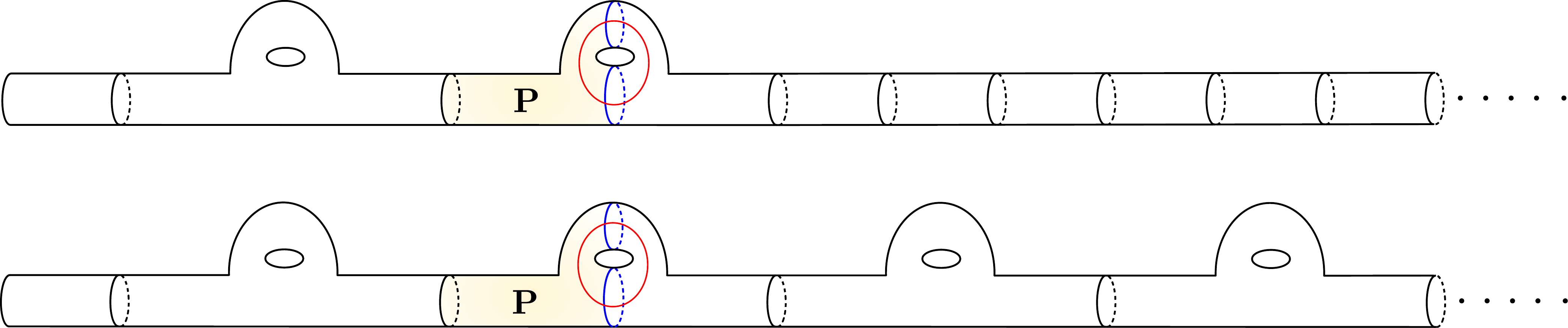
}$$\caption{Finding an essential pair of pants $\mathbf P$ in each of $S_{2,0,1}$ and Loch Ness Monster by decomposing a two-holed torus into two copies of the pair of pants.}\label{infinitegenus}\end{figure}

\begin{proof}
Consider an inductive construction of $\Sigma$; see \Cref{incudctiveconstruction}. Since $g(\Sigma)\geq 2$, at least two smoothly embedded copies of $S_{1,2}$ are used in this inductive construction. By \Cref{interchange}, without loss of generality, we may assume that two smoothly embedded copies of $S_{1,2}$ are used successively just after the initial disk; see \Cref{infinitegenus}. Among these two copies of $S_{1,2}$, breaking the last one (i.e., that copy of $S_{1,2}$ which we just used to construct $K_3$ from $K_2$) into two copies of the pair of pants, as illustrated in \Cref{infinitegenus}, we get the required essential pair of pants.\end{proof}
\begin{subsublemma}
 \textup{Let $f\colon \Sigma'\to \Sigma$ be a $\pi_1$-surjective map between two non-compact surfaces, where $\Sigma$ has the genus of at least two. Consider an essential pair of pants $\mathbf P$ in $\Sigma$ given by \Cref{exitenceofessentialpairofpants1}. Then $f^{-1}(\text{int }\mathbf P)\neq \varnothing$ and $f^{-1}(c)\neq \varnothing$ for each component $c$ of $\partial \mathbf P$.}\label{applicationofsurjectivitybetweenfundamentalgroups}
\end{subsublemma}

\begin{proof}
Let $\mathcal S$ be a smoothly embedded copy  of $S_{1,2}$ in $\Sigma$ such that $\mathbf P$ is obtained from decomposing $\mathcal S$ into two copies of the pair of pants. If possible, let $f^{-1}(\text{int }\mathbf P)\neq \varnothing$. By continuity of $f$, the image of $f$ is contained in precisely one of the two components of $\Sigma\setminus \text{int}(\mathbf P)$. But each component of $\Sigma\setminus \text{int}(\mathbf P)$ has a non-abelian fundamental group, i.e., $\pi_1(f)\colon \pi_1(\Sigma')\to \pi_1(\Sigma)$ is not surjective, a contradiction. Therefore, $f^{-1}(\text{int }\mathbf P)$ must be non-empty.

To prove the second part, let $c_1$, $c_2$, and $c_3$ denote all three components of $\mathbf P$ such that both $\Sigma\setminus c_1$ and $\Sigma\setminus (c_2\cup c_2)$ are disconnected, but neither $\Sigma\setminus c_2$ nor $\Sigma\setminus c_3$ is disconnected. In \Cref{infinitegenus},  $c_2$ and $c_3$ are blue circles, whereas the color of the third component $c_1$ is black.  Notice that we have a smoothly embedded primitive circle $\mathcal C\subseteq \text{int}(\mathcal S)$ (in \Cref{infinitegenus}, each red circle denotes $\mathcal C$) so that for each $k=2,3$, $c_k\cap \mathcal C$ is a single point, where $c_k$ intersects $\mathcal C$ transversally. Therefore, for each $k=2,3$, using the bigon criterion \cite[Proposition 1.7]{MR2850125}, any loop belonging to class $[\mathcal C]\in \pi_1(\Sigma)$ must intersect $c_k$. That is, if any of $f^{-1}(c_2)$ or $f^{-1}(c_3)$ were empty, then $[\mathcal C]$ would not belong to the image of $\pi_1(f)\colon \pi_1(\Sigma')\to \pi_1(\Sigma)$. But $f$ is $\pi_1$-surjective. Thus $f^{-1}(c_2)\neq \varnothing \neq f^{-1}(c_3)$. On the other hand, $\Sigma\setminus c_1$ has precisely two components, and each component of $\Sigma\setminus c_1$ has a non-abelian fundamental group, i.e., by continuity and $\pi_1$-surjectivity of $f$, we can say that $f^{-1}(c_1)\neq \varnothing$. 
\end{proof}

Now, we consider the second case of finding an essential pair of pants in a non-compact surface, namely when the space of ends has at least six elements.

\begin{subsublemma}\textup{Let $\Sigma$ be a non-compact surface with at least six ends. Then $\Sigma$ has an essential pair of pants $\mathbf P$ such that $\Sigma\setminus \mathbf P$ has precisely three components and each component of $\Sigma\setminus \mathbf P$ has a non-abelian fundamental group.}\label{exitenceofessentialpairofpants2}  
\end{subsublemma}

\begin{figure}[ht]$$\adjustbox{trim={0.0\width} {0.0\height} {0.0\width} {0.0\height},clip,scale=1}{\def\svgwidth{\linewidth}
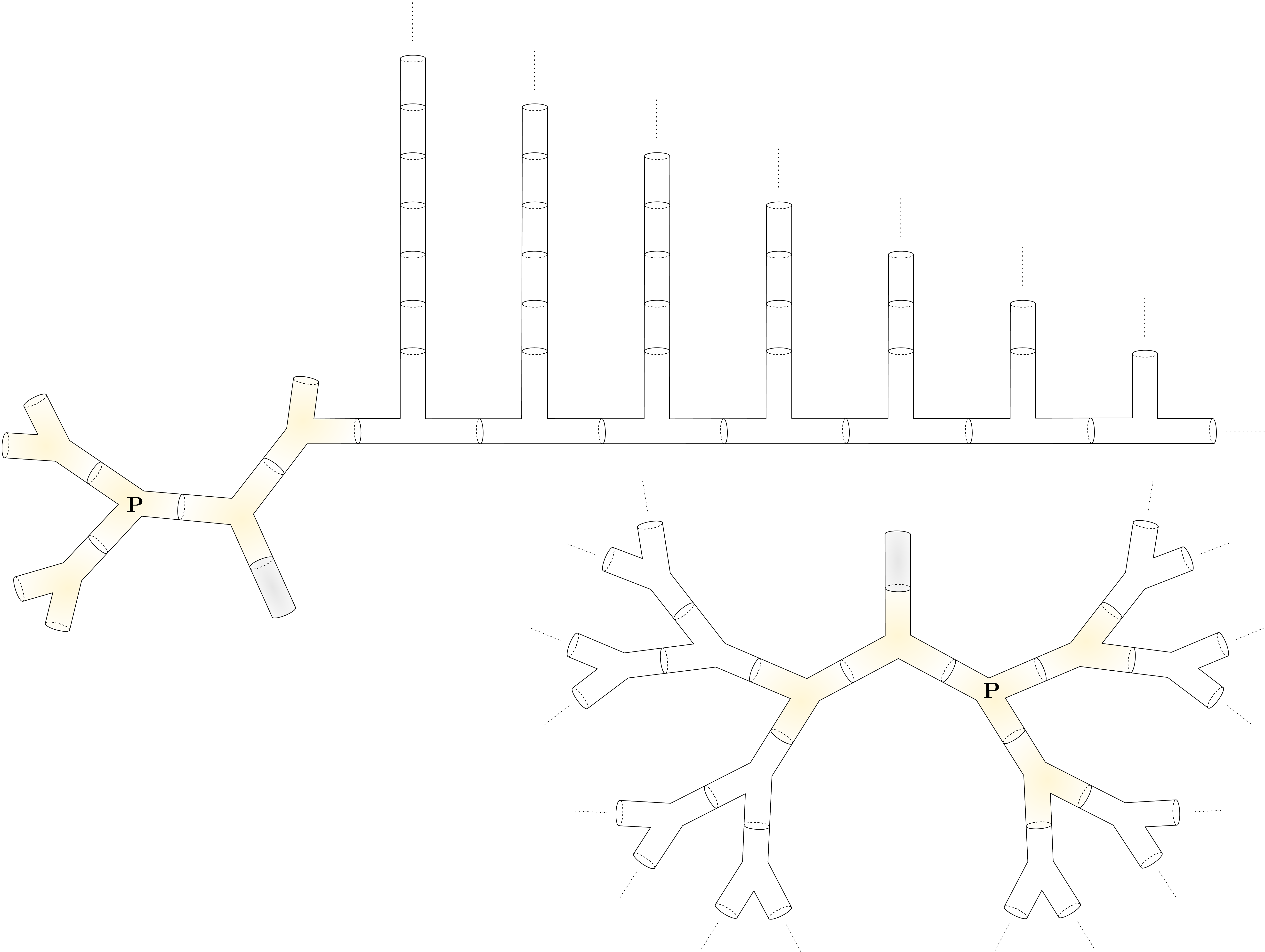
}$$\caption{Finding an essential pair of pants $\mathbf P$ in a non-compact surface with at least six ends.}\label{cantor}\end{figure}

\begin{proof}
Consider an inductive construction of $\Sigma$; see \Cref{incudctiveconstruction}. Since $\big|\text{Ends}(\Sigma)\big|\geq 6$, at least five smoothly embedded copies of $S_{0,3}$ are used in this inductive construction. By \Cref{interchange}, without loss of generality, we may assume that five smoothly embedded copies of $S_{0,3}$ are used successively just after the initial disk. Let $\mathbf P$ be the copy that shares all three boundary components with three other copies of this sequence of five copies of $S_{0,3}$; see \Cref{cantor}. Thus,  $\Sigma\setminus \mathbf P$ has precisely three components, and each component of $\Sigma\setminus \mathbf P$ has a non-abelian fundamental group.

In \Cref{cantor}, inductive constructions (up to a sufficient number of steps) of two surfaces have been given: the surface at the top contains a copy of $\{z\in \Bbb C: |z|\geq 1, z\not\in \Bbb N\times 0\}$, and the bottom is the Cantor tree surface (i.e., the planar surface whose space of ends is homeomorphic to the Cantor set). In each surface, an essential pair of pants $\mathbf P$ is contained in the shaded compact bordered subsurface.\end{proof}

We can prove the following Lemma by a similar argument given the proof of \Cref{applicationofsurjectivitybetweenfundamentalgroups}.

\begin{subsublemma}
 \textup{Let $f\colon \Sigma'\to \Sigma$ be a $\pi_1$-surjective proper map between two non-compact surfaces, where $\Sigma$ has at least six ends. Consider an essential pair of pants $\mathbf P$ in $\Sigma$ given by \Cref{exitenceofessentialpairofpants2}. Then $f^{-1}(\text{int }\mathbf P)\neq \varnothing$ and $f^{-1}(c)\neq \varnothing$ for each component $c$ of $\partial \mathbf P$.}\label{applicationofsurjectivitybetweenfundamentalgroups1}
\end{subsublemma}
The following theorem completes the whole process of finding an essential pair of pants, which will be used to find the degree of a pseudo-proper homotopy equivalence.
\begin{subsubtheorem}
\textup{Let $f\colon \Sigma'\to \Sigma$ be a $\pi_1$-surjective proper map between two non-compact surfaces. Suppose $\Sigma$ is either an infinite-type surface or a finite-type surface $S_{g,0,p}$ with high complexity (i.e., $g+p\geq 4$ or $p\geq 6$). Then $\Sigma$ contains an essential pair of pants $\mathbf P$ such that $f^{-1}(\text{int }\mathbf P)\neq \varnothing$ and $f^{-1}(c)\neq \varnothing$ for each component $c$ of $\partial \mathbf P$. }\label{essentialpairofpants}
\end{subsubtheorem}
\begin{proof}
If an infinite-type surface has a finite genus, then it must have infinitely many ends; see \Cref{infinitetypefinitegenusmeansinfinitelymanyends}. Thus using \Cref{exitenceofessentialpairofpants1}, \Cref{applicationofsurjectivitybetweenfundamentalgroups}, \Cref{exitenceofessentialpairofpants2}, and \Cref{applicationofsurjectivitybetweenfundamentalgroups1}, the proof is complete in all cases, except when $\Sigma$ is homeomorphic to either $S_{1,0,3}$ or $S_{1,0,4}$ or $S_{1,0,5}$. We consider the case when $\Sigma\cong S_{1,0,3}$; the other cases are similar. 

Define an inductive construction of $S_{1,0,3}$ in the following way: Start with a copy of $S_{0,1}$, then consecutively add two copies of $S_{0,3}$, and after that, a copy of $S_{1,2}$; and finally, add three sequences of annuli to obtain three planar ends; see \Cref{figureofinductiveconstruction}. Therefore, in this inductive construction, $K_4$ is obtained from $K_3$, adding a copy $\mathcal S$ of $S_{1,2}$. Let $\mathbf P$ be a smoothly embedded copy of the pair of pants in $\mathcal S$ such that $\mathbf P$ is obtained from decomposing $\mathcal S$ into two copies of the pair of pants and $\mathbf P\cap K_3\neq \varnothing$. Now, an argument similar to that given in \Cref{applicationofsurjectivitybetweenfundamentalgroups} completes the proof.
\end{proof}

At this stage, we need a couple of lemmas. The first one, \Cref{HNN}, is well-known; its proof has been given for reader convenience.
\begin{subsublemma}
\textup{Let $\Sigma$ be a surface, and let $\mathbf S$ be a smoothly embedded bordered sub-surface of $\Sigma$. Then the inclusion induced map $\pi_1(\mathbf S)\to \pi_1(\Sigma)$ is injective if either of the following satisfies:}
\begin{itemize}
    \item[$(1)$] \textup{$\partial\mathbf S$ is a separating primitive circle on $\Sigma$ and $\mathbf S$ is one of the two sides of $\partial \mathbf S$ in $\Sigma$;}
    \item[$(2)$] \textup{$\mathbf S$ is compact and each component of $\partial \mathbf S$ is a primitive circle on $\Sigma$.}
\end{itemize}\label{HNN}
\end{subsublemma}

\begin{proof}[Proof of part (1) of \Cref{HNN}] Since $\pi_1(\Sigma)\cong\pi_1(\mathbf S)*_{\pi_1(\partial\mathbf S)}\pi_1(\Sigma\setminus \text{int }\mathbf S)$ (by Seifert-van Kampen theorem) and the inclusions $\partial\mathbf S\hookrightarrow\mathbf S, \Sigma\setminus \text{int}(\mathbf S)$ are $\pi_1$-injective, we are done. \label{proofHNN1} \end{proof}\begin{proof}[Proof of part (2) of \Cref{HNN}]It is enough to construct a sequence $\Sigma=\mathbf S_0\supseteq \mathbf S_1\supseteq\cdots\supseteq \mathbf S_n=\mathbf S$ of sub-surfaces of $\Sigma$, where $n$ is the number of components of $\partial \mathbf S$, such that for each $k=1,..., n$, the following hold: \begin{itemize}
    \item[$\mathbf{(i)}$] $\mathbf S_k$ is a bordered sub-surface of $\mathbf S_{k-1}$ and the inclusion map $\mathbf S_k\hookrightarrow \mathbf S_{k-1}$ is $\pi_1$-injective;
    \item[$\mathbf{(ii)}$] $\partial \mathbf S_k\setminus \partial \mathbf S_{k-1}$ is either a component of $\partial\mathbf S_{k}$ or union of two components of $\partial\mathbf S_{k}$. In either case, $\partial \mathbf S_k\setminus \partial \mathbf S_{k-1}$ shares only one component with $\partial\mathbf S$.
\end{itemize}We construct this sequence inductively as follows: To construct $\mathbf S_{k}$ from $\mathbf S_{k-1}$, pick a component $\mathbf c$ of $\partial\mathbf S\setminus \partial \mathbf S_{k-1}$. If $\mathbf c$ separates $\mathbf S_{k-1}$, define $\mathbf S_k$ as that side of $\mathbf c$ in $\mathbf S_{k-1}$, which contains $\mathbf S$; then consider an argument similar to the proof of part (1) of \Cref{HNN}. Now, if $\mathbf c$ doesn't separate $\mathbf S_{k-1}$, pick a smoothly embedded annulus $\mathbf A\subset \text{int}(\mathbf S_{k-1})$ such that $\mathbf A\cap\mathbf S=\mathbf c$. Define $\mathbf S_k:=\mathbf S_{k-1}\setminus\text{int}(\mathbf A)$. Now, $\mathbf S_{k-1}$ is obtained from $\mathbf S_k$ identifying $\mathbf c$ with $\partial \mathbf A\setminus \mathbf c$ by an orientation-reversing diffeomorphism $\varphi\colon \mathbf c\to\partial \mathbf A\setminus \mathbf c$. By HNN-Seifert-van Kampen theorem, $\pi_1(\mathbf S_{k-1})\cong \pi_1(\mathbf S_k)*_{\pi_1(\varphi)}$; where the map $\pi_1(\mathbf S_k)\to \pi_1(\mathbf S_{k-1})$ (which is inclusion induced) is injective due to Britton's lemma. This completes the proof.\end{proof}

The following lemma roughly says that the degree of a map between two compact bordered surfaces can be determined from the degree of its restriction on the boundaries.

\begin{subsublemma}
\textup{Let $\varphi\colon \mathbf S_{g_1,b_1}\to  \mathbf S_{g_2,b_2}$ be a map between two compact bordered surfaces. If $\varphi\vert\partial \mathbf S_{g_1,b_1}\hookrightarrow \partial  \mathbf S_{g_2,b_2}$ is an embedding, then $\varphi(\partial \mathbf S_{g_1,b_1})= \partial \mathbf S_{g_2,b_2}$ and $\deg(\varphi)=\pm 1$.}\label{homeomorphismoboundarymensdegreeone}
\end{subsublemma}
 \begin{proof}Notice that $\varphi$ maps each component of $\partial \mathbf S_{g_1,b_1}$ homeomorphically onto a component of $\partial \mathbf S_{g_2,b_2}$, and any two distinct components of $\partial \mathbf S_{g_1,b_1}$ have distinct $\varphi$-images. Now, naturality of homology long exact sequences of $\left(\mathbf S_{g_1,b_1}, \partial \mathbf S_{g_1,b_1}\right)$ and $\left(\mathbf S_{g_2,b_2}, \partial \mathbf S_{g_2,b_2}\right)$ give following commutative diagram:
 $$\begin{tikzcd}
{H_2\left(\mathbf S_{g_1,b_1}, \partial \mathbf S_{g_1,b_1}\right)\cong \Bbb Z} \arrow[d, "\times \deg(\varphi)"'] \arrow[rr, "{1\longmapsto \overset{b_1}{\oplus} 1}"] &  & { \overset{b_1}{\oplus}\Bbb Z\cong H_1\left(\partial \mathbf S_{g_1,b_1}\right)} \arrow[d, "{\overset{b_1}{\oplus} 1\longmapsto \overset{b_1}{\oplus} (\pm 1)\bigoplus \overset{b_2-b_1}{\oplus} 0}"] \\
{ H_2\left(\mathbf S_{g_2,b_2}, \partial \mathbf S_{g_2,b_2}\right)\cong \Bbb Z} \arrow[rr, "{ 1\longmapsto \overset{b_2}{\oplus}1}"]                                   &  & { \overset{b_2}{\oplus}\Bbb Z\cong H_1\left(\partial \mathbf S_{g_2,b_2}\right)}                                                     
\end{tikzcd}$$
The horizontal maps are the connecting homomorphisms for homology long exact sequences, and for their description, see \cite[Exercise 31 of Section 3.3]{MR1867354}. Now, commutativity of this diagram gives $b_2=b_1$ \big(the integer $\deg(\varphi)$ can't be simultaneously $0$ as well as $\pm 1$\big), and thus $\deg(\varphi)=\pm 1$.
\end{proof}

The proof of \Cref{rigidityofpairofpants} below can be found in \cite[Theorem 4.6.2.]{scottbook}. It also follows from the much more general result, \cite[Theorem 3.1.]{MR541331}. Since compact bordered surfaces are aspherical, an application of the Whitehead theorem says that the assumption ``$\varphi\colon \mathbf S'\to \mathbf S$ is a homotopy equivalence'' in  \Cref{rigidityofpairofpants} is equivalent to the assumption ``$\pi_1(\varphi)$ is an isomorphism''.
 \begin{subsubtheorem}{\textup{(Rigidity of compact bordered surfaces)}} \textup{Let $\varphi\colon \mathbf S'\to \mathbf S$ be a homotopy equivalence between two compact bordered surfaces such that $\varphi^{-1}(\partial \mathbf S)=\partial\mathbf S'$. If $\varphi\vert\partial\mathbf S'\to \partial\mathbf S$ is a homeomorphism, then $\varphi$ is homotopic to a homeomorphism relative to $\partial \mathbf S'$. } \label{rigidityofpairofpants}\end{subsubtheorem}

The following lemma gives some sufficient conditions so that the pre-image of a compact bordered subsurface under a proper map becomes a compact bordered subsurface of the same homeomorphism type. Its usage is two-fold: firstly, in \Cref{cl1}, to find the degree of a pseudo proper homotopy equivalence; and secondly, in the proof of \Cref{MR1}.

\begin{subsublemma}
\textup{Let $f\colon \Sigma'\to \Sigma$ be a $\pi_1$-injective proper map between two non-compact oriented surfaces, and let $\mathbf S$ be a smoothly embedded compact bordered subsurface of $\Sigma$ with $f^{-1}(\textup{int }\mathbf S)\neq\varnothing$. Suppose $f^{-1}(\partial \mathbf S)$ is a pairwise disjoint collection of smoothly embedded primitive circles on $\Sigma'$ such that $f$ sends $f^{-1}(\partial \mathbf S)$ homeomorphically onto $\partial \mathbf S$. Then $f^{-1}(\mathbf S)$ is a copy $\mathbf S$ in $\Sigma'$ with $\partial f^{-1}(\mathbf S)=f^{-1}(\partial \mathbf S)$, and $\deg(f)=\pm 1$.} \label{inversepairofpants} 
\end{subsublemma}
\begin{proof}
Since $f^{-1}(\textup{int }\mathbf S)\neq\varnothing$ and $f$ is proper, the continuity of $f\big|\Sigma'\setminus f^{-1}(\partial \mathbf S)\to \Sigma\setminus \partial \mathbf S$ tells that  $\Sigma'\setminus f^{-1}(\partial \mathbf S)$ is disconnected. Let $\mathbf S'\subset \Sigma'$ be a bordered sub-surface obtained as a complementary component of the decomposition of $\Sigma'$ by $f^{-1}(\partial \mathbf S)$ such that $f(\mathbf S')\subseteq \mathbf S$. That is, $\mathbf S'$ is the closure of one of the components of $\Sigma'\setminus f^{-1}(\partial \mathbf S)$ and $\mathbf S'$ is contained in the compact set $f^{-1}(\mathbf S)$. So, $\mathbf S'$ is a compact bordered subsurface of $\Sigma'$, and each component of $\partial \mathbf S'$ is a component of $f^{-1}(\partial \mathbf S)$. In the following few lines, we will show that each component of $f^{-1}(\partial \mathbf S)$ is also a component of $\partial \mathbf S'$. Anyway, since $f\vert f^{-1}(\partial \mathbf S)\to \partial \mathbf S$ is a homeomorphism, we can say that $f\vert \partial \mathbf S'\hookrightarrow \partial \mathbf S$ is an embedding. Now, by \Cref{homeomorphismoboundarymensdegreeone}, $\partial \mathbf S'=f^{-1}(\partial \mathbf S)$ and $\deg\left(f\vert\mathbf S'\to\mathbf  S\right)=\pm 1$.  Next, by \Cref{degreeonemapsarepi1surjective},  $f\vert\mathbf S'\to\mathbf  S$ is $\pi_1$-surjective. Since the inclusion $\mathbf S'\hookrightarrow \Sigma'$ and $f$ are $\pi_1$-injective, $f\vert\mathbf S'\to\mathbf  S$ is also so \big(see part (2) of \Cref{HNN}\big). Thus, $f\vert\mathbf S'\to\mathbf  S$ is $\pi_1$-bijective, and so  \Cref{rigidityofpairofpants} tells that $\mathbf S'\cong \mathbf S$. Finally, if  $\mathbf S''$ is another bordered sub-surface obtained as a complementary component of decomposition of $\Sigma'$ by $f^{-1}(\partial \mathbf S)$ with $f(\mathbf S'')\subseteq \mathbf S$, then similarly, $\mathbf S''\cong \mathbf S$. Since $f\vert f^{-1}(\partial\mathbf S)\to\partial\mathbf  S$ is a homeomorphism and $\Sigma'$ is connected, $\mathbf S''=\mathbf S'$ (otherwise, $\displaystyle \Sigma'$ would be the compact surface $\mathbf S'\cup_{\partial \mathbf S'=\partial \mathbf S''} \mathbf S''$). Therefore, $f^{-1}(\mathbf S)=\mathbf S'\cong \mathbf S$, and thus the proof of the first part is completed.
 
Now, we will prove that $\deg(f)=\pm 1$. Since $\deg(f)$ remains invariant after any proper homotopy of $f$, we can properly homotope $f$ as we want. So, apply \Cref{rigidityofpairofpants} to $f\vert\mathbf S'\to\mathbf  S$. Thus, $f\colon \Sigma'\to \Sigma$ can be properly homotoped relative to $\Sigma'\setminus\text{int}(\mathbf S')$ to map $\mathbf S'=f^{-1}(\mathbf S)$ is homeomorphically onto $\mathbf S$. Now, by \Cref{degreeonemapchecking}, $\deg(f)=\pm 1$.\end{proof}
 
We are now ready to prove that a pseudo proper homotopy equivalence is a map of degree $\pm 1$ if the co-domain contains an essential pair of pants, as said before.
\begin{subsubtheorem}
\textup{Let $f\colon \Sigma'\to \Sigma$ be a pseudo proper homotopy equivalence between two non-compact oriented surfaces, where $\Sigma$ is either an infinite-type surface or a finite-type non-compact surface $S_{g,0, p}$ with high complexity (to us, high complexity means $g+p\geq 4$ or $p\geq 6$). Then $\deg(f)=\pm 1$.}\label{cl1}
\end{subsubtheorem}

\begin{proof}Since $\deg(f)$ remains invariant after any proper homotopy of $f$, we can properly homotope $f$ as we want.
Now, \Cref{essentialpairofpants} gives an essential pair of pants $\mathbf P$ in $\Sigma$ such that $f^{-1}\big(\text{int}(\mathbf P)\big)\neq \varnothing$ and $f^{-1}(c)\neq \varnothing$ for each component $c$ of $\partial \mathbf P$, even after any proper homotopy of $f$. Using \Cref{transtoLFCS} and then \Cref{annulusremovalfinal}, after a proper homotopy, we may assume that  $f^{-1}(\text{int }\mathbf P)\neq \varnothing$ and $f^{-1}(\partial \mathbf P)$ is a pairwise-disjoint collection of three smoothly embedded circles on $\Sigma$ such that $f\vert f^{-1}(\partial \mathbf P)\to  \partial \mathbf P$ is a homeomorphism. 

Now, if possible, let $c'$ be a component of $f^{-1}(\partial \mathbf P)$ such that there is an embedding $i'\colon \Bbb D^2\hookrightarrow \Sigma'$ with $c'=i'(\Bbb S^1)$.  Then the embedding $f\circ i'\vert\Bbb S^1\hookrightarrow \Sigma$ is null-homotopic and $c\coloneqq f\circ i'(\Bbb S^1)$ is a component of $\partial \mathbf P$. But $\mathbf P$ is an essential pair of pants in $\Sigma$ implies each component of  $\partial \mathbf P$ is a primitive circle on $\Sigma$. Now, \Cref{primitivecircle} tells us we have reached a contradiction.  Hence, each component of $f^{-1}(\partial \mathbf P)$ is a primitive circle on $\Sigma'$. Finally, applying \Cref{inversepairofpants}, we complete the proof.
\end{proof}

\subsubsection{An essential punctured disk of a proper map and the degree of a pseudo proper homotopy equivalence}\label{essentialpuncureddiskanddegree}

We first build up notations for \cref{essentialpuncureddiskanddegree}. Let $\Sigma$ be a non-compact surface. Since the $\text{Ends}(\Sigma)$ is independent of the choice of efficient exhaustion of $\Sigma$ by compacta, we will use Goldman's inductive construction to define $\text{Ends}(\Sigma)$; see \Cref{ends}. So, consider an inductive construction of $\Sigma$. For each $i\geq 1$, define $K_i$ to be the compact bordered subsurface of $\Sigma$ after the $i$-th step of the induction. Then $\{K_i\}_{i=1}^\infty$ is an efficient exhaustion of $\Sigma$ by compacta. Also, notice that $\text{int}(K_1)\subseteq \text{int}(K_2)\subseteq \cdots$ is an increasing sequence of open subsets of $\Sigma$ such that $\cup_{i=1}^\infty\text{int}(K_i)=\Sigma$; and thus every compact subset of $\Sigma$ is contained in some $\text{int}(K_i)$.

Suppose $\Sigma'$ is another non-compact surface and $f\colon \Sigma'\to \Sigma$ is a proper map. Let $(V_1, V_2,...)$ be an end of $\Sigma$, i.e., $V_i$ is a component of $\Sigma\setminus K_i$ and $V_1\supseteq V_2\supseteq \cdots$. With this setup, we are now ready to state a lemma that is more or less related to \Cref{endssandproperhomotopy}.
\begin{subsubtheorem}
\textup{Assume that $f^{-1}(V_i)\neq \varnothing$ for each $i\geq 1$. Then for every proper map $g\colon \Sigma'\to \Sigma$, which is properly homotopic to $f$, we have $g^{-1}(V_i)\neq \varnothing$ for each $i\geq 1$.} \label{surjectivityonendpreservesbyproperhomotopy}
\end{subsubtheorem}
\begin{proof}
Let $g\colon \Sigma'\to \Sigma$ be a proper map, and let $\mathcal H\colon \Sigma'\times [0,1]\to \Sigma$ be a proper homotopy from $f$ to $g$.  Notice that $V_i\to \infty$: If $\mathcal X$ is a compact subset of $\Sigma$, then $\mathcal X\subseteq \text{int}(K_{i_0})$ for some positive integer $i_0$, i.e., $\mathcal X\cap V_i=\varnothing$ for all $i\geq i_0$. Therefore, $f^{-1}(V_i)\to \infty$: If $\mathcal X'$ is a compact subset of $\Sigma'$, then $f(\mathcal X')$ is compact, so $f(\mathcal X')\cap V_i=\varnothing$ for all but finitely many $i$, i.e., $\mathcal X'\cap f^{-1}(V_i)=\varnothing$ for all but finitely many $i$.

Let $n$ be any positive integer. Consider the compact subset $p\big(\mathcal H^{-1}( K_{n})\big)$ of $\Sigma'$, where $p\colon \Sigma'\times [0,1]\to \Sigma'$ is the projection. Since $f^{-1}(V_i)\to \infty$, we have an integer $i_n>n$ such that $f^{-1}(V_{i_n})\subseteq \Sigma'\setminus p\big(\mathcal H^{-1}( K_{n})\big)$. Now, consider any $x_{i_n}\in f^{-1}(V_{i_n})$. Then $\mathcal H\big(x_{i_n}\times[0,1]\big)\subseteq \Sigma\setminus K_{n}$, i.e., the connected set $\mathcal H\big(x_{i_n}\times[0,1]\big)$ is contained in one of the components of $\Sigma\setminus K_{n}$. But $\mathcal H(x_{i_n},0)=f(x_{i_n})\in V_{i_n}\subseteq V_n$, i.e., $\mathcal H\big(x_{i_n}\times[0,1]\big)\subseteq V_n$. In particular, this means $g(x_{i_n})=\mathcal H(x_{i_n},1)\in V_n$. Since $n$ is an arbitrary positive integer, we are done.
\end{proof}

\begin{subsubdefinition}
\textup{Let $e=(V_1, V_2,...)$ be an end of $\Sigma$ such that for some non-negative integer $i_e$, $\overline{V_i}\cong S_{0,1,1}$ for all $i\geq i_e$ (i.e., $e$ is an isolated planar end of $\Sigma$).  If $f^{-1}(V_i)\neq \varnothing$ for all $i\geq 1$, then for each integer $i\geq i_e$, we say $\overline{V_i}$ is an \emph{essential punctured disk} of $f$.} 
\end{subsubdefinition}

\Cref{surjectivityonendpreservesbyproperhomotopy} says that the notion of an essential punctured disk is invariant under the proper homotopy. In \Cref{cl2}, we show that, after a proper homotopy, the pre-image of the boundary of an essential punctured disk under a pseudo proper homotopy equivalence bounds a planar end of the domain. But before moving into its proof, we need to prove the following lemma, which gives some sufficient conditions so that the pre-image of a punctured disk in the co-domain under a proper map becomes a punctured disk in the domain.

\begin{subsublemma}
\textup{Let $f\colon \Sigma'\to \Sigma$ be a $\pi_1$-injective proper map between two non-compact oriented surfaces, and let $\mathcal C$ be a smoothly embedded separating circle on $\Sigma$ such that one of the two sides of $\mathcal C$ in $\Sigma$ is a punctured disk $\mathcal D_*$. Also, let $\Sigma'$ is  homeomorphic to neither $\Bbb S^1\times \Bbb R$ nor $\Bbb R^2$. If $f^{-1}(\mathcal C)$ is a smoothly embedded primitive circle on $\Sigma'$ so that $f\vert f^{-1}(\mathcal C)\to \mathcal C$ is a homeomorphism and $f^{-1}(\textup{int }\mathcal D_*)\neq\varnothing$, then $f^{-1}(\mathcal D_*)$ is a copy the  punctured disk in $\Sigma'$ with $\partial f^{-1}(\mathcal D_*)=f^{-1}(\mathcal C)$ and $\deg(f)=\pm 1$.} \label{inversepunctureddisk}
\end{subsublemma}
\begin{proof}Notice that $\Sigma'\not \cong \Bbb R^2,\Bbb S^1\times \Bbb R$, i.e., $\pi_1(\Sigma')$ is non-abelian by \Cref{finitetypenoncompactsurace}. Since $f^{-1}(\textup{int }\mathcal D_*)\neq\varnothing$ and  $\pi_1(f)\big(\pi_1(\Sigma')\big)$ is non-abelian, by continuity of $f\vert\Sigma'\setminus f^{-1}(\mathcal C)\to \Sigma\setminus \mathcal C$, we can say that $\Sigma'\setminus f^{-1}(\mathcal C)$ is disconnected. Let $\mathbf S'$ be a side of $f^{-1}(\mathcal C)$ in $\Sigma'$ for which $f(\mathbf S')\subseteq \mathcal D_*$. Since $f$ is $\pi_1$-injective, by  part $(1)$ of \Cref{HNN}, $f\vert\mathbf S'\to \mathcal D_*$ is also so. Thus, $\pi_1(\mathbf S')$ is a subgroup of $\Bbb Z$. Now, $\text{int}(\mathbf S')$ is homotopy equivalent to $\mathbf S'$ and bounded by the primitive circle $f^{-1}(\mathcal C)$ on $\Sigma'$; so, using \Cref{finitetypenoncompactsurace}, $\mathbf S'\cong S_{0,1,1}$. Next, if $\mathbf S''$ is another side of $f^{-1}(\mathcal C)$ in $\Sigma'$ for which $f(\mathbf S'')\subseteq \mathcal D_*$, then similarly, $\mathbf S''\cong S_{0,1,1}$. Since $f\vert f^{-1}(\mathcal C)\to \mathcal C$ is homeomorphism and $\Sigma'$ is connected, $\mathbf S''=\mathbf S'$; otherwise, $\displaystyle \Sigma'$ would be $\mathbf S'\cup_{\mathcal C} \mathbf S''\cong \Bbb S^1\times \Bbb R$. Therefore, $f^{-1}(\mathcal D_*)=\mathbf S'\cong \mathcal D_*$, and thus the proof of the first part is completed.

Now, we will prove that $\deg(f)=\pm 1$. Since $\deg(f)$ remains invariant after any proper homotopy of $f$, we can properly homotope $f$ as we want. So, apply \Cref{Alexandertrick}  to $f\vert\mathbf S'\to \mathcal D_*$. Thus, $f\colon \Sigma'\to \Sigma$ can be properly homotoped relative to $\Sigma'\setminus\text{int}(\mathbf S')$ to map $\mathbf S'=f^{-1}(\mathcal D_*)$ is homeomorphically onto $\mathcal D_*$.  Now, by \Cref{degreeonemapchecking}, $\deg(f)=\pm 1$.
\end{proof}

Now, we prove a well-known theorem used in the previous lemma.

 \begin{subsubtheorem}{\textup{(Proper rigidity of the punctured disk)}} \textup{Let $\mathbf D_*$ be a punctured disk, and let $\varphi\colon \mathbf D_*\to \mathbf D_*$ be a proper map such that $\varphi^{-1}(\partial \mathbf D_*)=\partial \mathbf D_*$ and $\varphi\vert\partial \mathbf D_*\to \partial \mathbf D_*$ is a homeomorphism. Then $\varphi$ is properly homotopic to a homeomorphism $\mathbf D_*\to \mathbf D_*$ relative to the boundary $\partial \mathbf D_*$. }\label{Alexandertrick}\end{subsubtheorem}
\begin{proof}
  Without loss of generality, we may assume $\mathbf D_*=\{z\in \Bbb C: 0<|z|\leq 1\}$. Define $\mathcal H\colon \mathbf D_*\times [0,1]\to \mathbf D_*$  by $$\mathcal H(z,t)\coloneqq \begin{cases}(1-t)\cdot \varphi\left(\frac{z}{1-t}\right) & \text{ if }0<|z|\leq 1-t,\\  |z|\cdot \varphi\left(\frac{z}{|z|}\right) & \text{ if }1-t<|z|\leq 1.
      
  \end{cases}$$ Notice that $\varphi\simeq \mathcal H(-,1)$ relative to $\partial \mathbf D_*$, and $\mathcal H(-,1)\colon \mathbf D_*\to \mathbf D_*$ is a homeomorphism. 

Now, we prove $\mathcal H$ is a proper map. So let $\{(z_n,t_n)\}$ is a sequence in $\mathbf D_*\times [0,1]$ with $z_n\to 0$. We need to show that $\mathcal H(z_n,t_n)\to 0$. Define $\mathscr A\coloneqq\big \{n\in \Bbb N: 1-t_n<|z_n|\big\}$ and $\mathscr B\coloneqq \big \{n\in \Bbb N: |z_n|\leq 1-t_n\big\}$. Then $\Bbb N=\mathscr A\cup \mathscr B$. Therefore, it is enough to show $\{\mathcal H(z_n,t_n):n\in \mathscr A\}\to 0$ \big(resp. $\{\mathcal H(z_n,t_n):n\in \mathscr B\}\to 0$\big) whenever $\mathscr A$ (resp. $\mathscr B$) is infinite.  

If $\mathscr A$ is infinite, then $\{\mathcal H(z_n,t_n):n\in \mathscr A\}\to 0$, since $\big|\mathcal H(z_n,t_n)\big|=|z_n|\cdot \left|\varphi\left(\frac{z_n}{|z_n|}\right)\right|\leq |z_n|$ for all $n\in \mathscr A$. 

Next, assume $\mathscr B$ is infinite. We will prove that $\{\mathcal H(z_n,t_n):n\in \mathscr B\}\to 0$. So consider any $\varepsilon>0$. We need to show $\big|\mathcal H(z_n,t_n)\big|<\varepsilon$ for all but finitely many $n\in \mathscr B$. Let $\mathscr B'_\varepsilon\coloneqq \{n\in \mathscr B:1-t_n<\varepsilon\}$. Therefore $\big|\mathcal H(z_n,t_n)\big|=(1-t_n)\cdot \left|\varphi\left(\frac{z_n}{1-t_n}\right)\right|\leq (1-t_n)<\varepsilon$ for all $n\in \mathscr B'_\varepsilon$. Also, if $\mathscr B\setminus \mathscr B'_\varepsilon$ is infinite, then $\left\{\frac{z_n}{1-t_n}:n\in \mathscr B\setminus \mathscr B'_\varepsilon\right\}\to 0$, which implies $\left\{\varphi\left(\frac{z_n}{1-t_n}\right):n\in \mathscr B\setminus \mathscr B'_\varepsilon\right\}\to 0$ (as $\varphi$ is proper), and thus $\big|\mathcal H(z_n,t_n)\big|\leq \left|\varphi\left(\frac{z_n}{1-t_n}\right)\right|<\varepsilon$ for all but finitely many $n\in \mathscr B\setminus \mathscr B'_\varepsilon$. Now, the previous two lines together imply that  $\big|\mathcal H(z_n,t_n)\big|<\varepsilon$ for all but finitely many $n\in \mathscr B$. 
\end{proof}

\begin{subsubremark}
    \textup{Notice that \Cref{Alexandertrick} is obtained from a straightforward modification of the Alexander trick \cite[Lemma 2.1]{MR2850125}.}
\end{subsubremark}

\begin{subsubtheorem}
\textup{Let $f\colon \Sigma'\to \Sigma$ be a pseudo proper homotopy equivalence between two non-compact oriented surfaces. Suppose $\pi_1(\Sigma)$ is finitely-generated non-abelian group \big(equivalently $\Sigma\cong S_{g, 0,p}$  for some $(g,p)\neq (0,1), (0,2)$\big). Then $\deg(f)=\pm 1$.}\label{cl2}
\end{subsubtheorem}

\begin{proof}
Since $\deg(f)$ remains invariant after any proper homotopy of $f$, we can properly homotope $f$ as we want. Now, $\Sigma$ is a finite-type non-compact surface implies each end of $\Sigma$ is an isolated planar end, i.e., for every $e=(V_1, V_2,...)\in \text{Ends}(\Sigma)$, we have an integer $i_e$ such that  $\overline{V_i}$ is homeomorphic to the punctured disk for each $i\geq i_e$. Next, $f$ is proper implies there exists  $\mathcal E=(\mathcal W_{1}, \mathcal W_{2},...)\in \text{Ends}(\Sigma)$ such that $f^{-1}(\mathcal W_{i})\neq \varnothing$ for each $i\geq 1$. Notice that $\overline{\mathcal W_{i_{\mathcal E}}}$ is an essential punctured disk and $\mathcal C_{i_{\mathcal E}}\coloneqq \partial\overline{\mathcal W_{i_{\mathcal E}}}$ is a smoothly embedded circle separating on $\Sigma$. Also, $\mathcal C_{i_{\mathcal E}}$  is a primitive circle on $\Sigma'$ as $\mathcal C_{i_{\mathcal E}}$ bound the punctured disk $\overline{\mathcal W_{i_{\mathcal E}}}$ on $\Sigma'\not \cong \Bbb R^2$.

We aim to use \Cref{inversepunctureddisk}, but some observations are needed before that. Let $g\colon \Sigma'\to \Sigma$ be a proper map such that $g$ is properly homotopic to $f$ (note that $f$ is properly homotopic to itself, i.e., $g$ can be $f$). If possible, assume $g^{-1}(\mathcal C_{i_{\mathcal E}})=\varnothing$. Then continuity of $g$ implies $g(\Sigma')$ is contained in one of the two components of $\Sigma\setminus \mathcal C_{i_{\mathcal E}}$. By \Cref{surjectivityonendpreservesbyproperhomotopy}, $g(\Sigma')$ must be contained in $\mathcal W_{i_{\mathcal E}}$. But, then $\pi_1(f)=\pi_1(g)$ is non-surjective as $\pi_1(\Sigma\setminus\mathcal W_{i_{\mathcal E}})=\pi_1(\Sigma)$ is non-abelian. Therefore, $g^{-1}(\mathcal C_{i_{\mathcal E}})\neq\varnothing$. Also, by \Cref{surjectivityonendpreservesbyproperhomotopy}, $g^{-1}(\mathcal W_{i})\neq \varnothing$ for each $i\geq 1$, and thus  $g^{-1}(\mathcal W_{i_{\mathcal E}})\neq \varnothing$.

Now, we are ready to apply \Cref{inversepunctureddisk} after the observation given in the previous paragraph. At first, notice that $\Sigma'$ is homeomorphic to neither the plane nor the punctured plane as $\pi_1(\Sigma')=\pi_1(\Sigma)$ is non-abelian.  After a proper homotopy of $f$, we may assume that $f\ \stackinset{c}{}{c}{0.1ex}{$\top$}{$\abxpitchfork$}\ \mathcal C_{i_{\mathcal E}}$; see \Cref{transtoLFCS}. By the previous paragraph, $f^{-1}(\mathcal C_{i_{\mathcal E}})$  is a pairwise disjoint non-empty collection of finitely many smoothly embedded circles on $\Sigma'$. Now, by \Cref{annulusremovalfinal} and the previous paragraph, after a proper homotopy of $f$, we may further assume that $\mathcal C_{i_{\mathcal E}}'\coloneqq f^{-1}(\mathcal C_{i_{\mathcal E}})$ is a (single) smoothly embedded circle on $\Sigma'$ and  $f\vert \mathcal C_{i_{\mathcal E}}'\to \mathcal C_{i_{\mathcal E}}$ is a homeomorphism. The previous paragraph also tells that after all these proper homotopies, $f^{-1}(\mathcal W_{i_{\mathcal E}})$ remains non-empty.

We show that $\mathcal C_{i_{\mathcal E}}'$ is a primitive circle on $\Sigma'$. On the contrary, let there be an embedding $i'\colon \Bbb D^2\hookrightarrow \Sigma'$ with $\mathcal C_{i_{\mathcal E}}'=i'(\Bbb S^1)$.  Then the embedding $f\circ i'\vert\Bbb S^1\hookrightarrow \Sigma$ is null-homotopic and $\mathcal C_{i_{\mathcal E}}=f\circ i'(\Bbb S^1)$. But $\mathcal C_{i_{\mathcal E}}$ is a primitive circle on $\Sigma$. Now, \Cref{primitivecircle} tells us we have reached a contradiction. Finally, applying \Cref{inversepunctureddisk}, we can say that $\deg(f)=\pm 1$.
\end{proof}

\subsubsection{Most pseudo proper homotopy equivalences between non-compact surfaces are maps of degree  \texorpdfstring{$\pm 1$}{±1}}
\begin{subsubtheorem}
\textup{Let $f\colon \Sigma'\to \Sigma$ be a pseudo proper homotopy equivalence between two non-compact oriented surfaces. If $\Sigma\not\cong \Bbb S^1\times \Bbb R, \Bbb R^2$ (equivalently $\Sigma'\not\cong \Bbb S^1\times \Bbb R, \Bbb R^2$), then $\deg(f)=\pm 1$.} \label{properplushomotopyequivalnceisofdegreeone}
\end{subsubtheorem}

 \begin{proof}Combining \Cref{cl1} and \Cref{cl2}, we complete the proof.
 \end{proof}
 
The following proposition, \emph{which we don't need to use anywhere}, says that if either of the integers $1$ and $-1$ appears as the degree of a pseudo proper homotopy equivalence between two non-compact oriented surfaces, then the other also appears.

 \begin{subsubproposition}
 \textup{Let $f\colon \Sigma'\to \Sigma$ be a pseudo proper homotopy equivalence between two non-compact oriented surfaces. Then there exists another pseudo proper homotopy equivalence $\overline f\colon \Sigma'\to \Sigma$ such that $\deg(\overline f)=-\deg(f)$.}
 \end{subsubproposition}
 \begin{proof}
 Write $\Sigma$ as the double of a bordered surface $\mathbf S$; see \Cref{richard1}. Define a homeomorphism $\varphi\colon \Sigma\to \Sigma$ by sending $[p,t]\in \Sigma$ to $[p,1-t]\in \Sigma$ for all $(p,t)\in \mathbf S\times \{0,1\}$.  Then $\varphi$ is an orientation-reversing homeomorphism. Therefore, the degree of $\overline f\coloneqq \varphi\circ f$ is $-\deg(f)$ as the degree is multiplicative; see \Cref{degreeofapropermap}. 
 \end{proof}
 
 \subsubsection{An application of the non-vanishing degree of a pseudo proper homotopy equivalence}
Consider a non-surjective map $\varphi\colon \mathscr M\to \mathscr N$ between two closed, oriented, connected $n$-manifolds. Then for any $p\in \mathscr N\setminus \text{im}(\varphi)$,  the map $H^n(\varphi)$ factors through the inclusion-induced zero map $H^n(\mathscr N) \cong \Bbb Z \longrightarrow 0\cong H^n(\mathscr N\setminus p)$ (recall that top integral singular cohomology of any connected, non-compact, boundaryless manifold is zero), i.e., $\deg(\varphi)=0$. The lemma below generalizes this phenomenon in the proper category.

\begin{subsublemma}
 \textup{Let $\Phi\colon M\to N$ be a proper map between two connected, oriented, boundaryless, smooth $k$-dimensional manifolds. If $\deg(\Phi)\neq 0$, then $\Phi$ is surjective.} \label{non-surjectivepropermaphasdegreezero}
\end{subsublemma}

\begin{proof}Being a proper map between two manifolds, $\Phi$ is a closed map; see \cite{MR254818}. Now, if possible, let $\Phi$ be non-surjective. Therefore, $N\setminus \Phi(M)$ is a non-empty open subset of $N$. Pick a point $y\in N\setminus \Phi(M)$. Since $N$ is locally Euclidean, there is a smoothly embedded closed ball $B\subset N$ such that $B\subseteq N\setminus \Phi(M)$. Notice that $N\setminus \text{int}(B)$ is a smoothly embedded co-dimension zero submanifold of $N$ with $\partial\big(N\setminus \text{int}(B)\big)=\partial B$. By Poincaré duality \big(see \cite[Exercise 35 of Section 3.3]{MR1867354}\big), $H^k_\textbf{c}\big(N\backslash\textup{int}(B);\Bbb Z\big)\cong H_0\big(N\backslash\textup{int}(B), \partial B;\Bbb Z\big)$. Also,  $H_0\big(N\backslash\textup{int}(B), \partial B;\Bbb Z\big)=0$ as $N$ is path-connected; see \cite[Exercise 16.(a) of Section 2.1]{MR1867354}. Now, $\Phi\colon M\to N$ can be thought as the composition $M\xrightarrow{\Phi^\dag} N\setminus \text{int}(B)\xhookrightarrow{i} N$, where $i$ is the inclusion map and $\Phi^\dag(m)\coloneqq \Phi(m)$ for all $m\in M$. Certainly, $\Phi^\dag$ and $i$ are both proper maps. Therefore, $H_\textbf{c}^k(\Phi)$ is the composition $$H^k_\textbf{c}(N;\Bbb Z)\xrightarrow{H^k_\textbf{c}(i)}H^k_\textbf{c}\big(N\backslash\textup{int}(B);\Bbb Z\big)=0\xrightarrow{H^k_\textbf{c}\left(\Phi^\dag\right)}H^k_\textbf{c}(M;\Bbb Z),$$ i.e., $H_\textbf{c}^k(\Phi)=0$, which contradicts $\deg(\Phi)\neq0$. Thus, $\Phi$ must be a surjective map.
\end{proof}The above lemma, together with  \Cref{properplushomotopyequivalnceisofdegreeone}, gives the following corollary.
\begin{subsubcorollary}
\textup{A pseudo proper homotopy equivalence between two non-compact surfaces is a surjective map, provided surfaces are homeomorphic to neither the plane nor the punctured plane.}\label{properplushomotopyequivalnceissurjective}
\end{subsubcorollary}

The following lemma tells that one way to achieve the surjectivity throughout a proper homotopy is to assume that the initial map of this proper homotopy is a map of non-zero degree. Note that any proper map $f\colon X\to Y$ is properly homotopic to itself due to the proper homotopy $X\times [0,1]\ni (x,t)\longmapsto f(x)\in Y$.
\begin{subsublemma}
 \textup{Let $\Phi\colon M\to N$ be a proper map of non-zero degree between two connected, oriented, boundaryless, smooth $k$-dimensional manifolds, and let $\Psi\colon M\to N$ be a proper map such that $\Psi$ is properly homotopic to $\Phi$. Then $\Psi$ is a surjective map.} \label{properhomotopypreservessurjectivity}
\end{subsublemma}
\begin{proof}Since $\Psi$ is properly homotopic to $\Phi$, $\deg(\Psi)=\deg(\Phi)\neq0$; see \Cref{degreeofapropermap}. Now, to conclude, consider \Cref{non-surjectivepropermaphasdegreezero}.
\end{proof}

 Here is the main application of the non-vanishing degree of a pseudo proper homotopy equivalence.

 \begin{subsubtheorem}
\textup{Let $f\colon \Sigma'\to \Sigma$ be a smooth pseudo proper homotopy equivalence between two non-compact surfaces, where $\Bbb S^1\times \Bbb R\not\cong \Sigma\not\cong \Bbb R^2$; and let  $\mathscr A$ be a preferred \textup{LFCS} on $\Sigma$ such that $f\ \stackinset{c}{}{c}{0.1ex}{$\top$}{$\abxpitchfork$}\ \mathscr A$. Suppose any two distinct components of $\mathscr A$ don't co-bound an annulus in $\Sigma$. In that case, $f$ can be properly homotoped to a proper map $g$ such that for each component $\mathcal C$ of $\mathscr A$, $g^{-1}(\mathcal C)$ is a component of $f^{-1}(\mathscr A)$ that is mapped homeomorphically onto $\mathcal C$ by $g$.} \label{annulusremovalfinal2}
\end{subsubtheorem}
 \begin{proof}
\Cref{annulusremovalfinal} gives a proper map $g\colon \Sigma'\to \Sigma$ such that the following hold: $(1)$ $g$ is properly homotopic to $f$, and  $(2)$ for each component $\mathcal C$ of $\mathscr A$, if $g^{-1}(\mathcal C)\neq \varnothing$, then  $g^{-1}(\mathcal C)$ is a component of $f^{-1}(\mathscr A)$ such that $g\vert g^{-1}(\mathcal C)\to \mathcal C$ is homeomorphism. But $\deg(f)=\pm 1$, by \Cref{properplushomotopyequivalnceisofdegreeone}. Thus, the map $g$ is surjective since it is properly homotopic to the non-zero degree map $f$; see  \Cref{properhomotopypreservessurjectivity}. So, for each component $\mathcal C$ of $\mathscr A$, $g^{-1}(\mathcal C)$ is a component of $f^{-1}(\mathscr A)$ such that $g\vert g^{-1}(\mathcal C)\to \mathcal C$ is homeomorphism.
 \end{proof}
 \begin{subsubremark}
 \textup{For closed surfaces, the analog of \Cref{annulusremovalfinal2} can be stated far before, exactly in the ``annulus removal'' section, as every homotopy equivalence between two closed manifolds has a homotopy inverse, hence is a map of degree $\pm 1$, and hence is surjective. But, before \Cref{degreeofapseudoproperhomotopyequivalence}, we didn't know the degree of a pseudo proper homotopy equivalence; even in this stage, we don't know whether a pseudo proper homotopy equivalence has a proper homotopy inverse or not.}
 \end{subsubremark}

 \end{subsection}
 
 \section{Finishing the proofs of \texorpdfstring{\Cref{MR1}}{Theorem \ref{MR1}}, \texorpdfstring{\Cref{MC1.1.}}{Theorem \ref{MC1.1.}}, and \texorpdfstring{\Cref{MR2}}{Theorem \ref{MR2}}}\label{allproofs}

 \begin{proof}[Proof of \cref{MR1}] Consider an \textup{LFCS} $\mathscr C$ on $\Sigma$ provided by \Cref{completedecomposition}. Using \Cref{transtoLFCS}, assume $f$ is smooth as well as $f\ \stackinset{c}{}{c}{0.1ex}{$\top$}{$\abxpitchfork$}\ \mathscr C$. Thus $f^{-1}(\mathscr C)$ is a non-empty LFCS on $\Sigma'$; see \Cref{properplushomotopyequivalnceissurjective} and \Cref{remarktranshomotopy}. By \Cref{annulusremovalfinal2}, $f$ can be properly homotoped to a proper map $g$ such that for each component $\mathcal C$ of $\mathscr C$, $g^{-1}(\mathcal C)$ is a component of $f^{-1}(\mathscr C)$ that is mapped homeomorphically onto $\mathcal C$ by $g$. Thus, $g^{-1}(\mathscr C)$ decomposes $\Sigma'$ into bordered sub-surfaces and each component of $\Sigma\setminus \mathscr C$ has non-empty pre-image; see \Cref{properplushomotopyequivalnceissurjective}. Let $\mathbf S\subset \Sigma$ be a bordered sub-surface obtained as a complementary component of the decomposition of $\Sigma$ by $\mathscr C$. Now,  $\mathbf S\cong g^{-1}(\mathbf S)$; by \Cref{inversepairofpants} (see its proof also) and \Cref{inversepunctureddisk}. Since $g$ sends $\text{int}\big(g^{-1}(\mathbf S)\big)$ onto $\text{int}(\mathbf S)$ and $\partial g^{-1}(\mathbf S)$ homeomorphically onto $\partial \mathbf S$, we can properly homotope $g\vert g^{-1}(\mathbf S)\to \mathbf S$ relative to $\partial g^{-1}(\mathbf S)$ to a homeomorphism $g^{-1}(\mathbf S)\to \mathbf S$; see \Cref{rigidityofpairofpants} and \Cref{Alexandertrick}. Finally, vary $\mathbf S$ over different complementary components of $\Sigma$ decomposed by $\mathscr C$ to collect these boundary-relative proper homotopies and then paste them to get a proper homotopy from $g$ to a homeomorphism $\Sigma'\to \Sigma$. Since $g$ is properly homotopic to $f$, we are done. \label{proofoftheorem1}\end{proof} The proof of \Cref{MR1} shows that we are using the non-zero degree assumption of the pseudo proper homotopy equivalence  (which is gifted by \Cref{properplushomotopyequivalnceisofdegreeone}) to ensure surjectivity after each proper homotopy. Thus, by a similar argument, we can prove the \Cref{rigidityofnonzerodegreemap} below.
\begin{sectheorem}
\textup{Let $f\colon \Sigma'\to \Sigma$ be a pseudo proper homotopy equivalence between two non-compact oriented surfaces. Suppose $\Sigma$ is not homeomorphic to $\Bbb R^2$ and $\deg(f)\neq0$ . Then $\Sigma'$  is homeomorphic to $\Sigma$ and $f$ is properly homotopic to a homeomorphism.} \label{rigidityofnonzerodegreemap}
 \end{sectheorem}
 \begin{sectheorem}
 \textup{Let $f\colon \mathbb R^2\to \mathbb R^2$ be a proper map of degree $\pm 1$. Then $f$ is properly homotopic to a homeomorphism $\mathbb R^2\to \mathbb R^2$.} \label{degreeoneinplane}
 \end{sectheorem}
 \begin{proof}
 By \Cref{reverseprocessindegreefinding}, $f$ can be properly homotoped to get smoothly embedded closed disks $\mathbf D, \mathbf D'\subseteq \mathbb R^2$ such that $\mathbf D'=f^{-1}(\mathbf D)$ and $f\vert\mathbf D'\to \mathbf D$ is a homeomorphism. Using the Jordan-Schönflies theorem, $f\vert\overline{\mathbb R^2\setminus  \mathbf D'}\to \overline{\mathbb R^2\setminus \mathbf D}$ resembles a map between two punctured disks, on which applying \Cref{Alexandertrick}  we conclude.
 \end{proof}

 \begin{proof}[Proof of \cref{MR2}] Since 
$\deg(f)=\pm 1$, by \Cref{degreeonemapsarepi1surjective}, $\pi_1(f)$ is surjective. Thus, $\pi_1(f)$ is bijective. Now, both $\Sigma'$ and $\Sigma$ are homotopy equivalent to $\bigvee_{\mathscr I}\Bbb S^1$ for some index set $\mathscr I$ with $|\mathscr I|\leq \aleph_0$, i.e., $\pi_k(\Sigma')=0=\pi_k(\Sigma)$ for all $k\geq 2$. So, by Whitehead theorem, $f$ is a homotopy equivalence (note that each surface has a CW-complex structure due to its $C^\infty$-smooth structure). Now, a simply-connected non-compact surface is homeomorphic to $\Bbb R^2$; see \Cref{finitetypenoncompactsurace}. So,  combining \Cref{rigidityofnonzerodegreemap} and \Cref{degreeoneinplane}, we are done.
 \end{proof}
 \begin{proof}[Proof of \Cref{MC1.1.}]
 A proper homotopy equivalence is a $\pi_1$-injective map of degree $\pm 1$. Now, apply \Cref{MR2}.
 \end{proof}
 The following proposition is an application of \Cref{MR1}.  
 
 \begin{secproposition} \textup{Let $\Sigma$ be a non-compact surface such that $\Bbb S^1\times \Bbb R\not\cong\Sigma\not \cong \Bbb R^2$. Suppose $f,g\colon (\Sigma,x)\to (\Sigma,y)$ are two pseudo proper homotopy equivalences with $\pi_1(f)=\pi_1(g)\colon \pi_1(\Sigma,x)\to \pi_1(\Sigma,y)$. Then $f$ is properly homotopic to $g$.}  
 \label{pseudoproperhomotopyequivalencesinducingsamemapsonthefundamentalgroupareproperlhomotopic}
 \end{secproposition}
 \begin{proof}
 By applying \Cref{MR1} up to proper homotopy, we may assume both $f$ and $g$ are homeomorphisms without loss of generality. Now, $\Sigma$ is homotopy equivalent to $K\big(\pi_1(\Sigma,y),1\big)$, i.e., $f$ is homotopic to $g$ \big(see \cite[Proposition 1B.9.]{MR1867354}\big). Since $\Bbb S^1\times \Bbb R\not\cong\Sigma\not \cong \Bbb R^2$ and $f^{-1}g$ is homotopic to $\text{Id}_\Sigma$, by \cite[Theorem 6.4.]{MR214087}, there exists a level-preserving homeomorphism $\mathcal H\colon \Sigma\times [0,1]\to \Sigma\times [0,1]$ which agrees with $f^{-1}g$ on $\Sigma\times 0$ and with $\text{Id}_\Sigma$ on $\Sigma\times1$. The projection $\Sigma\times [0,1]\to \Sigma$ is proper implies $f^{-1}g$ is properly homotopic to $\text{Id}_\Sigma$, so we are done. 
 \end{proof}
\begin{section}{Appendix}

\begin{subsection}{Approximation and transversality in the proper category}\label{appendix}

Throughout \Cref{appendix}, $M, N$ will denote two smooth boundaryless manifolds, possibly non-compact. Let $F\colon N\to M$ be a smooth map, and let $X$ be a smoothly embedded boundaryless submanifold of $M$. We say \emph{$F$ is transverse to $X$}, and write $F\ \stackinset{c}{}{c}{0.1ex}{$\top$}{$\abxpitchfork$}\ X$, if for every $p\in F^{-1}(X)$, we have $T_{F(p)}X+ dF_p(T_pN)=T_{F(p)}M$. If $F$ is transverse to $S$, then $F^{-1}(X)$ is a smoothly embedded boundaryless submanifold of $N$ such that $\dim (N)-\dim \left(F^{-1}(X)\right)=\dim(M)-\dim(X)$; see \cite[Theorem 6.30.(a)]{MR2954043}.

The Whitney approximation theorem \cite[Theorem 6.26]{MR2954043} says that any continuous map $N\to M$ is homotopic to a smooth map. The transversality homotopy theorem \cite[Theorem 6.36]{MR2954043} says that for any smooth map $F\colon N\to M$ and for any smoothly embedded boundaryless submanifold $X$ of $M$,  the smooth map $F$ can be homotoped to another smooth map $\widetilde F\colon N\to M$ such that $\widetilde F\ \stackinset{c}{}{c}{0.1ex}{$\top$}{$\abxpitchfork$}\ X$. We modify these two theorems in the proper category. Our interest is in the properness of homotopies; the extra stuff not related to properness is in \cite[Theorems 6.26, 6.36]{MR2954043}.   
\begin{theorem}{\textup{(Proper Whitney approximation theorem)}}
\textup{Let $f\colon N\to M$ be a continuous proper map. Then $f$ is properly homotopic to a smooth proper map.}\label{whiteny}
\end{theorem}
\begin{theorem}{\textup{(Proper transversality homotopy theorem)}}
\textup{Let $f\colon N\to M$ be a smooth proper map, and let $X$ be a smoothly embedded boundaryless submanifold of $M$. Then $f$ is properly homotopic to a smooth proper map  $g\colon N\to M$ which is transverse to $X$.} \label{transhomotopy}
\end{theorem}
We start by summarizing key facts in and around the tubular neighborhood theorem. Let $M\hookrightarrow \mathbb R^\ell$ be a smooth proper embedding; see \cite[Theorems 6.15]{MR2954043}. For each $x\in M$, define the \emph{normal space} $\mathcal N_x M$ to $M$ at $x$ as $\mathcal N_xM\coloneqq \{v\in \Bbb R^\ell: v\perp T_x M\}$. Then $\mathcal NM\coloneqq\big\{(x,v)\in \mathbb R^\ell\times \mathbb R^\ell:x\in M,v\perp T_xM\big\}$ is a smoothly embedded $\ell$-dimensional submanifold of $\Bbb R^\ell\times \Bbb R^\ell$ and  $\pi\colon \mathcal NM\ni (x,v)\longmapsto x\in M$ is vector bundle of rank $\ell-\dim(M)$,  called the \emph{normal bundle} of $M$ in $\Bbb R^\ell$; see \cite[Corollary 10.36]{MR2954043}.

Consider the smooth map $E\colon \mathcal NM\ni (x,v)\longmapsto x+v\in \mathbb R^\ell$. One can show that $dE_{(x,0)}$ is bijective for each $x\in M$. Thus, for each $x\in M$, we have $\delta>0$ such that $E$ maps diffeomorphically $V_\delta(x)\coloneqq\big\{(x',v')\in \mathcal NM: |x-x'|<\delta,|v'|<\delta\big\}$ onto an open neighborhood of $x$ in $\mathbb R^\ell$.
Now, the map $\rho\colon M\to (0,1]$ defined by   $$\rho(x)\coloneqq\sup\left\{\delta\leq 1: E\text{ maps }V_\delta(x)\text{ diffeomorphically }\text{onto an open neighborhood of }x\text{ in }\mathbb R^\ell\right\}$$ is continuous. Further, 
$V\coloneqq\left\{(x,v)\in \mathcal NM:|v|<\frac{1}{2}\rho(x)\right\}$ is an open subset of $NM$ and $E$ maps diffeomorphically $V$ onto an open subset $U$ of $\mathbb R^\ell$ with $M\subseteq U$, i.e., $U$ is a tubular neighborhood of $M$ in $\Bbb R^\ell$; see \cite[Theorem 6.24]{MR2954043}. Note that the map $r\colon U\to M$ defined by $r\coloneqq\pi\circ \left(E\vert V\to U\right)^{-1}$ is a retraction and submersion; see \cite[Proposition 6.25.]{MR2954043}. Denote $\{y\in \Bbb R^\ell: |y-x|<\varepsilon\}$ by $B_\varepsilon(x)$. By an argument similar to showing the continuity of $\rho$, one can prove that $\delta\colon M\to (0,1]$ defined by $\delta(x)\coloneqq\sup\big\{\varepsilon \leq 1:B_\varepsilon(x)\subseteq U\big\}$ for any $x\in M$, is also continuous.

 With this setup, we are now ready to state a crucial lemma, which in particular says that if two points are at the most unit distance, then the distance between their images under the tubular neighborhood retraction can be, at most, $2$.
\begin{lemma}
\textup{Let $\varepsilon>0$. If $y,y'\in U$ with $|y-y'|<\varepsilon$, then  $\big|r(y)-r(y')\big|\leq\varepsilon+1$.}\label{retraction}
\end{lemma}

\begin{proof}
Notice $\big|r(y)-r(y')\big|-|y-y'|\leq \big|y-r(y)\big|+\big|y'-r(y')\big|\leq \frac{1}{2}\rho\circ r(y)+\frac{1}{2}\rho\circ r(y')$ to conclude.
\end{proof} 

Consider another smooth proper embedding $N\hookrightarrow \Bbb R^k$ for the proof of the following three facts. The following lemma says that a homotopy lying in a $\lambda$-neighborhood (where $\lambda$ is a fixed positive number) of a proper map is a proper homotopy.
\begin{lemma}
\textup{Let $h\colon N\to M$ be a continuous proper map, and let $\mathcal H\colon N\times[0,1]\to M$ be a homotopy. If there exists a constant $\lambda$ so that  $\big|\mathcal H(p,t)-h(p)\big|\leq \lambda$ for each $(p,t)\in N\times [0,1]$, then $\mathcal H$ is proper.}\label{distance}
\end{lemma}

\begin{proof}
Note that the embeddings $M\hookrightarrow \mathbb R^\ell$ and $N\hookrightarrow \Bbb R^k$ are closed maps as they are proper maps; see \cite{MR254818}. Consider the induced metric $d_M$ on $M$ inherited from $\Bbb R^\ell$, i.e., $d_M(m,m')=|m-m'|$ for all $m,m'\in  M$. Also, we have the induced metric $d_{N\times [0,1]}$ on $N\times [0,1]$ inherited from $\Bbb R^k\times [0,1]$, i.e., $d_{N\times [0,1]}\big((n,t), (n',t')\big)=|n-n'|+|t-t'|$ for all $(n,t), (n',t')\in N\times [0,1]$. Thus, a subset of $N\times [0,1]$ (respectively, $M$) is compact if and only if it is closed and bounded in $N\times [0,1]$ (respectively, $M$).

Let $C$ be a compact subset of $M$. Continuity of $\mathcal H$ implies $\mathcal H^{-1}(C)$ is closed in $N\times [0,1]$. Also, if there were an unbounded sequence $\big\{(n_i,t_i)\big\}\subseteq  \mathcal H^{-1}(C)$, then $\{n_i\}$, and hence $\big\{h(n_i)\big\}$ would be unbounded (as $h$ is proper); and thus the unbounded set $\big\{h(n_i)\big\}$ would be inside the $\lambda$-neighborhood of the bounded set $C$, a contradiction. Therefore, $\mathcal H^{-1}(C)$ is closed and bounded in $N\times [0,1]$; and hence $\mathcal H^{-1}(C)$ is compact. Since $C$ is an arbitrary compact subset of $M$, we are done.
\end{proof}

Now we are ready to prove the analogs of the Whitney approximation theorem and transversality homotopy theorem in the proper category.

\begin{proof}[Proof of \Cref{whiteny}]
 Whitney Approximation theorem gives a smooth function $\widetilde f\colon N\to \mathbb R^\ell$ such that $\left|\widetilde f(y)-f(y)\right|<\delta\big(  f(y)\big)$ for each $y\in M$; see \cite[Theorem 6.21]{MR2954043}. Now, define $\mathcal H\colon N\times [0,1]\to M$ as $ \mathcal H(p,t)\coloneqq r\left((1-t)f(p)+t\widetilde f(p)\right)$ for all $(p,t)\in N\times [0,1].$ If $(p,t)\in N\times [0,1]$, then $$\left|(1-t)f(p)+t\widetilde f(p)-f(p)\right|\leq t\cdot\left|\widetilde f(p)-f(p)\right|\leq 1.$$ Therefore,   $\big|\mathcal H(p,t)-r\circ f(p)\big|=\big|\mathcal H(p,t)-f(p)\big|\leq 2$ for all $(p,t)\in N\times [0,1]$ by \Cref{retraction}.  Now, \Cref{distance} tells $\mathcal H$ is proper. Therefore, $\mathcal H(-,1)=r\circ \widetilde f$ is a smooth proper map that is properly homotopic to $f$ (recall that $r$ is a smooth retraction). So, we are done.\label{proofwhiteny}
\end{proof}
\begin{proof}[Proof of \Cref{transhomotopy}]
Whitney Approximation theorem gives a smooth function $e\colon N\to (0,\infty)$ with $0<e<\delta\circ f$; see \cite[Corollary 6.22]{MR2954043}. Let $\mathbb B^\ell\coloneqq \{\mathbf s\in \mathbb R^\ell:|\mathbf s|<1\}$. Define $F\colon N\times \mathbb B^\ell\to M$ as $F(p,\mathbf s)\coloneqq r\big(f(p)+e(p)\mathbf s\big)\text{ for any }(p,\mathbf s)\in N\times \Bbb B^\ell.$ If $p\in N$, the restriction of $F$ to $\{p\}\times \Bbb B^\ell$ is the composition of the local diffeomorphism $\mathbf s\longmapsto f(p)+e(p)\mathbf s$ followed by the smooth submersion $r$, so $F$ is a smooth submersion and hence transverse to $X$. 

By parametric transversality theorem \cite[Theorem 6.35]{MR2954043}, $F(-,\mathbf s_0)$ is transverse to $X$ for some $\mathbf s_0\in \mathbb B^\ell$. Now, define $\mathcal H\colon N\times [0,1]\to M$ as $\mathcal H(p,t)\coloneqq r\big(f(p)+te(p) \mathbf s_0\big)\text{ for all }(p,t)\in N\times [0,1].$ If $(p,t)\in N\times [0,1]$, then $$\left|\big(f(p)+te(p)\mathbf s_0\big)-f(p)\right|\leq te(p)\cdot|\mathbf s_0|< \delta\big(f(p)\big)\leq 1.$$ Therefore, $\big|\mathcal H(p,t)-r\circ f(p)\big|=\big| \mathcal H(p,t)-f(p)\big|\leq 2$ for all $(p,t)\in N\times [0,1]$ by \Cref{retraction}. Now, \Cref{distance} tells that $\mathcal H$ is proper. Define $g\coloneqq\mathcal H(-,1)$, i.e., $g=r\big(f(-)+e(-) \mathbf s_0\big)=F(-,\mathbf s_0)$ is properly homotopic to $f$ (recall that $r$ is a smooth retraction) as well as transverse to $X$.
\end{proof}
\end{subsection}

\begin{subsection}{Transversality of a proper map between two surfaces with respect to a circle}\label{Transversalitynearcircle}
Here are a couple of notations that will be used throughout \Cref{Transversalitynearcircle}. Let  $f\colon \Sigma'\to \Sigma$ be a smooth \emph{proper} map between two surfaces, and let $\mathcal C$ be a smoothly embedded circle on $\Sigma$ such that $f$ is transverse to $\mathcal C$. Also, let $\varphi\colon  \mathcal C\times [-1,1]\hookrightarrow \Sigma$ be a smooth embedding with $\varphi(\mathcal C,0)=\mathcal C$, i.e., $\text{im}(\varphi)$ is a \emph{two-sided (trivial) tubular neighborhood} of $\mathcal C$.  We call each of $\varphi\big(\mathcal C\times [-1,0]\big)$ and $\varphi\big(\mathcal C\times [0,1]\big)$ a \emph{one-sided tubular neighborhood of $\mathcal C$} (in short, \emph{a side of $\mathcal C$}). By scaling, we may replace $[-1,0]$ and $[0,1]$ with other closed intervals.

The following theorem says that $f$ is transverse to all circles, which are parallel to $\mathcal C$ and sufficiently near to $\mathcal C$.
\begin{theorem}
 \textup{There exists $\varepsilon_0\in (0,1)$ such that $f$ is transverse to $\mathcal C_\varepsilon\coloneqq\varphi(\mathcal C,\varepsilon)$ for each $\varepsilon\in [-\varepsilon_0,\varepsilon_0]$. Thus for any $\varepsilon\in [-\varepsilon_0,\varepsilon_0]$, $f^{-1}(\mathcal C_\varepsilon)$ is either empty or a pairwise disjoint collection of finitely many smoothly embedded circles on $\Sigma'$.} \label{trans1}
\end{theorem}

At first, we need a lemma to prove \Cref{trans1}.
\begin{lemma}
\textup{Let $g\colon \Bbb R^2\to \Bbb R^2$ be a smooth map and $x_n\to x$ in $\Bbb R^2$ with $r_n\coloneqq|g(x_n)|\to 1$. Write $S_r\coloneqq\{z\in \Bbb R^2:|z|=r\}$ and assume $\textup{im}(dg_{x_n})=T_{g(x_n)}(S_{r_n})$ for all $n$. If $dg_x\neq 0$, then $\textup{im}(dg_x)=T_{g(x)}(S_1)$.} \label{lemmafortrans1}
\end{lemma}
\begin{proof}
The derivative map $dg\colon \Bbb R^2\to \textup{L}(\Bbb R^2,\Bbb R^2)$ is continuous implies $dg_{x_n}\to dg_x$, and this convergence can be thought as convergence of $2\times 2$-matrices. In particular, if $\mathbf i,\mathbf j\in \Bbb R^2$ are two perpendicular unit vectors, then $dg_{x_n}(\mathbf i)\to dg_x(\mathbf i)$ and $dg_{x_n}(\mathbf j)\to dg_x(\mathbf j)$.

Recall that the tangent space at any point of a circle is the vector space of all points perpendicular to this point. So, $\big\langle dg_{x_n}(\mathbf i), g(x_n)\big\rangle=0=\big\langle dg_{x_n}(\mathbf j), g(x_n)\big\rangle$ by hypothesis, and now $\big\langle dg_{x}(\mathbf i), g(x)\big\rangle=0=\big\langle dg_{x}(\mathbf j), g(x)\big\rangle$ by the convergence of inner-product. Hence, $\text{im}(dg_x)\subseteq T_{g(x)}(S_1)$. Since $dg_x\neq 0$ and $\dim T_{g(x)}(S_1)=1$, we are done.
\end{proof}

\begin{proof}[Proof of \Cref{trans1}]
Suppose not. So, a sequence $\varepsilon_n\to 0$ and points $x_n\in f^{-1}(\mathcal C_{\varepsilon_n})$ exist such that $\text{im}(df_{x_n})+ T_{f(x_n)}\mathcal C_{\varepsilon_n}\subsetneqq T_{f(x_n)}\Sigma$ for all $n$. Hence,  $\text{im}(df_{x_n})\subseteq T_{f(x_n)}\mathcal C_{\varepsilon_n}$ as $T_{f(x_n)}\mathcal C_{\varepsilon_n}\oplus \mathcal N_{f(x_n)}\mathcal C_{\varepsilon_n}=T_{f(x_n)}\Sigma$ for all $n$. Now, $\{x_n\}$ is contained in the compact set $f^{-1}\big(\text{im}(\varphi)\big)$ (recall that $f$ is a proper map), i.e., passing to sub-sequence, if needed, assume $x_n\to x\in f^{-1}(\mathcal C)$. 

The continuity of the derivative map says $df_{x_n}\to df_x$. After discarding first few terms, we may assume $df_{x_n}\neq 0$ for all $n$ (otherwise, we would have $df_x=0$, i.e., $T_{f(x)}\mathcal C+\text{im}(df_x)=T_{f(x)}\mathcal C$ wouldn't be equal to $T_{f(x)}\Sigma$, i.e.,  $f$ wouldn't be transverse  to $\mathcal C$). So, $\text{im}(df_{x_n})=T_{f(x_n)}(\mathcal C_{\varepsilon_n})$ for all $n$ (a non-zero vector subspace of a one-dimensional vector space is equal to the whole space). 

Now, restricting $f$ to a coordinate ball containing $x$ and then post composing with $\varphi^{-1}$, we can consider \Cref{lemmafortrans1} above, which gives $\text{im}(df_{x})=T_{f(x)}(\mathcal C)$, a contradiction to the assumption  $f\ \stackinset{c}{}{c}{0.1ex}{$\top$}{$\abxpitchfork$}\ \mathcal C$. 
\end{proof}

The previous theorem guarantees transversality near $\mathcal C$. In the rest part of \Cref{Transversalitynearcircle}, we aim to prove that every small one-sided tubular neighborhood of a component of $f^{-1}(\mathcal C)$ maps into a small one-sided tubular neighborhood of $\mathcal C$.

At first, we fix some notations. So, let $\mathcal C'$ be a component of $f^{-1}(\mathcal C)$. Also, consider an $\varepsilon_0\in(0,1)$  such that $f\ \stackinset{c}{}{c}{0.1ex}{$\top$}{$\abxpitchfork$}\ \mathcal C_\varepsilon$ for every $\varepsilon\in [-\varepsilon_0,\varepsilon_0]$; see \Cref{trans1}.
\begin{theorem}
    \textup{Let $\varepsilon\in(0,\varepsilon_0]$, and let $\mathcal T'$ be a two-sided compact tubular neighborhood of $\mathcal C'$ in $\Sigma'$. Then there exist two one-sided compact tubular neighborhoods $\mathcal U_l', \mathcal U_r'$ of $\mathcal C'$ in $\Sigma'$ such that $\mathcal U_l'\cup \mathcal U_r'$ is a two-sided tubular neighborhood of $\mathcal C'$ with $\mathcal U_l'\cup \mathcal U_r'\subseteq \mathcal T'$, and for each $s\in\{l,r\}$ the following hold: $(i)$ $f^{-1}(\mathcal C)\cap \mathcal U'_s=\mathcal C'$, $(ii)$ either $f(\mathcal U'_s)\subseteq \varphi\big(\mathcal C\times [0,\varepsilon]\big)$ or $f(\mathcal U'_s)\subseteq \varphi\big(\mathcal C\times [-\varepsilon,0]\big)$. }\label{technicallemma}
\end{theorem}

\begin{proof}
    By \Cref{trans1}, $f^{-1}(\mathcal C_{-\varepsilon})\cup f^{-1}(\mathcal C)\cup f^{-1}(\mathcal C_{\varepsilon})$ is a pairwise disjoint collection of finitely many smoothly embedded circles on $\Sigma'$. Now, consider two one-sided compact tubular neighborhoods $\mathcal U'_l, \mathcal U'_r$ of $\mathcal C'$ in $\Sigma'$ such that such that $\mathcal U_l\cup \mathcal U_r$ is a two-sided tubular neighborhood of $\mathcal C'$ with $\mathcal U_l'\cup \mathcal U_r'\subseteq \mathcal T'$, and for each $s\in\{l,r\}$ the following hold: $(i)$ $f^{-1}(\mathcal C)\cap \mathcal U'_s=\mathcal C'$, $(ii)$ $\mathcal U'_s\cap f^{-1}(\mathcal C_{\varepsilon})=\varnothing=\mathcal U'_s\cap f^{-1}(\mathcal C_{-\varepsilon})$.

    Now, fix $s\in\{l,r\}$. Since $\mathcal U'_s\setminus \mathcal C'$ is connected and $f$ is continuous, $f(\mathcal U'_s\setminus \mathcal C')$ is contained in one of the components of $\Sigma\setminus (\mathcal C_{-\varepsilon}\cup \mathcal C\cup \mathcal C_{\varepsilon})$. But $f(\mathcal C')\subseteq \mathcal C$ implies either $f(\mathcal U'_s)\subseteq \varphi\big(\mathcal C\times [0,\varepsilon]\big)$ or $f(\mathcal U'_s)\subseteq \varphi\big(\mathcal C\times [-\varepsilon,0]\big)$. So, we are done.
\end{proof}

\begin{remark}
    \textup{In \Cref{technicallemma}, it is possible that $f(\mathcal U'_l\cup \mathcal U_r')$ is contained in either $\varphi\big(\mathcal C\times [0,\varepsilon]\big)$ or $\varphi\big(\mathcal C\times [-\varepsilon, 0]\big)$, i.e., $f$ may map both sides of $\mathcal C'$ in one of the two sides of $\mathcal C$.}
\end{remark}
Consider the one-sided compact tubular neighborhoods $\mathcal U_l,\mathcal U_r'$ of $\mathcal C'$ in $\Sigma'$ given by \Cref{technicallemma}. Notice that for some $s\in \{l,r\}$, it is possible that $f\big((\partial \mathcal U_s')\setminus \mathcal C'\big)\not\subseteq \varphi(\mathcal C\times t)$ for any $t\in [-\varepsilon, \varepsilon]$. A remedy for this is given in the following theorem.

\begin{theorem}
\textup{Let $\varepsilon\in (0,\varepsilon_0]$, and let $\mathcal U'$ be a one-sided compact tubular neighborhood of $\mathcal C'$ such that $f^{-1}(\mathcal C)\cap \mathcal U'=\mathcal C'$, and $f(\mathcal U')\subseteq \varphi\big(\mathcal C\times [0,\varepsilon]\big)$. Then there is a  $\delta\in (0,\varepsilon)$ and there is a component $\mathcal C_\delta'$ of $f^{-1}(\mathcal C_\delta)$ such that the following hold:}
\begin{itemize}
    \item[$(1)$] \textup{$\mathcal C_\delta'$ together with $\mathcal C'$ co-bounds an annulus $\mathcal A'\subseteq \mathcal U'$  so that any other component of $f^{-1}(\mathcal C_\delta)$ in $\textup{int}(\mathcal A')$, if any, bounds a disk inside $\mathcal A'$.}
    \item[$(2)$] \textup{The map $f$ sends $\mathcal A'$ into $\varphi\big(\mathcal C\times [0,\varepsilon]\big)$. Also, after  removing the interiors of all disks bounded by components of $f^{-1}(\mathcal C_\delta)$ from $\mathcal A'$, we can send it to $\varphi\big(\mathcal C\times [0,\delta]\big)$ by $f$.}
\end{itemize}
\label{thick}
\end{theorem}
\begin{proof}[Proof of part (1) of \Cref{thick}]Choose a $\delta\in (0,\varepsilon)$ such that $\varphi\big(\mathcal C\times [0,\delta]\big)\cap f\big((\partial \mathcal U')\setminus\mathcal C'\big)=\varnothing$. Note that such a $\delta$ exists; otherwise, using the compactness of $(\partial \mathcal U')\setminus\mathcal C'$, we would have a sequence $\{x_n'\}\subseteq (\partial \mathcal U')\setminus\mathcal C'$ converging to some $x'\in (\partial \mathcal U')\setminus\mathcal C'$ such that $f(x_n')\in \varphi\left(\mathcal C\times [0,1/n]\right)$, i.e., $f(x')$ would belong to $\mathcal C$, a contradiction to the assumption $f^{-1}(\mathcal C)\cap \mathcal U'=\mathcal C'$. Define an open set $\mathcal W'$ as follows $$
    \mathcal W'\coloneqq\text{int}(\mathcal U')\cap f^{-1}\big(\varphi\big(\mathcal C\times (0,\delta)\big)\big). $$  Notice that no sequence in $\mathcal W'$ converges to some point of $(\partial \mathcal U')\setminus \mathcal C'$. Otherwise, if we assume $\mathcal W_n'\ni w_n'\longrightarrow x'\in (\partial \mathcal U')\setminus \mathcal C'$, then $\varphi\big(\mathcal C\times (0,\delta)\big)\ni f(w_n')\longrightarrow f(x')$. Since $\varphi\big(\mathcal C\times [0,\delta]\big)$ is a closed set containing the sequence $\big\{f(w_n')\big\}$, we can say that  $f(x')\in f\big((\partial \mathcal U')\setminus\mathcal C'\big)\cap \varphi\big(\mathcal C\times [0,\delta]\big)$, which is impossible by our choice of $\delta$.

Therefore,  $\overline{\mathcal  W'}\subseteq \mathcal U'$ (as $\mathcal U'$ is compact) but $\big((\partial \mathcal U')\setminus \mathcal C'\big)\cap \overline{\mathcal  W'}=\varnothing$. In particular, $\partial \mathcal  W'\subseteq \mathcal U'$ but $\big((\partial \mathcal U')\setminus \mathcal C'\big)\cap \partial \mathcal W'=\varnothing$. 
\begin{claim}
    \textup{$\partial\mathcal W' \subseteq \mathcal C'\cup f^{-1}(\mathcal C_\delta)$. Thus $\partial\mathcal W'$ is contained in a finite union of pairwise disjoint circles. } \label{claim0}
\end{claim}
\begin{proof}[Proof of \Cref{claim0}]
    Let $y\in \partial \mathcal W'$ and consider a sequence $\{y_n'\}\subseteq \mathcal W'$ converging to $y'$. Then $\varphi\big(\mathcal C\times (0,\delta)\big)\ni f(y_n')\longrightarrow f(y')\in\varphi\big(\mathcal C\times [0,\delta]\big)$. If $f(y')\in \varphi\big(\mathcal C\times \{0,\delta\}\big)=\mathcal C\cup \mathcal C_\delta$, then we are done since $f^{-1}(\mathcal C)\cap \mathcal U'=\mathcal C'$. On the other hand, if $f(y')\in \varphi\big(\mathcal C\times (0,\delta)\big)$, then the definition of $\mathcal W'$ and $\mathcal W'\cap \partial \mathcal W'=\varnothing$ (as $\mathcal W'$ is open) together imply $y'\in \mathcal U'\setminus\text{int}(\mathcal U')=\partial \mathcal U'$, i.e., $y'\in \mathcal C'$ as $\big((\partial \mathcal U')\setminus \mathcal C'\big)\cap \partial \mathcal W'=\varnothing$. Since $y\in \partial \mathcal W'$ is arbitrary, we are done.
\end{proof}

The definition of $\mathcal W'$ tells that each point of $\text{int}(\mathcal U')$ that is sufficiently near to $\mathcal C'$ must belong to $\mathcal W'$. Now, using \Cref{claim0}, we can say that there is at least one component of $f^{-1}(\mathcal C_\delta)$, which co-bounds an annulus with $\mathcal C'$ inside $\mathcal U'$. Of all the $\mathcal C'$-parallel components of $f^{- 1}(\mathcal C_\delta)$, we consider the closest to $\mathcal C'$ as $\mathcal C_\delta'$. \end{proof}
\begin{proof}[Proof of part (2) of \Cref{thick}]
Certainly, $f(\mathcal A')\subseteq f(\mathcal U')\subseteq\varphi\big(\mathcal C\times [0,\varepsilon]\big)$. Now, the rest follows, once we observe that removing the interiors of all disks  bounded by components of $f^{-1}(\mathcal C_\delta)$ from $\mathcal A'$, $\mathcal A'$ remains connected, so by continuity of $f\vert\Sigma'\setminus f^{-1}(\mathcal C\cup\mathcal C_\delta)\to \Sigma\setminus(\mathcal C\cup \mathcal C_\delta)$, it maps into $\varphi\big(\mathcal C\times (0,\delta)\big)$. 
\end{proof}
\end{subsection}
\subsection*{Acknowledgments}
I want to thank my thesis adviser, Siddhartha Gadgil, for suggesting the proper rigidity problem, lots of helpful discussions, and considerable improvements to the first draft of this paper. I also want to thank Ajay Kumar Nair for several fruitful discussions. I am grateful to the anonymous referee for his careful reading of the paper and his comments and detailed suggestions, which helped considerably in improving the manuscript. This work is supported by a scholarship from the National Board for Higher Mathematics (Ref. No. 0203/14/2019/R\&D-II/7481).

\bibliographystyle{plain}
\bibliography{references.bib}

\newcommand{\Addresses}{{
  \bigskip
  \footnotesize

  \textsc{Department of Mathematics, Indian Institute of Science,
   Bangalore 560012, India}\par\nopagebreak
  \textit{E-mail address}: \texttt{\href{mailto:sumantadas@iisc.ac.in}{\color{black}sumantadas@iisc.ac.in}}

}}
\Addresses

\end{section}
\end{document}